\documentclass[a4paper,leqno,12pt]{amsart}
\usepackage[hmargin=1in,textheight=23cm,centering]{geometry}

\usepackage{amsmath}
\usepackage{amssymb}
\usepackage{mathtools}
\usepackage{graphicx}
\usepackage{enumitem}
\usepackage{tikz-cd}
\usepackage[hidelinks]{hyperref}
\hypersetup{
  pdftitle={Cartan--Ambrose--Hicks theorems and dual geodesic symmetries on statistical manifolds},
  pdfauthor={Naoto Satoh}
}

\theoremstyle{plain}
\newtheorem{theorem}{\indent\normalfont\scshape Theorem}[section]
\newtheorem{lemma}[theorem]{\indent\normalfont\scshape Lemma}
\newtheorem{corollary}[theorem]{\indent\normalfont\scshape Corollary}
\newtheorem{proposition}[theorem]{\indent\normalfont\scshape Proposition}

\theoremstyle{definition}
\newtheorem{definition}[theorem]{\indent\normalfont\scshape Definition}
\newtheorem{setting}[theorem]{\indent\normalfont\scshape Setting}
\newtheorem{remark}[theorem]{\indent\normalfont\scshape Remark}

\numberwithin{equation}{section}

\newcommand{\id}{\operatorname{id}}
\newcommand{\Iso}{\operatorname{Iso}}
\let\originaloverline\overline
\newcommand{\overlinescale}{0.85}
\newcommand{\shortoverlineaux}[2]{%
  \begingroup
  \setbox0=\hbox{$#1#2$}%
  \ooalign{%
    \hfil\scalebox{\overlinescale}[1]{$#1\originaloverline{\phantom{#2}}$}\hfil\cr
    \box0\cr
  }%
  \endgroup
}
\renewcommand{\overline}[1]{%
  \mathord{\mathpalette{\shortoverlineaux}{#1}}%
}

\allowdisplaybreaks

\begin{document}

\title[CAH theorems and dual geodesic symmetries]{Cartan--Ambrose--Hicks theorems\\ and dual geodesic symmetries\\ on statistical manifolds}

\author[N.~Satoh]{Naoto Satoh}
\address{Mathematics Laboratory \endgraf
Sapporo Campus, Hokkaido University of Education \endgraf
1-5, Ainosato 5-3, Kita-ku, Sapporo 002-8502 \endgraf
Japan}
\email{satoh.naoto@s.hokkyodai.ac.jp}

\dedicatory{Dedicated to Professor Jun-ichi Inoguchi on his sixtieth birthday}

\date{October 27, 2026}

\begin{abstract}
We establish Cartan–Ambrose–Hicks (CAH) theorems for statistical manifolds and characterize statistical symmetric spaces by dual geodesic symmetries.
These theorems determine statistical structures up to local statisticmorphism by comparing the curvature of an alpha-connection and the skewness tensor using radial geodesics and parallel displacement of that connection.
The Levi--Civita formulation characterizes local statisticmorphisms by comparing the radial Jacobi operators of Riemannian curvature and the skewness tensor using Levi--Civita parallel displacement.
The same characterization holds with statistical curvature in place of Riemannian curvature.
For Hessian manifolds, comparison of the skewness tensor alone is necessary and sufficient.

For connected statistical manifolds, we prove the equivalence between the homogeneous definition of statistical symmetric spaces and the existence of global dual geodesic symmetries at every point.
The CAH theorems give the necessary and sufficient conditions for maps constructed from exponential maps of dual connections to be local dual geodesic symmetries, preserving the metric and interchanging the dual connections.
Such local symmetries exist at every point if and only if the Riemannian curvature and skewness tensors are Levi–Civita parallel.
For Hessian manifolds, this condition is equivalent to the equality of the Hessian and dual Hessian curvature tensors.
\end{abstract}

\maketitle

%%%%%%%%%%%%%%%%%%%%%%%%%%%%
%%%%%%%%%%%%%%%%%%%%%%%%%%%%
%%%%%%%%%%%%%%%%%%%%%%%%%%%%
\section*{Introduction}
%%%%%%%%%%%%%%%%%%%%%%%%%%%%
%%%%%%%%%%%%%%%%%%%%%%%%%%%%
%%%%%%%%%%%%%%%%%%%%%%%%%%%%

The Cartan--Ambrose--Hicks (CAH) theorems determine Riemannian metrics and affine connections by tensor comparison under parallel displacement.
They give necessary and sufficient conditions for a prescribed tangent-space isometry to extend to a Riemannian isometry and, in the affine case, for a linear isomorphism to extend to a connection-preserving diffeomorphism.
Cartan's local theorem compares Riemannian curvature along radial Levi--Civita geodesics \cite{Cartan1928}.
Ambrose extended the result to complete simply connected Riemannian manifolds by using broken geodesics \cite{Ambrose}; Hicks extended it to affine connections, comparing curvature and torsion with respect to the connection \cite{Hicks1959}.

A statistical manifold carries a Riemannian metric $g$ and a torsion-free connection $\nabla$ for which the cubic form $C=\nabla g$ is totally symmetric.
The pair $(g,\nabla)$ is called a statistical structure, and $\nabla$ is its statistical connection.
A local statisticmorphism is a local diffeomorphism preserving both the metric and the statistical connection.
For statistical manifolds, the question is which tensor comparisons characterize local statisticmorphisms, and which connections can be used for the geodesics and parallel displacements.
We use the skewness tensor $K$ and the alpha-connections
\[
C(X,Y,Z)=g(K(X,Y),Z),
\qquad
\nabla^{(\alpha)}=\nabla^{g}-\frac{\alpha}{2}K,
\]
where $\nabla^{g}$ is the Levi--Civita connection of $g$, $\nabla^{(1)}=\nabla$, and $\nabla^{(-1)}=\nabla^{*}$ is the dual connection of $\nabla$ with respect to $g$.
We denote the curvature tensors of $\nabla^{g}$ and $\nabla^{(\alpha)}$ by $R^{g}$ and $R^{(\alpha)}$, respectively, and write $R=R^{(1)}$ and $R^{*}=R^{(-1)}$.

The local CAH theorem in \cite[Theorem~3.2]{Satoh2019} characterizes local statisticmorphisms by comparing the full Riemannian curvature tensor $R^{g}$ and $K$ along Levi--Civita geodesics using Levi--Civita parallel displacement.
The same work gives a sufficient criterion using $R$, $K$, and $\nabla K$ \cite[Corollary~3.4]{Satoh2019}.
This raises two questions: whether geodesics and parallel displacement relative to the statistical connection give such a characterization, and whether statistical curvature $R^{\mathsf{S}}=(R+R^{*})/2$ can replace Riemannian curvature.
The present paper answers both questions and characterizes statistical symmetric spaces by dual geodesic symmetries.

Given statistical manifolds $(M,g,\nabla)$ and $(\overline{M},\overline{g},\overline{\nabla})$ of the same dimension and an initial linear isometry $\Phi:T_{p}M\longrightarrow T_{\overline{p}}\overline{M}$, fix $\alpha,\beta\ne0$.
On sufficiently small normal neighborhoods, define
\[
\varphi^{(\alpha,\beta)}
=\overline{\operatorname{Exp}}_{\overline{p}}^{(\beta)}
\circ\Phi\circ
\bigl(\operatorname{Exp}_{p}^{(\alpha)}\bigr)^{-1}.
\]
Following Hicks \cite{Hicks1966}, we call this map a spray map.
Parallel displacement with respect to $\nabla^{(\alpha)}$ and $\overline{\nabla}^{(\beta)}$ determines linear isomorphisms from $\Phi$ along corresponding radial geodesics.
Theorem~\ref{thm:alpha-beta-CAH} states that $\varphi^{(\alpha,\beta)}$ preserves the metrics and satisfies $(\varphi^{(\alpha,\beta)})_{*}\nabla^{(\alpha)}=\overline{\nabla}^{(\beta)}$ if and only if, along all corresponding radial geodesics, these linear isomorphisms intertwine the Jacobi operators of $R^{(\alpha)}$ and $\overline{R}^{(\beta)}$ in the radial directions and identify the scaled skewness tensors $\alpha K$ and $\beta\overline{K}$.

Although these parallel displacements need not preserve the metrics, $\nabla^{(\alpha)}g=\alpha C$ and $\overline{\nabla}^{(\beta)}\overline{g}=\beta \overline{C}$ give linear ordinary differential equations for the Gram matrices in the corresponding parallel frames.
The comparison of $\alpha K$ and $\beta\overline{K}$, together with the initial isometry $\Phi$, shows that the two Gram matrices satisfy the same linear initial-value problem.
Therefore, uniqueness implies that the Gram matrices coincide and that the linear isomorphisms obtained by parallel displacement are isometries.
The condition on the Jacobi operators identifies these linear isomorphisms with the differential of $\varphi^{(\alpha,\beta)}$.
For $\alpha\ne0$, the cases $\beta=\alpha$ and $\beta=-\alpha$ give local statisticmorphisms to the target structure and its dual, respectively.

The Levi--Civita formulation in Theorem~\ref{thm:LC-CAH} replaces comparison of the full Riemannian curvature tensor by comparison of the radial Jacobi operators of $R^{g}$, while retaining comparison of $K$.
Theorem~\ref{thm:LC-CAH-curv} gives a necessary and sufficient criterion using the radial Jacobi operators of the statistical curvature $R^{\mathsf{S}}=(R+R^{*})/2$ and $K$.
The comparisons use Levi--Civita parallel displacement along all corresponding radial Levi--Civita geodesics.
Equivalently, along each such geodesic $\gamma$, the criterion in Theorem~\ref{thm:LC-CAH-curv} compares $K$ and $R(u,\dot{\gamma})w$ for arbitrary tangent vectors $u,w$, without requiring a comparison of $\nabla K$.
In dimensions at least two, these criteria also admit a formulation using statistical sectional curvature on radial two-planes.

For Hessian manifolds, $R=R^{*}=0$.
When $\alpha^{2}=\beta^{2}\ne0$, Corollary~\ref{cor:Hess-alpha-CAH} reduces the CAH criterion to radial comparison of $\alpha K$ and $\beta\overline{K}$.
In the Levi--Civita formulation, Corollary~\ref{cor:LC-CAH-conj-Hess} shows that radial comparison of $K$ alone is necessary and sufficient.
Thus the comparison of Hessian curvature tensors in \cite[Theorem~4.1]{Satoh2019} can be omitted.

Combining the case $\beta=\alpha\ne0$ with the Levi--Civita formulation gives necessary and sufficient conditions for an initial tangent-space isometry to extend to a local statisticmorphism.
For each fixed $\alpha\in\mathbb{R}$, these conditions compare the radial Jacobi operators of $R^{(\alpha)}$ and the skewness tensor $K$ along radial $\nabla^{(\alpha)}$-geodesics using the parallel displacement of that connection.
Thus, these comparisons determine the statistical structure up to local statisticmorphism.

For Riemannian manifolds satisfying $\nabla^{g}R^{g}=0$, Levi--Civita parallel displacement preserves Riemannian curvature.
Hence the curvature condition in the Riemannian CAH theorem need only be checked at the base points.
In particular, for the same manifold and base point $p$, the initial isometry $-\id_{T_{p}M}$ preserves $R^{g}_{p}$.
Therefore, the CAH theorem shows that the corresponding spray map, which is the Riemannian geodesic symmetry, is a local isometry.
Cartan \cite{Cartan1926} established the theory of Riemannian symmetric spaces.
He proved that Riemannian curvature is Levi--Civita parallel if and only if the Riemannian geodesic symmetry $s_{p}^{(0)}(\operatorname{Exp}_{p}^{(0)}(v)) =\operatorname{Exp}_{p}^{(0)}(-v)$ is a local isometry at every point.

For statistical manifolds, the corresponding conditions concern both $R^{g}$ and $K$.
When both are Levi--Civita parallel, \cite[Corollary~3.5]{Satoh2019} reduces comparison to base-point data.
The map $-\id_{T_{p}M}$ preserves $R^{g}_{p}$ and reverses $K_{p}$.
The dual statistical structure has skewness tensor $-K$.
Therefore, comparison with the dual structure constructs a local dual geodesic symmetry.

A local dual geodesic symmetry at $p$ is a local isometry $s_{p}$ with $s_{p}(p)=p$, $(ds_{p})_{p}=-\id_{T_{p}M}$, and $(s_{p})_{*}\nabla=\nabla^{*}$.
For $\alpha\ne0$, Corollary~\ref{cor:alpha-dual-CAH} characterizes when the spray map
\[
s_{p}^{(\alpha)}
=\operatorname{Exp}_{p}^{(-\alpha)}
\circ(-\id_{T_{p}M})\circ
\bigl(\operatorname{Exp}_{p}^{(\alpha)}\bigr)^{-1}
\]
is such a symmetry for $(g,\nabla^{(\alpha)})$.
Then it agrees locally with the Riemannian geodesic symmetry.
Thus the statistical CAH theorem constructs the symmetry from the exponential maps of the dual connections.

Nielsen and Okamura \cite{NielsenOkamura} use dual geodesic symmetries to study geodesic midpoints associated with dual connections.
On Riemannian symmetric spaces with Levi--Civita parallel cubic forms $\nabla^{g}C=0$, they prove that geodesic symmetries interchange the dual connections.
Moreover, Kobayashi and Okuda \cite{KobayashiOkuda} treat a $G$-invariant statistical structure on $G/H$ and prove that the cubic form is Levi--Civita parallel when $(\mathfrak{g},\mathfrak{h})$ is a symmetric Lie algebra pair.

In this article, we extend Cartan's characterization of Riemannian local symmetry to statistical manifolds.
Theorem~\ref{thm:local-stat-sym} proves that local dual geodesic symmetries exist at every point if and only if
\begin{equation}
\nabla^{g}R^{g}=0,
\qquad
\nabla^{g}K=0.
\label{eq:intro-parallel-R-K}
\end{equation}
The condition $\nabla^{g}K=0$ is equivalent to $\nabla^{g}C=0$.
These conditions characterize locally statistical symmetric spaces.
Each dual geodesic symmetry exchanges $\nabla^{(\alpha)}$ and $\nabla^{(-\alpha)}$ for every $\alpha\in\mathbb{R}$.

Inoguchi and Ohno define statistical symmetric spaces by a homogeneous presentation $M\simeq G/H$, with $G$ acting effectively by statistical automorphisms and an $\operatorname{Ad}(H)$-invariant reductive decomposition $\mathfrak{g}=\mathfrak{h}\oplus\mathfrak{m}$ satisfying $[\mathfrak{m},\mathfrak{m}]\subset\mathfrak{h}$.
They ask for an interpretation of this definition in terms of curvature and geodesics \cite[Definition~10.3 and Problem~10.5]{InoguchiOhno}.
In Theorem~\ref{thm:homo-stat-sym}, we prove that, on a connected statistical manifold, this definition is equivalent to the existence of global dual geodesic symmetries at every point.
The Riemannian transvection group gives the effective symmetric homogeneous presentation and preserves the statistical structure.
This equivalence requires no simple connectivity.

On a simply connected statistical manifold complete with respect to $g$ and satisfying the conditions \eqref{eq:intro-parallel-R-K}, Theorem~\ref{thm:global-stat-sym} extends the local dual geodesic symmetries obtained in Theorem~\ref{thm:local-stat-sym} uniquely to global dual geodesic symmetries.

For the symmetry problem on Hessian manifolds, the relation $R^{g}=-\frac{1}{4}[K,K]$ shows that $\nabla^{g}K=0$ also implies $\nabla^{g}R^{g}=0$.
Let $Q=\frac{1}{2}\nabla K$ and $Q^{*}=\frac{1}{2}\nabla^{*}K^{*}$ denote the Hessian curvature tensor and its dual \cite{Shima1995}.
The identity $Q-Q^{*}=\nabla^{g}K$ and Corollary~\ref{cor:Hess-stat-sym} show that a Hessian manifold is a locally statistical symmetric space if and only if $Q=Q^{*}$.
Conjugate symmetry is automatic here, whereas equality of the Hessian curvature tensors is an additional condition.

On the cone of symmetric positive-definite matrices, the spray map $s_{P}^{(\alpha)}$ agrees with the global Riemannian geodesic symmetry at $P$, $\Sigma\mapsto P\Sigma^{-1}P$, for every $\alpha\in\mathbb{R}$.
On the Euclidean motion group of the plane $\mathrm{E}(2)$, the statistical structure from \cite{InoguchiOhno} has a symmetric homogeneous presentation by $\mathrm{SO}(2)\times\mathbb{R}^{2}$.
The simply transitive left action of $\mathrm{E}(2)$ preserves the structure but does not satisfy the bracket condition.
Thus a transitive group of statistical automorphisms need not give a symmetric homogeneous presentation.
This example also admits global dual geodesic symmetries without simple connectivity.

With its Fisher metric and dually flat connections, the family of normal distributions is statistically homogeneous and its underlying Riemannian structure is symmetric.
However, $\nabla^{g}K\ne0$ at every point.
Therefore, for every $\alpha\ne0$, the structure $(g,\nabla^{(\alpha)})$ admits no local dual geodesic symmetry at any point.
Its Riemannian geodesic symmetry preserves the metric but does not exchange the dual connections.
The spray map for $\alpha=1$ sends $\nabla$ to $\nabla^{*}$ locally, but it neither sends $\nabla^{*}$ to $\nabla$ nor preserves the metric.

Section~\ref{sec:stat-curv} introduces statistical structures, alpha-connections, and curvature tensors on statistical manifolds.
Section~\ref{sec:stat-morph} gives equivalent conditions for statisticmorphisms and characterizes infinitesimal statistical automorphisms by Lie derivatives.
Section~\ref{sec:alpha-CAH} proves the local statistical CAH theorem using radial geodesics and parallel displacement of alpha-connections.
Section~\ref{sec:LC-CAH} gives the Levi--Civita formulation, including the criterion using statistical curvature and the Hessian specialization.
Section~\ref{sec:local-stat-sym} characterizes locally statistical symmetric spaces by the conditions $\nabla^{g}R^{g}=0$ and $\nabla^{g}K=0$.
Section~\ref{sec:transv-stat-sym} develops statistical transvections and treats the global extension of local dual geodesic symmetries.
It proves the equivalence between the existence of global dual geodesic symmetries at every point and the homogeneous definition of statistical symmetric spaces, and gives the criterion $Q=Q^{*}$ for local statistical symmetry of Hessian manifolds.
Section~\ref{sec:examples} treats the cone of symmetric positive-definite matrices, the Euclidean motion group $\mathrm{E}(2)$, and the family of normal distributions.

%%%%%%%%%%%%%%%%%%%%%%%%%%%%
%%%%%%%%%%%%%%%%%%%%%%%%%%%%
%%%%%%%%%%%%%%%%%%%%%%%%%%%%
\section{Statistical structures and curvature tensors}
\label{sec:stat-curv}
%%%%%%%%%%%%%%%%%%%%%%%%%%%%
%%%%%%%%%%%%%%%%%%%%%%%%%%%%
%%%%%%%%%%%%%%%%%%%%%%%%%%%%

Throughout this paper, $M$ denotes a smooth manifold, and all objects under consideration are assumed to be smooth.
For a vector bundle $E\to M$, we denote by $\Gamma(E)$ the space of smooth sections of $E$.
For nonnegative integers $p$ and $q$, put $T^{p}_{q}(M):=(TM)^{\otimes p}\otimes(T^{*}M)^{\otimes q}$.
We denote by $\Gamma(T^{p}_{q}(M))$ the space of tensor fields of type $(p,q)$ on $M$, where $p$ and $q$ are the contravariant and covariant degrees, respectively.
We write $\Lambda^{k}(E)$ and $S^{k}(E)$ for the $k$th exterior and symmetric powers of $E$, respectively, and use the same notation for vector spaces.
All metrics are Riemannian.
Unless otherwise stated, all affine connections are assumed to be torsion free, and $X,Y,Z$ denote smooth vector fields on $M$.
We write $\mathfrak{X}(M):=\Gamma(TM)$.

%%%%%%%%%%%%%%%%%%%%%%%%%%%%
%%%%%%%%%%%%%%%%%%%%%%%%%%%%
\subsection{Statistical manifolds and alpha-connections}
%%%%%%%%%%%%%%%%%%%%%%%%%%%%
%%%%%%%%%%%%%%%%%%%%%%%%%%%%

\begin{definition}
Let $g$ be a Riemannian metric and $\nabla$ a torsion-free affine connection on $M$.
The triple $(M,g,\nabla)$ is called a \emph{statistical manifold} if the covariant $3$-tensor $C:=\nabla g$ is totally symmetric. 
In this case, the pair $(g,\nabla)$ is called a \emph{statistical structure} on $M$, and $\nabla$ the \emph{statistical connection}.
The tensor $C$ is called the \emph{cubic form} of $(g,\nabla)$.
\end{definition}

The \emph{dual connection} $\nabla^{*}$ of $\nabla$ with respect to $g$ is characterized by
\[
Xg(Y,Z)
=g(\nabla_{X}Y,Z)+g(Y,\nabla_{X}^{*}Z).
\]
If $(g,\nabla)$ is a statistical structure, the pair $(g,\nabla^{*})$ is also a statistical structure.

Let $\nabla^{g}$ be the Levi--Civita connection of $g$.
Since $\nabla^{g}g=0$, the pair $(g,\nabla^{g})$ is a statistical structure with vanishing cubic form.

We define the \emph{skewness tensor} $K$ by
\begin{equation}
g(K(X,Y),Z)=C(X,Y,Z).
\label{eq:K}
\end{equation}
For $X\in\mathfrak{X}(M)$, the \emph{skewness operator} $K_{X}\in\Gamma\bigl(\operatorname{End}(TM)\bigr)$ is the endomorphism field defined by
\[
K_{X}(Y):=K(X,Y)
\qquad
(Y\in\mathfrak{X}(M)).
\]
For $q\in M$ and $x\in T_{q}M$, we use the same notation $K_{x}\in\operatorname{End}(T_{q}M)$ for the endomorphism defined by
\[
K_{x}(y):=K_{q}(x,y)
\qquad
(y\in T_{q}M).
\]
Since $\nabla^{g}g=0$, differentiating \eqref{eq:K} gives
\[
g\bigl((\nabla^{g}_{W}K)(X,Y),Z\bigr)
=
(\nabla^{g}_{W}C)(X,Y,Z)
\]
for all $X,Y,Z,W\in\mathfrak{X}(M)$.
Thus, by nondegeneracy of $g$,
\begin{equation}
\nabla^{g}K=0
\quad\Longleftrightarrow\quad
\nabla^{g}C=0.
\label{eq:parallel-K-C}
\end{equation}
The statistical structure $(g,\nabla)$ is said to have a \emph{Levi--Civita parallel skewness tensor}, or equivalently a \emph{Levi--Civita parallel cubic form}, when these equivalent conditions hold.

For a fixed Riemannian metric $g$, a statistical structure can be described equivalently by its connection $\nabla$, its skewness tensor $K$, or its cubic form $C$.

\begin{lemma}
\label{lem:stat-data}
Let $g$ be a Riemannian metric on $M$.
The following three classes of pairs are in one-to-one correspondence\textup{:}
\begin{enumerate}[label={\rm(\roman*)}]
\item
A statistical structure $(g,\nabla)$.

\item
A pair $(g,K)$, where $K\in\Gamma(T^{1}_{2}(M))$ satisfies
\[
K(X,Y)=K(Y,X),\qquad
g\bigl(K_{X}(Y),Z\bigr)=g\bigl(Y,K_{X}(Z)\bigr).
\]

\item
A pair $(g,C)$, where $C\in\Gamma(S^{3}(T^{*}M))$.
\end{enumerate}
The correspondences are given by $C=\nabla g$ and \eqref{eq:K}.
The corresponding connection and its dual satisfy
\begin{equation}
\nabla=\nabla^{g}-\frac{1}{2}K,
\qquad
\nabla^{*}=\nabla^{g}+\frac{1}{2}K.
\label{eq:nabla-K}
\end{equation}
\end{lemma}

\begin{proof}
By nondegeneracy of $g$, \eqref{eq:K} gives a bijection between \textup{(ii)} and \textup{(iii)}: total symmetry of $C$ is equivalent to symmetry of $K$ and $g$-self-adjointness of each $K_{X}$.

Next, let $(g,\nabla)$ be a statistical structure, and define $C=\nabla g$ and $K$ by \eqref{eq:K}.
Then $(g,K)$ satisfies \textup{(ii)}.
The Koszul formula and the total symmetry of $C$ give
\[
2g\bigl(\nabla^{g}_{X}Y-\nabla_{X}Y,Z\bigr)
=C(X,Y,Z)
=g(K(X,Y),Z).
\]
Therefore, $\nabla=\nabla^{g}-\frac{1}{2}K$.

Conversely, let $(g,K)$ satisfy \textup{(ii)} and set $\nabla:=\nabla^{g}-\frac{1}{2}K$.
Symmetry of $K$ makes $\nabla$ torsion free.
Using $\nabla^{g}g=0$ and the $g$-self-adjointness of $K_{X}$, we obtain
\[
(\nabla_{X}g)(Y,Z)
=\frac{1}{2}\bigl\{
g(K(X,Y),Z)+g(Y,K(X,Z))
\bigr\}
=C(X,Y,Z).
\]
Thus $(g,\nabla)$ is a statistical structure with the prescribed $K$.
These constructions are inverse to one another, proving the bijection between \textup{(i)} and \textup{(ii)}.

Finally, the defining identity of the dual connection gives $\nabla^{*}=\nabla+K=\nabla^{g}+\frac{1}{2}K$.
\end{proof}

The tensor $\nabla-\nabla^{g}$ is also called the \emph{difference tensor}.
Thus the skewness tensor $K$ is $-2$ times the difference tensor.
Let $K^{*}$ denote the skewness tensor of the dual statistical structure $(g,\nabla^{*})$.
Since $\nabla^{*}g=-\nabla g$, one has
\[
K^{*}=-K.
\]
For $\alpha\in\mathbb{R}$, the associated \emph{$\alpha$-connection} is
\begin{equation}
\nabla^{(\alpha)}
:=\frac{1+\alpha}{2}\nabla
+\frac{1-\alpha}{2}\nabla^{*}
=\nabla^{g}-\frac{\alpha}{2}K.
\label{eq:alpha-family}
\end{equation}
Thus
\begin{equation}
\nabla^{(0)}=\nabla^{g},\qquad
\nabla^{(1)}=\nabla,\qquad
\nabla^{(-1)}=\nabla^{*}.
\label{eq:alpha-special}
\end{equation}
We use $\nabla^{g}$ when emphasizing the Levi--Civita connection of $g$ and $\nabla^{(0)}$ when emphasizing the member of the $\alpha$-family.

Moreover, we have
\begin{equation}
\nabla^{(\alpha)}g=\alpha C.
\label{eq:alpha-C}
\end{equation}
Indeed, substituting \eqref{eq:alpha-family} and using the $g$-self-adjointness of $K_{X}$ yields
\[
(\nabla^{(\alpha)}_{X}g)(Y,Z)
=\frac{\alpha}{2}
\{g(K(X,Y),Z)+g(Y,K(X,Z))\}
=\alpha C(X,Y,Z).
\]
Equivalently, $\nabla^{(\alpha)}$ is the statistical connection determined by the pair $(g,\alpha C)$.

\begin{remark}
\label{rem:rescale-stat-data}
Let $\lambda_{1}>0$ and $\lambda_{2}\in\mathbb{R}$ be constants, and set $(\overline{g},\overline{C}):=(\lambda_{1}g,\lambda_{2}C)$ on $M$.
Denote by $\overline{\nabla}$ and $\overline{K}$ the statistical connection and the skewness tensor determined by these data.
The defining relation \eqref{eq:K} gives $\overline{K}=(\lambda_{2}/\lambda_{1})K$.
Indeed,
\[
\overline{g}\left(
\frac{\lambda_{2}}{\lambda_{1}}K(X,Y),Z
\right)
=\lambda_{2}C(X,Y,Z)
=\overline{C}(X,Y,Z).
\]
Since a positive constant rescaling of a Riemannian metric preserves its Levi--Civita connection, one has $\nabla^{\overline{g}}=\nabla^{g}$.
Consequently,
\[
\overline{\nabla}
=\nabla^{\overline{g}}-\frac{1}{2}\overline{K}
=\nabla^{g}-\frac{1}{2}\frac{\lambda_{2}}{\lambda_{1}}K.
\]
In the three descriptions of a statistical structure in Lemma~\ref{lem:stat-data}, the same rescaling takes the following forms:
\[
\begin{aligned}
(g,\nabla)
&\longmapsto
\left(\lambda_{1}g,
\nabla^{g}-\frac{1}{2}\frac{\lambda_{2}}{\lambda_{1}}K\right),\\
(g,K)
&\longmapsto
\left(\lambda_{1}g,\frac{\lambda_{2}}{\lambda_{1}}K\right),\\
(g,C)
&\longmapsto
(\lambda_{1}g,\lambda_{2}C).
\end{aligned}
\]
For $\lambda_{2}=0$, the rescaled data determine the statistical structure $(\lambda_{1}g,\nabla^{g})$.
A common positive rescaling $(g,C)\mapsto(\lambda g,\lambda C)$ with $\lambda>0$ sends $(g,\nabla)$ to $(\lambda g,\nabla)$ while it preserves $K$.
Related observations concerning constant rescalings of statistical data appear in \cite{Inoguchi2026}.
\end{remark}

%%%%%%%%%%%%%%%%%%%%%%%%%%%%
%%%%%%%%%%%%%%%%%%%%%%%%%%%%
\subsection{Curvature of alpha-connections}
%%%%%%%%%%%%%%%%%%%%%%%%%%%%
%%%%%%%%%%%%%%%%%%%%%%%%%%%%

Let $\mathcal{R} \in \Gamma\bigl( \Lambda^{2}(T^{*}M)\otimes\operatorname{End}(TM) \bigr)$ be an $\operatorname{End}(TM)$-valued $2$-form.
For each $Y\in\mathfrak{X}(M)$, define the associated endomorphism field $\mathcal{R}_{Y}\in\Gamma\bigl(\operatorname{End}(TM)\bigr)$ by
\[
\mathcal{R}_{Y}(X):=\mathcal{R}(X,Y)Y
\qquad
(X\in\mathfrak{X}(M)).
\]
For each $q\in M$, the endomorphism $\bigl(\mathcal{R}_{Y}\bigr)_{q}$ depends only on $Y_{q}$.
Accordingly, for $y\in T_{q}M$, we write $\mathcal{R}_{y}:=\bigl(\mathcal{R}_{Y}\bigr)_{q}$, where $Y_{q}=y$.
We call $\mathcal{R}_{y}$ the \emph{Jacobi operator} associated with $\mathcal{R}$ in the direction $y$.

We call $\mathcal{R}$ \emph{curvature-like} with respect to $g$ if it satisfies the first Bianchi identity
\[
\mathcal{R}(X,Y)Z
+\mathcal{R}(Y,Z)X
+\mathcal{R}(Z,X)Y
=0
\]
and if $\mathcal{R}(X,Y)\in\Gamma\bigl(\operatorname{End}(TM)\bigr)$ is $g$-skew-adjoint:
\[
g\bigl(\mathcal{R}(X,Y)Z,W\bigr)
=
-g\bigl(Z,\mathcal{R}(X,Y)W\bigr)
\]
for all $X,Y,Z,W\in\mathfrak{X}(M)$.
Together with the skew-symmetry in the first two arguments, these conditions imply
\[
g\bigl(\mathcal{R}_{Y}(X),Z\bigr)
=
g\bigl(X,\mathcal{R}_{Y}(Z)\bigr),
\]
so $\mathcal{R}_{Y}$ is $g$-self-adjoint.

If $\mathcal{R}$ is curvature-like, its sectional curvature on a two-plane $\Pi\subset T_{q}M$ with basis $\{x,y\}$ is defined by
\[
\operatorname{Sec}_{\mathcal{R}}(\Pi)
:=
\frac{
g_{q}\bigl(\mathcal{R}_{y}(x),x\bigr)
}{
g_{q}(x,x)g_{q}(y,y)-g_{q}(x,y)^{2}
}.
\]
It is independent of the choice of $\{x,y\}$.
In particular, if $\{x,y\}$ is an orthonormal basis of $\Pi$, then one has $\operatorname{Sec}_{\mathcal{R}}(\Pi) = g_{q}\bigl(\mathcal{R}_{y}(x),x\bigr)$.

For a statistical manifold $(M,g,\nabla)$, the curvature tensor $R^{\nabla}$ of $\nabla$ is defined by
\[
R^{\nabla}(X,Y)Z
=
\nabla_{X}\nabla_{Y}Z
-\nabla_{Y}\nabla_{X}Z
-\nabla_{[X,Y]}Z.
\]
For $Y\in\mathfrak{X}(M)$, write $R^{\nabla}_{Y}(X):=R^{\nabla}(X,Y)Y$ for its Jacobi operator.
We write
\[
R^{(\alpha)}:=R^{\nabla^{(\alpha)}},
\qquad
R:=R^{\nabla}=R^{(1)},
\qquad
R^{*}:=R^{(-1)},
\qquad
R^{g}:=R^{(0)}.
\]
Since each $\nabla^{(\alpha)}$ is torsion free, $R^{(\alpha)}$ satisfies the first Bianchi identity.
The Riemannian curvature $R^{g}$ is curvature-like, but $R^{(\alpha)}(X,Y)$ need not be $g$-skew-adjoint when $\alpha\ne0$.

Define
\[
[K,K](X,Y)Z
:=[K_{X},K_{Y}]Z
=
K_{X}K_{Y}Z-K_{Y}K_{X}Z.
\]
Thus $[K,K]$ is an $\operatorname{End}(TM)$-valued $2$-form.
Regarding $K$ as the $\operatorname{End}(TM)$-valued one-form $X\mapsto K_{X}$, its covariant exterior derivative with respect to $\nabla$ is defined by
\[
\bigl(d^{\nabla}K\bigr)(X,Y)
:=
\nabla_{X}K_{Y}-\nabla_{Y}K_{X}-K_{[X,Y]}.
\]
Since $\nabla$ is torsion free, evaluating this identity on $Z$ gives
\begin{equation}
\bigl(d^{\nabla}K\bigr)(X,Y)Z
=
(\nabla_{X}K)(Y,Z)
-
(\nabla_{Y}K)(X,Z).
\label{eq:dK}
\end{equation}
Thus $d^{\nabla}K$ is the alternation of the covariant derivative tensor $\nabla K$ in its first two covariant arguments and is an $\operatorname{End}(TM)$-valued $2$-form:
\[
d^{\nabla}K
\in
\Gamma\bigl(
\Lambda^{2}(T^{*}M)\otimes\operatorname{End}(TM)
\bigr).
\]

\begin{proposition}
\label{prop:alpha-curv}
For every $\alpha\in\mathbb{R}$,
\begin{align}
&R^{(\alpha)}
=R^{g}-\frac{\alpha}{2}d^{\nabla^{g}}K
+\frac{\alpha^{2}}{4}[K,K],
\label{eq:Ralpha-Rg}\\
&d^{\nabla^{(\alpha)}}K
=d^{\nabla^{g}}K-\alpha[K,K],
\label{eq:dK-alpha}\\
&R^{g}
=R^{(\alpha)}
+\frac{\alpha}{2}d^{\nabla^{(\alpha)}}K
+\frac{\alpha^{2}}{4}[K,K].
\label{eq:Rg-Ralpha}
\end{align}
In particular,
\begin{equation}
R^{g}
=R+\frac{1}{2}d^{\nabla}K+\frac{1}{4}[K,K].
\label{eq:Rg-R}
\end{equation}
\end{proposition}

\begin{proof}
Substitution of \eqref{eq:alpha-family} into the curvature definition yields \eqref{eq:Ralpha-Rg}.
Alternating the covariant derivative of $K$ with respect to $\nabla^{(\alpha)}$ yields \eqref{eq:dK-alpha}.
Eliminating $d^{\nabla^{g}}K$ from these two identities yields \eqref{eq:Rg-Ralpha}, whose specialization to $\alpha=1$ is \eqref{eq:Rg-R}.
\end{proof}

The curvature tensors of dual connections satisfy
\begin{equation}
g(R^{(\alpha)}(X,Y)Z,W)
=
-g(Z,R^{(-\alpha)}(X,Y)W).
\label{eq:dual-curv}
\end{equation}
For $X,Y\in\mathfrak{X}(M)$, let $R(X,Y)^{\dagger}$ denote the $g$-adjoint of $R(X,Y)$, defined by
\[
g(R(X,Y)Z,W)
=
g\bigl(Z,R(X,Y)^{\dagger}W\bigr).
\]
Equation~\eqref{eq:dual-curv} with $\alpha=1$ gives $R(X,Y)^{\dagger}=-R^{*}(X,Y)$.

Following Lauritzen \cite{Lauritzen}, a statistical manifold is called \emph{conjugate symmetric} if $R=R^{*}$.
Equations~\eqref{eq:dK} and \eqref{eq:Ralpha-Rg} give
\begin{equation}
\nabla^{g}K=0
\;\Longrightarrow\;
d^{\nabla^{g}}K=0
\;\Longleftrightarrow\;
R=R^{*}.
\label{eq:conj-sym-dK}
\end{equation}

\begin{definition}
\label{def:stat-curv}
Define
\begin{equation}
R^{\mathsf{S}}
:=\frac{1}{2}(R+R^{*}),
\qquad
R^{\mathsf{A}}
:=\frac{1}{2}(R-R^{*}).
\label{eq:RS-RA}
\end{equation}
We call $R^{\mathsf{S}}$ the \emph{statistical curvature tensor} and $R^{\mathsf{A}}$ the \emph{conjugate-asymmetry tensor}.
\end{definition}

\begin{proposition}
\label{prop:curv-split}
The following assertions hold.
\begin{enumerate}[label=\textup{(\arabic*)}]
\item
The curvature tensors $R$ and $R^{*}$ decompose as
\[
R=R^{\mathsf{S}}+R^{\mathsf{A}},
\qquad
R^{*}=R^{\mathsf{S}}-R^{\mathsf{A}}.
\]

\item
For every $X,Y\in\mathfrak{X}(M)$,
\[
\bigl(R^{\mathsf{S}}(X,Y)\bigr)^{\dagger}
=
-R^{\mathsf{S}}(X,Y),
\qquad
\bigl(R^{\mathsf{A}}(X,Y)\bigr)^{\dagger}
=
R^{\mathsf{A}}(X,Y).
\]
Thus $R^{\mathsf{S}}$ is the $g$-skew-adjoint part of $R$, whereas $R^{\mathsf{A}}$ is its $g$-self-adjoint part.

\item
The tensors $R^{\mathsf{S}}$ and $R^{\mathsf{A}}$ are expressed in terms of $R^{g}$ and $K$ as
\begin{equation}
R^{\mathsf{S}}
=
R^{g}+\frac{1}{4}[K,K],
\qquad
R^{\mathsf{A}}
=
-\frac{1}{2}d^{\nabla^{g}}K.
\label{eq:RS-RA-K}
\end{equation}

\item
The tensor $R^{\mathsf{S}}$ is curvature-like.
Moreover, the statistical manifold is conjugate symmetric if and only if $R^{\mathsf{A}}=0$.
\end{enumerate}
\end{proposition}

\begin{proof}
Assertion \textup{(1)} follows directly from Definition~\ref{def:stat-curv}.
For \textup{(2)}, equation~\eqref{eq:dual-curv} with $\alpha=1$ gives
\[
R(X,Y)^{\dagger}=-R^{*}(X,Y),
\]
which yields the two adjoint relations.
Assertion \textup{(3)} follows from \eqref{eq:Ralpha-Rg} with $\alpha=1$ and $\alpha=-1$.
For \textup{(4)}, the curvature tensors $R$ and $R^{*}$ satisfy the first Bianchi identity, and assertion \textup{(2)} shows that $R^{\mathsf{S}}(X,Y)$ is $g$-skew-adjoint.
Hence $R^{\mathsf{S}}$ is curvature-like.
The equality $R^{\mathsf{A}}=0$ is equivalent to $R=R^{*}$.
\end{proof}

Equation~\eqref{eq:RS-RA-K} also gives $[K,K]=4(R^{\mathsf{S}}-R^{g})$.
Since $R^{\mathsf{S}}$ and $R^{g}$ are curvature-like, $[K,K]$ is curvature-like as well.
The sectional curvature $\operatorname{Sec}_{R^{\mathsf{S}}}$ determined by $R^{\mathsf{S}}$ is the \emph{statistical sectional curvature} considered in \cite{FuruhataHasegawa}.

%%%%%%%%%%%%%%%%%%%%%%%%%%%%
%%%%%%%%%%%%%%%%%%%%%%%%%%%%
\subsection{Hessian and dual Hessian curvature}
%%%%%%%%%%%%%%%%%%%%%%%%%%%%
%%%%%%%%%%%%%%%%%%%%%%%%%%%%

A statistical manifold $(M,g,\nabla)$ is \emph{Hessian} when $R=0$.
By \eqref{eq:dual-curv}, its dual statistical manifold $(M,g,\nabla^{*})$ is also Hessian.

Following Shima \cite{Shima1995}, we use the following definition and apply it also to the dual Hessian structure.

\begin{definition}
\label{def:Hess-curv}
Let $(M,g,\nabla)$ be a Hessian manifold.
Its \emph{Hessian curvature tensor} is
\begin{equation}
Q(X,Y)Z:=\frac{1}{2}(\nabla_{X}K)(Y,Z).
\label{eq:Q}
\end{equation}
We call the Hessian curvature tensor of the dual Hessian manifold $(M,g,\nabla^{*})$ the \emph{dual Hessian curvature tensor} and denote it by
\[
Q^{*}(X,Y)Z
:=\frac{1}{2}(\nabla^{*}_{X}K^{*})(Y,Z).
\]
\end{definition}

\begin{proposition}
\label{prop:Hess-curv-dual}
The following assertions hold.
\begin{enumerate}[label=\textup{(\arabic*)}]
\item
The dual Hessian curvature tensor satisfies
\[
Q^{*}=-\frac{1}{2}\nabla^{*}K.
\]

\item
The Hessian and dual Hessian curvature tensors satisfy
\begin{equation}
Q-Q^{*}=\nabla^{g}K.
\label{eq:Q-dual-diff}
\end{equation}

\item
The following equivalence holds\textup{:}
\[
Q=Q^{*}
\quad\Longleftrightarrow\quad
\nabla^{g}K=0.
\]
\end{enumerate}
\end{proposition}

\begin{proof}
Assertion \textup{(1)} follows from $K^{*}=-K$.
For \textup{(2)}, covariant differentiation of a fixed tensor depends affine-linearly on the connection.
Since $\nabla^{g}=(\nabla+\nabla^{*})/2$, it follows that
\[
\frac{1}{2}\bigl(\nabla K+\nabla^{*}K\bigr)
=\nabla^{g}K.
\]
Together with the definitions of $Q$ and $Q^{*}$, this proves \eqref{eq:Q-dual-diff}.
Assertion \textup{(3)} follows immediately from \textup{(2)}.
\end{proof}

In the Hessian case, \eqref{eq:RS-RA-K} gives $R^{g}=-\frac{1}{4}[K,K]$.
Thus the Riemannian curvature is determined algebraically by $K$.

%%%%%%%%%%%%%%%%%%%%%%%%%%%%
%%%%%%%%%%%%%%%%%%%%%%%%%%%%
%%%%%%%%%%%%%%%%%%%%%%%%%%%%
\section{Statisticmorphisms}
\label{sec:stat-morph}
%%%%%%%%%%%%%%%%%%%%%%%%%%%%
%%%%%%%%%%%%%%%%%%%%%%%%%%%%
%%%%%%%%%%%%%%%%%%%%%%%%%%%%

This section gives equivalent conditions for statisticmorphisms and characterizes infinitesimal statistical automorphisms by Lie derivatives.

%%%%%%%%%%%%%%%%%%%%%%%%%%%%
%%%%%%%%%%%%%%%%%%%%%%%%%%%%
\subsection{Statisticmorphisms and statistical automorphisms}
%%%%%%%%%%%%%%%%%%%%%%%%%%%%
%%%%%%%%%%%%%%%%%%%%%%%%%%%%

Let $(M,g,\nabla)$ and $(\overline{M},\overline{g},\overline{\nabla})$ be statistical manifolds of the same dimension.
Their cubic forms, skewness tensors, and $\alpha$-families are denoted by
\[
C,\quad K,\quad \nabla^{(\alpha)}
\qquad\hbox{and}\qquad
\overline{C},\quad \overline{K},\quad
\overline{\nabla}^{(\alpha)}.
\]
Their curvature tensors are denoted by
\[
R^{(\alpha)}:=R^{\nabla^{(\alpha)}},
\qquad
\overline{R}^{(\alpha)}:=R^{\overline{\nabla}^{(\alpha)}}.
\]
We write $\nabla^{\overline{g}}$ for the Levi--Civita connection of $\overline{g}$, and $R^{\overline{g}}$ for its curvature tensor.
For a diffeomorphism $\varphi:M\to\overline{M}$, the pushforward $\varphi_{*}:\Gamma(T^{p}_{q}(M)) \longrightarrow\Gamma(T^{p}_{q}(\overline{M}))$ is given by
\[
(\varphi_{*}F)(\varphi(x))
=\left[
\bigl((d\varphi)_{x}\bigr)^{\otimes p}
\otimes
\bigl(((d\varphi^{-1})_{\varphi(x)})^{*}\bigr)^{\otimes q}
\right]F(x).
\]
The pushforward of an affine connection $\nabla$ on $M$ is defined by
\[
(\varphi_{*}\nabla)_{\varphi_{*}X}(\varphi_{*}Y)
:=\varphi_{*}(\nabla_{X}Y).
\]
The pullback $\varphi^{*}:\Gamma(T^{p}_{q}(\overline{M})) \longrightarrow\Gamma(T^{p}_{q}(M))$ is defined by $\varphi^{*}\overline{F}=(\varphi^{-1})_{*}\overline{F}$.
For an affine connection $\overline{\nabla}$ on $\overline{M}$, put $\varphi^{*}\overline{\nabla} =(\varphi^{-1})_{*}\overline{\nabla}$.

\begin{definition}
A diffeomorphism $\varphi$ between statistical manifolds is called a \emph{statisticmorphism} (or a \emph{statistical diffeomorphism}) if
\[
\varphi^{*}\overline{g}=g,
\qquad
\varphi_{*}\nabla=\overline{\nabla}.
\]
For $\alpha,\beta\in\mathbb{R}$ satisfying $\alpha\ne0$ and $\beta\ne0$, a diffeomorphism $\varphi$ is called an \emph{$(\alpha,\beta)$-statisticmorphism} if
\begin{equation}
\varphi^{*}\overline{g}=g,
\qquad
\varphi_{*}\nabla^{(\alpha)}
=
\overline{\nabla}^{(\beta)}.
\label{eq:alpha-beta-morph}
\end{equation}
\end{definition}

The statisticmorphisms from $(M,g,\nabla)$ to itself form its \emph{statistical automorphism group}
\[
\operatorname{Aut}(M,g,\nabla)
=\Iso(M,g)\cap\operatorname{Aff}(M,\nabla),
\]
where $\Iso(M,g)$ is the isometry group and $\operatorname{Aff}(M,\nabla)$ is the group of diffeomorphisms of $M$ preserving $\nabla$.
A statisticmorphism between open subsets of $(M,g,\nabla)$ is called a \emph{local statistical automorphism}.

\begin{proposition}
\label{prop:stat-morph-equiv}
Let $\alpha,\beta\in\mathbb{R}$ satisfy $\alpha\ne0$ and $\beta\ne0$, and let $\varphi$ be a diffeomorphism satisfying $\varphi^{*}\overline{g}=g$.
The following conditions are equivalent.
\begin{enumerate}[label={\rm(\roman*)}]
\item
$\varphi$ is an $(\alpha,\beta)$-statisticmorphism.

\item
$\alpha\varphi_{*}\bigl(K(X,Y)\bigr) = \beta\overline{K}(\varphi_{*}X,\varphi_{*}Y)$ for all vector fields $X,Y$.

\item
$\alpha C=\beta\varphi^{*}\overline{C}$.

\item
$\varphi$ is a $(-\alpha,-\beta)$-statisticmorphism.

\item
For some $\eta\ne0$, $\varphi_{*}\nabla^{(\eta)} =\overline{\nabla}^{(\eta\beta/\alpha)}$.

\item
For every $\eta\in\mathbb{R}$,
\begin{equation}
\varphi_{*}\nabla^{(\eta)}
=
\overline{\nabla}^{(\eta\beta/\alpha)}.
\label{eq:family-scaling}
\end{equation}
\end{enumerate}
\end{proposition}

\begin{proof}
Since $\varphi$ is an isometry, it preserves the Levi--Civita connections.
Equation~\eqref{eq:alpha-family} gives
\begin{equation}
\varphi_{*}\nabla^{(\eta)}
=
\nabla^{\overline{g}}-\frac{\eta}{2}\varphi_{*}K
\qquad(\eta\in\mathbb{R}).
\label{eq:eta-K}
\end{equation}
Taking $\eta=\alpha$ and comparing this expression with $\overline{\nabla}^{(\beta)} =\nabla^{\overline{g}}-(\beta/2)\overline{K}$ proves the equivalence of \textup{(i)} and \textup{(ii)}.
Taking the $\overline{g}$-inner product of the relation in \textup{(ii)} with $\varphi_{*}Z$ and using $\varphi^{*}\overline{g}=g$ proves the equivalence of \textup{(ii)} and \textup{(iii)}.
Simultaneous replacement of $(\alpha,\beta)$ by $(-\alpha,-\beta)$ leaves \textup{(ii)} unchanged, so \textup{(iv)} is equivalent to \textup{(i)}.

Under \textup{(ii)}, substitution of $\varphi_{*}K=(\beta/\alpha)\overline{K}$ into \eqref{eq:eta-K} proves \textup{(vi)}.
The implication \textup{(vi)} $\Rightarrow$ \textup{(v)} is trivial.
Conversely, \textup{(v)} and the same formula give $\eta\varphi_{*}K=(\eta\beta/\alpha)\overline{K}$.
Since $\eta\ne0$, this implies \textup{(ii)}.
\end{proof}

For $\alpha=\beta=1$, this proposition characterizes ordinary statisticmorphisms.

For automorphisms of a fixed statistical manifold, the definition and the preceding equivalences give
\[
\begin{aligned}
\operatorname{Aut}(M,g,\nabla)
&=\Iso(M,g)\cap\operatorname{Aff}(M,\nabla)\\
&=\{\varphi\in\Iso(M,g)\mid \varphi_{*}K=K\}.
\end{aligned}
\]
When $M$ is connected, $\operatorname{Aut}(M,g,\nabla)$ is a closed Lie subgroup of $\Iso(M,g)$.
The action is $\operatorname{Aut}(M,g,\nabla)\times M\to M$, $(\varphi,q)\mapsto \varphi(q)$.
This action is smooth because it is the restriction of the smooth action of $\Iso(M,g)$.
Proposition~\ref{prop:stat-morph-equiv} gives
\[
\operatorname{Aut}(M,g,\nabla^{(\eta)})
=\operatorname{Aut}(M,g,\nabla)
\qquad(\eta\ne0),
\]
whereas $\operatorname{Aut}(M,g,\nabla^{(0)})=\Iso(M,g)$.

\begin{remark}
\label{rem:rescale-stat-morph}
For the rescaling $(\overline{g},\overline{C})=(\lambda_{1}g,\lambda_{2}C)$ in Remark~\ref{rem:rescale-stat-data}, the associated connections satisfy
\[
\overline{\nabla}^{(\beta)}
=\nabla^{g}-\frac{\beta\lambda_{2}}{2\lambda_{1}}K
=\nabla^{(\beta\lambda_{2}/\lambda_{1})}
\qquad(\beta\in\mathbb{R}).
\]
Consequently, for all $\alpha,\beta\in\mathbb{R}$,
\[
\nabla^{(\alpha)}=\overline{\nabla}^{(\beta)}
\quad\Longleftrightarrow\quad
\alpha K=\beta\overline{K}
\quad\Longleftrightarrow\quad
\left(
\alpha-\frac{\beta\lambda_{2}}{\lambda_{1}}
\right)K=0.
\]
Thus, if $K\not\equiv0$, the two connections agree if and only if $\alpha=\beta\lambda_{2}/\lambda_{1}$.
If $K\equiv0$, both families consist only of $\nabla^{g}$.

For $\lambda_{2}\ne0$, the rescaled connection family is obtained from the original family by the parameter change $\beta\mapsto\beta\lambda_{2}/\lambda_{1}$.
For $\lambda_{2}=0$, every $\overline{\nabla}^{(\beta)}$ equals $\nabla^{g}$.
For $\lambda_{1}=\lambda_{2}>0$, we have $\overline{\nabla}^{(\alpha)}=\nabla^{(\alpha)}$ for every $\alpha\in\mathbb{R}$.

For $\alpha,\beta\ne0$, the relation $\alpha K=\beta\overline{K}$ is the condition on the skewness tensors in Proposition~\ref{prop:stat-morph-equiv} for the identity map.
Under this condition, the identity map from $(M,g,\nabla)$ to $(M,\overline{g},\overline{\nabla})$ is an $(\alpha,\beta)$-statisticmorphism if and only if it also preserves the metric, that is, $\overline{g}=g$.
If $\dim M>0$, this condition is equivalent to $\lambda_{1}=1$.
\end{remark}

%%%%%%%%%%%%%%%%%%%%%%%%%%%%
%%%%%%%%%%%%%%%%%%%%%%%%%%%%
\subsection{Infinitesimal statistical automorphisms}
\label{subsec:inf-stat-aut}
%%%%%%%%%%%%%%%%%%%%%%%%%%%%
%%%%%%%%%%%%%%%%%%%%%%%%%%%%

Let $\varphi_{t}$ be the local flow generated by a vector field $X$ on $M$.
The Lie derivatives of a tensor field $F\in\Gamma(T^{p}_{q}(M))$ and an affine connection $\nabla$ along $X$ are defined by
\[
\pounds_{X}F
=\left.\frac{d}{dt}\right|_{t=0}\varphi_{t}^{*}F,
\qquad
\pounds_{X}\nabla
=\left.\frac{d}{dt}\right|_{t=0}\varphi_{t}^{*}\nabla.
\]
Here $\varphi_{t}^{*}\nabla=(\varphi_{-t})_{*}\nabla$, and $\pounds_{X}\nabla\in\Gamma(T^{1}_{2}(M))$.

\begin{definition}
\label{def:inf-stat-aut}
A vector field $X$ on $M$ is called a \emph{Killing vector field} if its local flow consists of local isometries.
It is called an \emph{affine Killing vector field} for $\nabla$ if its local flow preserves $\nabla$.
It is called an \emph{infinitesimal statistical automorphism} if its local flow consists of local statistical automorphisms.
\end{definition}

Equivalently, the corresponding spaces of vector fields are
\[
\begin{aligned}
\mathfrak{iso}(M,g)
&:=\{X\in\mathfrak{X}(M)\mid\pounds_{X}g=0\},\\
\mathfrak{aff}(M,\nabla)
&:=\{X\in\mathfrak{X}(M)\mid\pounds_{X}\nabla=0\},\\
\mathfrak{aut}(M,g,\nabla)
&:=\mathfrak{iso}(M,g)\cap\mathfrak{aff}(M,\nabla).
\end{aligned}
\]
For an open subset $O\subset M$, the notation $\mathfrak{iso}(O,g)$, $\mathfrak{aff}(O,\nabla)$, and $\mathfrak{aut}(O,g,\nabla)$ refers to the corresponding spaces of vector fields on $O$, with $g$ and $\nabla$ restricted to $O$.

For vector fields $Y,Z$ on $M$, metric compatibility and torsion-freeness of $\nabla^{g}$ give
\[
\begin{aligned}
(\pounds_{X}g)(Y,Z)
&=X\bigl(g(Y,Z)\bigr)-g([X,Y],Z)-g(Y,[X,Z])\\
&=g(\nabla^{g}_{Y}X,Z)+g(Y,\nabla^{g}_{Z}X).
\end{aligned}
\]
The second covariant derivative $\nabla^{2}X:=\nabla(\nabla X)$ satisfies $(\nabla^{2}X)(Y,Z) =\nabla_{Y}\nabla_{Z}X-\nabla_{\nabla_{Y}Z}X$.
Since $\nabla$ is torsion free, the Lie derivative of the connection is given by
\[
\begin{aligned}
(\pounds_{X}\nabla)(Y,Z)
&=[X,\nabla_{Y}Z]-\nabla_{[X,Y]}Z-\nabla_{Y}[X,Z]\\
&=(\nabla^{2}X)(Y,Z)+R(X,Y)Z.
\end{aligned}
\]

Fix $p\in M$, and let $O\subset M$ be a connected open neighborhood of $p$.

\begin{proposition}
\label{prop:inf-stat-aut}
The following assertions hold on $O$.
\begin{enumerate}[label={\rm(\roman*)}]
\item
The spaces $\mathfrak{iso}(O,g)$ and $\mathfrak{aff}(O,\nabla)$ are Lie algebras under the Lie bracket.
Their intersection $\mathfrak{aut}(O,g,\nabla)$ is a Lie subalgebra of both.
Moreover,
\[
\mathfrak{aut}(O,g,\nabla)
=\{X\in\mathfrak{iso}(O,g)\mid\pounds_{X}K=0\}.
\]

\item
A vector field $X\in\mathfrak{aff}(O,\nabla)$ is uniquely determined by $X_{p}$ and $(\nabla X)_{p}$.
A vector field $X\in\mathfrak{iso}(O,g)$ is uniquely determined by $X_{p}$ and $(\nabla^{g}X)_{p}$.
In particular, all three Lie algebras are finite-dimensional.
\end{enumerate}
\end{proposition}

\begin{proof}
For \textup{(i)}, the real linearity of the Lie derivative in $X$ shows that $\mathfrak{iso}(O,g)$ and $\mathfrak{aff}(O,\nabla)$ are real vector spaces.
For vector fields $X,Y$ on $O$, the commutator identity for tensor fields gives
\[
\pounds_{[X,Y]}g
=\pounds_{X}(\pounds_{Y}g)-\pounds_{Y}(\pounds_{X}g).
\]
The formula for the Lie derivative of a connection above and the Jacobi identity give
\[
\pounds_{[X,Y]}\nabla
=\pounds_{X}(\pounds_{Y}\nabla)-\pounds_{Y}(\pounds_{X}\nabla),
\]
where the outer Lie derivatives act on the $(1,2)$-tensors $\pounds_{Y}\nabla$ and $\pounds_{X}\nabla$.
Therefore, both spaces are closed under the Lie bracket and are Lie subalgebras of $\mathfrak{X}(O)$.
Their intersection $\mathfrak{aut}(O,g,\nabla)$ is a Lie subalgebra of both.

For $X\in\mathfrak{iso}(O,g)$, the local flow preserves $\nabla^{g}$.
Hence \eqref{eq:nabla-K} gives
\[
\pounds_{X}\nabla=-\frac{1}{2}\pounds_{X}K.
\]
Therefore, for such $X$, the condition $\pounds_{X}\nabla=0$ is equivalent to $\pounds_{X}K=0$.

For \textup{(ii)}, let $X\in\mathfrak{aff}(O,\nabla)$, and let $\gamma$ be a $\nabla$-geodesic in $O$ with $\gamma(0)=p$ and $\dot{\gamma}(0)=v$.
Put $J(t):=X_{\gamma(t)}$.
Since $\nabla_{\dot{\gamma}}\dot{\gamma}=0$, the formula for $\pounds_{X}\nabla$ gives the Jacobi equation
\[
0=(\pounds_{X}\nabla)(\dot{\gamma},\dot{\gamma})
=\left(\nabla_{\dot{\gamma}}\right)^{2}J
+R_{\dot{\gamma}}(J).
\]
Its initial data are
\[
J(0)=X_{p},
\qquad
\left(\nabla_{\dot{\gamma}}J\right)(0)
=(\nabla X)_{p}(v).
\]
Uniqueness for this second-order linear equation along the geodesics issuing from $p$ determines $X$ on a normal neighborhood of $p$ contained in $O$ from $X_{p}$ and $(\nabla X)_{p}$.
Repeating the argument in overlapping normal neighborhoods gives uniqueness throughout $O$.
For $X\in\mathfrak{iso}(O,g)$, the same argument applies to $\nabla^{g}$, since its local flow preserves $\nabla^{g}$.

Hence the linear maps from $\mathfrak{aff}(O,\nabla)$ and $\mathfrak{iso}(O,g)$ to $T_{p}M\oplus\operatorname{End}(T_{p}M)$ given, respectively, by
\[
X\longmapsto\bigl(X_{p},(\nabla X)_{p}\bigr),
\qquad
X\longmapsto\bigl(X_{p},(\nabla^{g}X)_{p}\bigr)
\]
are injective.
Therefore, $\mathfrak{aff}(O,\nabla)$ and $\mathfrak{iso}(O,g)$ are finite-dimensional, as is their common subspace $\mathfrak{aut}(O,g,\nabla)$.
\end{proof}

%%%%%%%%%%%%%%%%%%%%%%%%%%%%
%%%%%%%%%%%%%%%%%%%%%%%%%%%%
%%%%%%%%%%%%%%%%%%%%%%%%%%%%
\section{The alpha-connection formulation of the statistical CAH theorem}
\label{sec:alpha-CAH}
%%%%%%%%%%%%%%%%%%%%%%%%%%%%
%%%%%%%%%%%%%%%%%%%%%%%%%%%%
%%%%%%%%%%%%%%%%%%%%%%%%%%%%

This section proves a local statistical Cartan--Ambrose--Hicks theorem for $(\alpha,\beta)$-statisticmorphisms with $\alpha\ne0$ and $\beta\ne0$.
The proof uses Jacobi fields to determine the differential of the spray map and Gram matrices to prove metric preservation.
The construction of the spray map and the comparison lemmas are formulated for arbitrary real parameters.
The cases in which one or both parameters vanish are discussed in Remark~\ref{rem:zero-alpha-beta}.
Section~\ref{sec:LC-CAH} treats the Levi--Civita case $\alpha=\beta=0$.

%%%%%%%%%%%%%%%%%%%%%%%%%%%%
%%%%%%%%%%%%%%%%%%%%%%%%%%%%
\subsection{Spray maps and parallel displacement}
%%%%%%%%%%%%%%%%%%%%%%%%%%%%
%%%%%%%%%%%%%%%%%%%%%%%%%%%%

If $\gamma:[0,1]\to M$ is a smooth curve and $t\in[0,1]$, its $\nabla$-parallel displacement from $\gamma(0)$ to $\gamma(t)$ is the linear isomorphism
\[
P^{\nabla}_{\gamma,0\to t}:
T_{\gamma(0)}M\longrightarrow T_{\gamma(t)}M,
\qquad
P^{\nabla}_{\gamma,0\to t}(v):=V_{v}(t),
\]
where $V_{v}$ is the unique vector field along $\gamma$ satisfying $\nabla_{\dot{\gamma}}V_{v}=0$ and $V_{v}(0)=v$.

We use the following setting throughout the remainder of this section.

\begin{setting}
\label{set:alpha-beta-CAH}
Fix $\alpha,\beta\in\mathbb{R}$, points $p\in M$, $\overline{p}\in\overline{M}$, and a linear isometry
\[
\Phi:(T_{p}M,g_{p})\longrightarrow
(T_{\overline{p}}\overline{M},\overline{g}_{\overline{p}}).
\]
We write $\operatorname{Exp}_{p}^{(\alpha)}$ for the exponential map of $\nabla^{(\alpha)}$ at $p$ and use the corresponding barred notation for $\overline{\nabla}^{(\beta)}$ at $\overline{p}$.

For $v\in T_{p}M$, let $\gamma_{v}^{(\alpha)}:I_{v}^{(\alpha)}\longrightarrow M$ be the maximal $\nabla^{(\alpha)}$-geodesic satisfying $\gamma_{v}^{(\alpha)}(0)=p$ and $\dot{\gamma}_{v}^{(\alpha)}(0)=v$, and put
\[
\Omega_{p}^{(\alpha)}
:=
\{v\in T_{p}M\mid 1\in I_{v}^{(\alpha)}\}.
\]
Similarly, for $\overline{v}\in T_{\overline{p}}\overline{M}$, let
\[
\overline{\gamma}_{\overline{v}}^{(\beta)}:
\overline{I}_{\overline{v}}^{(\beta)}
\longrightarrow\overline{M}
\]
be the maximal $\overline{\nabla}^{(\beta)}$-geodesic satisfying $\overline{\gamma}_{\overline{v}}^{(\beta)}(0)=\overline{p}$ and $\dot{\overline{\gamma}}_{\overline{v}}^{(\beta)}(0)=\overline{v}$, and put
\[
\overline{\Omega}_{\overline{p}}^{(\beta)}
:=
\{\overline{v}\in T_{\overline{p}}\overline{M}
\mid 1\in\overline{I}_{\overline{v}}^{(\beta)}\}.
\]
The corresponding exponential maps are
\[
\operatorname{Exp}_{p}^{(\alpha)}:
\Omega_{p}^{(\alpha)}\longrightarrow M,
\qquad
\operatorname{Exp}_{p}^{(\alpha)}(v)
:=
\gamma_{v}^{(\alpha)}(1),
\]
and
\[
\overline{\operatorname{Exp}}_{\overline{p}}^{(\beta)}:
\overline{\Omega}_{\overline{p}}^{(\beta)}
\longrightarrow\overline{M},
\qquad
\overline{\operatorname{Exp}}_{\overline{p}}^{(\beta)}
(\overline{v})
:=
\overline{\gamma}_{\overline{v}}^{(\beta)}(1).
\]
The sets $\Omega_{p}^{(\alpha)}$ and $\overline{\Omega}_{\overline{p}}^{(\beta)}$ are open and star-shaped with respect to $0_{p}$ and $0_{\overline{p}}$, respectively.
By the inverse function theorem, we may choose a convex open neighborhood
\[
\mathcal{O}_{p}^{(\alpha,\beta)}
\subset
\Omega_{p}^{(\alpha)}
\cap
\Phi^{-1}\bigl(
\overline{\Omega}_{\overline{p}}^{(\beta)}
\bigr)
\]
of $0_{p}$ and open neighborhoods $O^{(\alpha,\beta)}(p)$ of $p$ and $\overline{O}^{(\alpha,\beta)}(\overline{p})$ of $\overline{p}$ such that
\[
\operatorname{Exp}_{p}^{(\alpha)}:
\mathcal{O}_{p}^{(\alpha,\beta)}
\longrightarrow O^{(\alpha,\beta)}(p)
\]
and
\[
\overline{\operatorname{Exp}}_{\overline{p}}^{(\beta)}:
\Phi\bigl(\mathcal{O}_{p}^{(\alpha,\beta)}\bigr)
\longrightarrow
\overline{O}^{(\alpha,\beta)}(\overline{p})
\]
are diffeomorphisms.
Here and below, the inverses of the exponential maps refer to these restrictions.
Define the local diffeomorphism
\[
\varphi^{(\alpha,\beta)}:
O^{(\alpha,\beta)}(p)
\longrightarrow
\overline{O}^{(\alpha,\beta)}(\overline{p})
\]
by
\begin{equation}
\varphi^{(\alpha,\beta)}
:=
\overline{\operatorname{Exp}}_{\overline{p}}^{(\beta)}
\circ\Phi
\circ
\bigl(\operatorname{Exp}_{p}^{(\alpha)}\bigr)^{-1}.
\label{eq:varphi-alpha-beta}
\end{equation}
Following Hicks \cite{Hicks1966}, we call $\varphi^{(\alpha,\beta)}$ the \emph{spray map} determined by $(p,\overline{p},\Phi)$ with respect to $\nabla^{(\alpha)}$ and $\overline{\nabla}^{(\beta)}$.
Equivalently, the following diagram commutes:
\[
\begin{tikzcd}[column sep=large,row sep=large]
\mathcal{O}_{p}^{(\alpha,\beta)}
  \arrow[r,"\Phi"]
  \arrow[d,"\operatorname{Exp}_{p}^{(\alpha)}"']
&
\Phi\bigl(\mathcal{O}_{p}^{(\alpha,\beta)}\bigr)
  \arrow[d,
  "\overline{\operatorname{Exp}}_{\overline{p}}^{(\beta)}"]
\\
O^{(\alpha,\beta)}(p)
  \arrow[r,"\varphi^{(\alpha,\beta)}"']
&
\overline{O}^{(\alpha,\beta)}(\overline{p})
\end{tikzcd}
\]
For $v\in\mathcal{O}_{p}^{(\alpha,\beta)}$ and $0\leq t\leq1$, the affine reparametrization property of geodesics and \eqref{eq:varphi-alpha-beta} give
\begin{align*}
\gamma_{v}^{(\alpha)}(t)
&=
\operatorname{Exp}_{p}^{(\alpha)}(tv),\\
\overline{\gamma}_{v}^{(\beta)}(t)
&:=
\overline{\gamma}_{\Phi(v)}^{(\beta)}(t)
=
\overline{\operatorname{Exp}}_{\overline{p}}^{(\beta)}
\bigl(t\Phi(v)\bigr)
=
\varphi^{(\alpha,\beta)}
\bigl(\gamma_{v}^{(\alpha)}(t)\bigr).
\end{align*}
Since they are obtained by applying the respective exponential maps to the radial segments $t\mapsto tv$ and $t\mapsto t\Phi(v)$, we call $\gamma_{v}^{(\alpha)}$ and $\overline{\gamma}_{v}^{(\beta)}$ the corresponding \emph{radial geodesics} from $p$ and $\overline{p}$, respectively.
For $v\ne0$, their velocity vectors determine the corresponding \emph{radial directions}.

Let
\[
P^{(\alpha)}_{v,t}
:=
P^{\nabla^{(\alpha)}}_{\gamma_{v}^{(\alpha)},0\to t},
\qquad
\overline{P}^{(\beta)}_{v,t}
:=
P^{\overline{\nabla}^{(\beta)}}
_{\overline{\gamma}_{v}^{(\beta)},0\to t},
\]
and define the linear isomorphism
\[
\Phi^{(\alpha,\beta)}_{v,t}:
T_{\gamma_{v}^{(\alpha)}(t)}M
\longrightarrow
T_{\overline{\gamma}_{v}^{(\beta)}(t)}\overline{M}
\]
by
\begin{equation}
\Phi^{(\alpha,\beta)}_{v,t}
:=
\overline{P}^{(\beta)}_{v,t}
\circ\Phi
\circ
\bigl(P^{(\alpha)}_{v,t}\bigr)^{-1}.
\label{eq:Phi-alpha-beta}
\end{equation}
We call $\Phi^{(\alpha,\beta)}_{v,t}$ the \emph{parallel identification} determined by $\Phi$ along the corresponding radial geodesics.
By construction, the following diagram commutes:
\[
\begin{tikzcd}[column sep=large,row sep=large]
T_{p}M
  \arrow[r,"\Phi"]
  \arrow[d,"P^{(\alpha)}_{v,t}"']
&
T_{\overline{p}}\overline{M}
  \arrow[d,"\overline{P}^{(\beta)}_{v,t}"]
\\
T_{\gamma_{v}^{(\alpha)}(t)}M
  \arrow[r,"\Phi^{(\alpha,\beta)}_{v,t}"']
&
T_{\overline{\gamma}_{v}^{(\beta)}(t)}\overline{M}
\end{tikzcd}
\]
In general, $\Phi^{(\alpha,\beta)}_{v,t}$ need not be an isometry, since the relevant parallel displacements need not preserve the metrics.
Since the two curves are geodesics,
\[
P^{(\alpha)}_{v,t}(v)
=
\dot{\gamma}_{v}^{(\alpha)}(t),
\qquad
\overline{P}^{(\beta)}_{v,t}\bigl(\Phi(v)\bigr)
=
\dot{\overline{\gamma}}_{v}^{(\beta)}(t).
\]
It follows from \eqref{eq:Phi-alpha-beta} that
\begin{equation}
\Phi^{(\alpha,\beta)}_{v,t}
\bigl(\dot{\gamma}_{v}^{(\alpha)}(t)\bigr)
=
\dot{\overline{\gamma}}_{v}^{(\beta)}(t).
\label{eq:radial-tangent}
\end{equation}
\end{setting}

%%%%%%%%%%%%%%%%%%%%%%%%%%%%
%%%%%%%%%%%%%%%%%%%%%%%%%%%%
\subsection{Jacobi fields and metric preservation}
%%%%%%%%%%%%%%%%%%%%%%%%%%%%
%%%%%%%%%%%%%%%%%%%%%%%%%%%%

The proof of the local theorem is based on two lemmas.
The first identifies the differential of the spray map by the classical argument using Jacobi fields; see \cite{Hicks1966}.
The second derives metric preservation from the relation between the scaled skewness operators in the radial tangent direction.

\begin{lemma}
\label{lem:radial-Jacobi}
In Setting~\textup{\ref{set:alpha-beta-CAH}}, with arbitrary $\alpha,\beta\in\mathbb{R}$, suppose that, for every $v\in\mathcal{O}_{p}^{(\alpha,\beta)}$ and every $t\in[0,1]$, the corresponding Jacobi operators satisfy
\begin{equation}
\Phi^{(\alpha,\beta)}_{v,t}\circ
R^{(\alpha)}_{\dot{\gamma}_{v}^{(\alpha)}(t)}
=
\overline{R}^{(\beta)}_{\dot{\overline{\gamma}}_{v}^{(\beta)}(t)}
\circ\Phi^{(\alpha,\beta)}_{v,t}.
\label{eq:radial-Ralpha}
\end{equation}
Then, for every $v\in\mathcal{O}_{p}^{(\alpha,\beta)}$ and every $t\in[0,1]$,
\begin{equation}
(d\varphi^{(\alpha,\beta)})_{\gamma_{v}^{(\alpha)}(t)}
=
\Phi^{(\alpha,\beta)}_{v,t}.
\label{eq:dvarphi-Phi}
\end{equation}
\end{lemma}

\begin{proof}
Fix $v\in\mathcal{O}_{p}^{(\alpha,\beta)}$.
For each $w\in T_{p}M$, choose $\varepsilon>0$ such that $v+sw\in\mathcal{O}_{p}^{(\alpha,\beta)}$ whenever $|s|<\varepsilon$.
The convexity of $\mathcal{O}_{p}^{(\alpha,\beta)}$ gives $t(v+sw)\in\mathcal{O}_{p}^{(\alpha,\beta)}$ for $|s|<\varepsilon$ and $t\in[0,1]$.

Consider the geodesic variation
\[
\sigma_{w}(s,t)
:=
\operatorname{Exp}_{p}^{(\alpha)}
\bigl(t(v+sw)\bigr),
\qquad
|s|<\varepsilon,\quad 0\leq t\leq1,
\]
and let
\[
J_{w}(t)
:=
\left.\frac{\partial}{\partial s}\right|_{s=0}
\sigma_{w}(s,t)
\]
be its variational vector field.
The vector field $J_{w}$ is a Jacobi field along $\gamma_{v}^{(\alpha)}$ satisfying
\[
\left(
\nabla^{(\alpha)}_{\dot{\gamma}_{v}^{(\alpha)}}
\right)^{2}J_{w}
+
R^{(\alpha)}_{\dot{\gamma}_{v}^{(\alpha)}}(J_{w})
=
0.
\]
Its initial data are
\[
J_{w}(0)=0_{p},
\qquad
\left(
\nabla^{(\alpha)}_{\dot{\gamma}_{v}^{(\alpha)}}J_{w}
\right)(0)
=
w.
\]
The chain rule gives
\[
J_{w}(t)
=
\left(
d\operatorname{Exp}_{p}^{(\alpha)}
\right)_{tv}(tw).
\]

Define a vector field along $\overline{\gamma}_{v}^{(\beta)}$ by
\begin{equation}
\overline{J}_{w}(t)
:=
\Phi^{(\alpha,\beta)}_{v,t}
\bigl(J_{w}(t)\bigr).
\label{eq:Jacobi-compare}
\end{equation}
The definition of $\Phi^{(\alpha,\beta)}_{v,t}$ by parallel displacement, together with \eqref{eq:radial-Ralpha}, gives
\begin{align*}
\left(
\overline{\nabla}^{(\beta)}
_{\dot{\overline{\gamma}}_{v}^{(\beta)}}
\overline{J}_{w}
\right)(t)
&=
\Phi^{(\alpha,\beta)}_{v,t}
\left(
\left(
\nabla^{(\alpha)}_{\dot{\gamma}_{v}^{(\alpha)}}J_{w}
\right)(t)
\right),\\
\left(
\overline{\nabla}^{(\beta)}
_{\dot{\overline{\gamma}}_{v}^{(\beta)}}
\right)^{2}
\overline{J}_{w}(t)
&=
\Phi^{(\alpha,\beta)}_{v,t}
\left(
\left(
\nabla^{(\alpha)}_{\dot{\gamma}_{v}^{(\alpha)}}
\right)^{2}
J_{w}(t)
\right),\\
\overline{R}^{(\beta)}
_{\dot{\overline{\gamma}}_{v}^{(\beta)}(t)}
\bigl(\overline{J}_{w}(t)\bigr)
&=
\Phi^{(\alpha,\beta)}_{v,t}
\left(
R^{(\alpha)}_{\dot{\gamma}_{v}^{(\alpha)}(t)}
\bigl(J_{w}(t)\bigr)
\right).
\end{align*}
The Jacobi equation for $J_{w}$ now gives
\[
\left(
\overline{\nabla}^{(\beta)}
_{\dot{\overline{\gamma}}_{v}^{(\beta)}}
\right)^{2}
\overline{J}_{w}
+
\overline{R}^{(\beta)}
_{\dot{\overline{\gamma}}_{v}^{(\beta)}}
\bigl(\overline{J}_{w}\bigr)
=
0.
\]
Thus $\overline{J}_{w}$ is a Jacobi field along $\overline{\gamma}_{v}^{(\beta)}$.
Since $\Phi^{(\alpha,\beta)}_{v,0}=\Phi$, its initial data are
\[
\overline{J}_{w}(0)=0_{\overline{p}},
\qquad
\left(
\overline{\nabla}^{(\beta)}
_{\dot{\overline{\gamma}}_{v}^{(\beta)}}
\overline{J}_{w}
\right)(0)
=
\Phi(w).
\]

Consider the corresponding target geodesic variation
\[
\overline{\sigma}_{w}(s,t)
:=
\overline{\operatorname{Exp}}_{\overline{p}}^{(\beta)}
\bigl(t\Phi(v+sw)\bigr)
=
\varphi^{(\alpha,\beta)}
\bigl(\sigma_{w}(s,t)\bigr).
\]
Its variational vector field is a Jacobi field along $\overline{\gamma}_{v}^{(\beta)}$ with initial data $0_{\overline{p}}$ and $\Phi(w)$.
Uniqueness for the target Jacobi equation gives
\[
\left.\frac{\partial}{\partial s}\right|_{s=0}
\overline{\sigma}_{w}(s,t)
=
\overline{J}_{w}(t).
\]
Differentiating $\overline{\sigma}_{w} =\varphi^{(\alpha,\beta)}\circ\sigma_{w}$ with respect to $s$ at $s=0$ and using \eqref{eq:Jacobi-compare} gives
\[
(d\varphi^{(\alpha,\beta)})_{\gamma_{v}^{(\alpha)}(t)}
\bigl(J_{w}(t)\bigr)
=
\Phi^{(\alpha,\beta)}_{v,t}
\bigl(J_{w}(t)\bigr).
\]

For each fixed $t\in(0,1]$, the map $w\mapsto J_{w}(t)$ is represented by the following diagram:
\[
\begin{array}{ccccc}
T_{p}M
&\xrightarrow{\;t\,\id\;}&
T_{tv}(T_{p}M)
&\xrightarrow{\;
\left(d\operatorname{Exp}_{p}^{(\alpha)}\right)_{tv}
\;}&
T_{\gamma_{v}^{(\alpha)}(t)}M
\\[2mm]
w
&\longmapsto&
tw
&\longmapsto&
J_{w}(t).
\end{array}
\]
Here $T_{tv}(T_{p}M)$ is canonically identified with $T_{p}M$.
The first map is a linear isomorphism because $t\ne0$, and the second map is a linear isomorphism because $\operatorname{Exp}_{p}^{(\alpha)}$ restricts to a diffeomorphism on $\mathcal{O}_{p}^{(\alpha,\beta)}$.
Hence $w\mapsto J_{w}(t)$ is a linear isomorphism, and \eqref{eq:dvarphi-Phi} holds for $t\in(0,1]$.

At $t=0$, the differentials of the exponential maps at their zero vectors are the identity maps.
Therefore,
\[
(d\varphi^{(\alpha,\beta)})_{p}
=
\left(
d\overline{\operatorname{Exp}}_{\overline{p}}^{(\beta)}
\right)_{0_{\overline{p}}}
\circ\Phi\circ
\left(
d\operatorname{Exp}_{p}^{(\alpha)}
\right)_{0_{p}}^{-1}
=
\Phi
=
\Phi^{(\alpha,\beta)}_{v,0}.
\]
\end{proof}

\begin{lemma}
\label{lem:radial-metric}
In Setting~\textup{\ref{set:alpha-beta-CAH}}, with arbitrary $\alpha,\beta\in\mathbb{R}$, suppose that, for every $v\in\mathcal{O}_{p}^{(\alpha,\beta)}$ and every $t\in[0,1]$,
\begin{equation}
\alpha\,\Phi^{(\alpha,\beta)}_{v,t}\circ
K_{\dot{\gamma}_{v}^{(\alpha)}(t)}
=
\beta\,
\overline{K}_{\dot{\overline{\gamma}}_{v}^{(\beta)}(t)}
\circ\Phi^{(\alpha,\beta)}_{v,t}.
\label{eq:radial-scaled-K}
\end{equation}
Then every $\Phi^{(\alpha,\beta)}_{v,t}$ is a linear isometry.
If, in addition, the radial curvature condition \eqref{eq:radial-Ralpha} holds, then $\varphi^{(\alpha,\beta)}$ is a local isometry\textup{:}
\[
\bigl(\varphi^{(\alpha,\beta)}\bigr)^{*}\overline{g}=g.
\]
\end{lemma}

\begin{proof}
Fix $v\in\mathcal{O}_{p}^{(\alpha,\beta)}$.
Choose a basis $\{e_{1},\ldots,e_{n}\}$ of $T_{p}M$, and set
\[
E_{i}(t):=P^{(\alpha)}_{v,t}e_{i},
\qquad
\overline{E}_{i}(t):=
\overline{P}^{(\beta)}_{v,t}\Phi(e_{i}).
\]
By the definition of $\Phi^{(\alpha,\beta)}_{v,t}$,
\[
\overline{E}_{i}(t)
=
\Phi^{(\alpha,\beta)}_{v,t}\bigl(E_{i}(t)\bigr).
\]
For each $t\in[0,1]$, define $A_{i}{}^{k}(t)$ by
\begin{equation}
\alpha
K_{\dot{\gamma}_{v}^{(\alpha)}(t)}
\bigl(E_{i}(t)\bigr)
=
\sum_{k}A_{i}{}^{k}(t)E_{k}(t).
\label{eq:Gram-A}
\end{equation}
Applying \eqref{eq:radial-scaled-K} to $E_{i}(t)$ and using \eqref{eq:Gram-A} gives
\begin{equation}
\beta
\overline{K}_{\dot{\overline{\gamma}}_{v}^{(\beta)}(t)}
\bigl(\overline{E}_{i}(t)\bigr)
=
\sum_{k}A_{i}{}^{k}(t)\overline{E}_{k}(t).
\label{eq:Gram-Abar}
\end{equation}
Thus the two scaled skewness operators have the same coefficient matrix in the corresponding parallel frames.

Set
\[
G_{ij}(t):=g\bigl(E_{i}(t),E_{j}(t)\bigr),
\qquad
\overline{G}_{ij}(t):=
\overline{g}\bigl(\overline{E}_{i}(t),\overline{E}_{j}(t)\bigr).
\]
Since the fields $E_{i}$ are $\nabla^{(\alpha)}$-parallel, \eqref{eq:alpha-C} and \eqref{eq:Gram-A} give
\begin{align}
\frac{d}{dt}G_{ij}(t)
&=
\bigl(
\nabla^{(\alpha)}_{\dot{\gamma}_{v}^{(\alpha)}(t)}g
\bigr)
\bigl(E_{i}(t),E_{j}(t)\bigr)
\notag\\
&=
g\bigl(
\alpha
K_{\dot{\gamma}_{v}^{(\alpha)}(t)}
\bigl(E_{i}(t)\bigr),
E_{j}(t)
\bigr)
\notag\\
&=
\sum_{k}A_{i}{}^{k}(t)G_{kj}(t).
\label{eq:Gram-source}
\end{align}
Similarly, the fields $\overline{E}_{i}$ are $\overline{\nabla}^{(\beta)}$-parallel.
The barred version of \eqref{eq:alpha-C}, together with \eqref{eq:Gram-Abar}, gives
\begin{align}
\frac{d}{dt}\overline{G}_{ij}(t)
&=
\bigl(
\overline{\nabla}^{(\beta)}
_{\dot{\overline{\gamma}}_{v}^{(\beta)}(t)}
\overline{g}
\bigr)
\bigl(\overline{E}_{i}(t),\overline{E}_{j}(t)\bigr)
\notag\\
&=
\overline{g}\bigl(
\beta
\overline{K}_{\dot{\overline{\gamma}}_{v}^{(\beta)}(t)}
\bigl(\overline{E}_{i}(t)\bigr),
\overline{E}_{j}(t)
\bigr)
\notag\\
&=
\sum_{k}A_{i}{}^{k}(t)\overline{G}_{kj}(t).
\label{eq:Gram-target}
\end{align}

At $t=0$, the isometry of $\Phi$ gives
\[
\overline{G}_{ij}(0)
=
\overline{g}\bigl(\Phi(e_{i}),\Phi(e_{j})\bigr)
=
g(e_{i},e_{j})
=
G_{ij}(0).
\]
Let $A(t):=(A_{i}{}^{k}(t))$, $G(t):=(G_{ij}(t))$, and $\overline{G}(t):=(\overline{G}_{ij}(t))$.
Equations~\eqref{eq:Gram-source} and \eqref{eq:Gram-target} give
\[
\dot{G}(t)=A(t)G(t),
\qquad
\dot{\overline{G}}(t)=A(t)\overline{G}(t).
\]
Since $G(0)=\overline{G}(0)$, uniqueness for the matrix-valued linear initial-value problem gives
\[
G(t)=\overline{G}(t)
\qquad
(t\in[0,1]).
\]
Since $\overline{E}_{i}(t) =\Phi^{(\alpha,\beta)}_{v,t}(E_{i}(t))$ and $\{E_{1}(t),\ldots,E_{n}(t)\}$ is a basis, this equality is equivalent to
\[
\left(\Phi^{(\alpha,\beta)}_{v,t}\right)^{*}
\overline{g}_{\overline{\gamma}_{v}^{(\beta)}(t)}
=
g_{\gamma_{v}^{(\alpha)}(t)}.
\]
Thus $\Phi^{(\alpha,\beta)}_{v,t}$ is a linear isometry.

Assume in addition that \eqref{eq:radial-Ralpha} holds.
For any $q\in O^{(\alpha,\beta)}(p)$, take $v\in\mathcal{O}_{p}^{(\alpha,\beta)}$ such that
\[
q
=
\operatorname{Exp}_{p}^{(\alpha)}(v)
=
\gamma_{v}^{(\alpha)}(1).
\]
Lemma~\ref{lem:radial-Jacobi} gives
\[
(d\varphi^{(\alpha,\beta)})_{q}
=
\Phi^{(\alpha,\beta)}_{v,1}.
\]
The right-hand side is a linear isometry by the first part of the proof.
Thus
\[
\bigl(\varphi^{(\alpha,\beta)}\bigr)^{*}\overline{g}=g.
\]
\end{proof}

%%%%%%%%%%%%%%%%%%%%%%%%%%%%
%%%%%%%%%%%%%%%%%%%%%%%%%%%%
\subsection{The local statistical Cartan--Ambrose--Hicks theorem}
%%%%%%%%%%%%%%%%%%%%%%%%%%%%
%%%%%%%%%%%%%%%%%%%%%%%%%%%%

The preceding lemmas identify the differential of the spray map and establish metric preservation under radial comparison conditions.
The following theorem gives necessary and sufficient conditions for the spray map to be a local $(\alpha,\beta)$-statisticmorphism in terms of the intertwining relations for the radial Jacobi operators and the scaled skewness operators in all tangent directions.

\begin{theorem}
\label{thm:alpha-beta-CAH}
In Setting~\textup{\ref{set:alpha-beta-CAH}}, assume that $\alpha\ne0$ and $\beta\ne0$.
Suppose that, for every $v\in\mathcal{O}_{p}^{(\alpha,\beta)}$ and every $t\in[0,1]$, the following two conditions hold\textup{:}
\begin{subequations}
\label{eq:alpha-CAH}
\begin{align}
&\Phi^{(\alpha,\beta)}_{v,t}\circ
R^{(\alpha)}_{\dot{\gamma}_{v}^{(\alpha)}(t)}
=
\overline{R}^{(\beta)}_{\dot{\overline{\gamma}}_{v}^{(\beta)}(t)}
\circ\Phi^{(\alpha,\beta)}_{v,t},
\label{eq:alpha-CAH-R}\\
&\alpha\,\Phi^{(\alpha,\beta)}_{v,t}\circ K_{u}
=
\beta\,
\overline{K}_{\Phi^{(\alpha,\beta)}_{v,t}(u)}
\circ\Phi^{(\alpha,\beta)}_{v,t}
\qquad
\bigl(u\in T_{\gamma_{v}^{(\alpha)}(t)}M\bigr).
\label{eq:alpha-CAH-K}
\end{align}
\end{subequations}
Then $\varphi^{(\alpha,\beta)}$ is a local $(\alpha,\beta)$-statisticmorphism\textup{:}
\begin{equation}
\bigl(\varphi^{(\alpha,\beta)}\bigr)^{*}\overline{g}=g,
\qquad
(\varphi^{(\alpha,\beta)})_{*}\nabla^{(\alpha)}
=
\overline{\nabla}^{(\beta)}.
\label{eq:alpha-CAH-morph}
\end{equation}

Conversely, let $\varphi$ be a local $(\alpha,\beta)$-statisticmorphism satisfying $\varphi(p)=\overline{p}$ and $(d\varphi)_{p}=\Phi$.
After restricting the domains, $\varphi=\varphi^{(\alpha,\beta)}$, and the two conditions in \eqref{eq:alpha-CAH} hold.
\end{theorem}

\begin{proof}
Taking $u=\dot{\gamma}_{v}^{(\alpha)}(t)$ in \eqref{eq:alpha-CAH-K} and using \eqref{eq:radial-tangent} gives \eqref{eq:radial-scaled-K}.
Lemma~\ref{lem:radial-metric} shows that every $\Phi^{(\alpha,\beta)}_{v,t}$ is a linear isometry.

Equation~\eqref{eq:alpha-CAH-R} is precisely the hypothesis \eqref{eq:radial-Ralpha} of Lemma~\ref{lem:radial-Jacobi}.
Hence
\[
(d\varphi^{(\alpha,\beta)})_{\gamma_{v}^{(\alpha)}(t)}
=
\Phi^{(\alpha,\beta)}_{v,t}.
\]
Thus $\varphi^{(\alpha,\beta)}$ is a local isometry.

At each point $q=\gamma_{v}^{(\alpha)}(t)$, equations \eqref{eq:alpha-CAH-K} and \eqref{eq:dvarphi-Phi} give
\[
\alpha\,
(d\varphi^{(\alpha,\beta)})_{q}\circ K_{u}
=
\beta\,
\overline{K}_{(d\varphi^{(\alpha,\beta)})_{q}(u)}
\circ(d\varphi^{(\alpha,\beta)})_{q}
\qquad
(u\in T_{q}M).
\]
Proposition~\ref{prop:stat-morph-equiv} now shows that $\varphi^{(\alpha,\beta)}$ is a local $(\alpha,\beta)$-statisticmorphism.

Conversely, let $\varphi$ be a local $(\alpha,\beta)$-statisticmorphism satisfying $\varphi(p)=\overline{p}$ and $(d\varphi)_{p}=\Phi$.
Shrink the normal neighborhoods so that $\varphi$ is defined on $O^{(\alpha,\beta)}(p)$.
For each $v\in\mathcal{O}_{p}^{(\alpha,\beta)}$, the curve $\varphi\circ\gamma_{v}^{(\alpha)}$ is a $\overline{\nabla}^{(\beta)}$-geodesic with initial point $\overline{p}$ and initial velocity $\Phi(v)$.
Uniqueness of geodesics gives
\[
\varphi\bigl(\gamma_{v}^{(\alpha)}(t)\bigr)
=
\overline{\gamma}_{v}^{(\beta)}(t).
\]
Thus $\varphi=\varphi^{(\alpha,\beta)}$, where $\varphi^{(\alpha,\beta)}$ is defined by \eqref{eq:varphi-alpha-beta}.

Since $\varphi$ intertwines $\nabla^{(\alpha)}$ and $\overline{\nabla}^{(\beta)}$, its differential maps $\nabla^{(\alpha)}$-parallel vector fields along $\gamma_{v}^{(\alpha)}$ to $\overline{\nabla}^{(\beta)}$-parallel vector fields along $\overline{\gamma}_{v}^{(\beta)}$.
Hence
\[
(d\varphi)_{\gamma_{v}^{(\alpha)}(t)}
\circ P_{v,t}^{(\alpha)}
=
\overline{P}_{v,t}^{(\beta)}\circ\Phi.
\]
By the definition of $\Phi^{(\alpha,\beta)}_{v,t}$, this identity is equivalent to
\[
(d\varphi)_{\gamma_{v}^{(\alpha)}(t)}
=
\Phi_{v,t}^{(\alpha,\beta)}.
\]

At every $q\in O^{(\alpha,\beta)}(p)$, the connection-intertwining relation gives
\[
(d\varphi)_{q}\circ R^{(\alpha)}_{z}
=
\overline{R}^{(\beta)}_{(d\varphi)_{q}(z)}
\circ(d\varphi)_{q}
\qquad
(z\in T_{q}M).
\]
Setting $q=\gamma_{v}^{(\alpha)}(t)$ and $z=\dot{\gamma}_{v}^{(\alpha)}(t)$, and using \eqref{eq:radial-tangent}, gives \eqref{eq:alpha-CAH-R}.

Finally, Proposition~\ref{prop:stat-morph-equiv} gives $\alpha\varphi_{*}K=\beta\overline{K}$.
At $q=\gamma_{v}^{(\alpha)}(t)$, substitution of $(d\varphi)_{q}=\Phi_{v,t}^{(\alpha,\beta)}$ gives \eqref{eq:alpha-CAH-K}.
\end{proof}

\begin{remark}
\label{rem:zero-alpha-beta}
Assume that the comparison conditions \eqref{eq:alpha-CAH} hold and that one or both parameters vanish.
Then $\varphi^{(\alpha,\beta)}$ is a local isometry by Lemma~\ref{lem:radial-metric}.

If $\alpha=0$ and $\beta\ne0$, equation~\eqref{eq:alpha-CAH-K} forces $\overline{K}=0$.
If $\alpha\ne0$ and $\beta=0$, it forces $K=0$.
In either case, $\varphi^{(\alpha,\beta)}$ is a statisticmorphism between the original statistical structures if and only if both $K$ and $\overline{K}$ vanish.

If $\alpha=\beta=0$, equation~\eqref{eq:alpha-CAH-K} is vacuous.
The local isometry $\varphi^{(0,0)}$ is a statisticmorphism if and only if
\[
\bigl(\varphi^{(0,0)}\bigr)_{*}K=\overline{K}.
\]
Neither tensor is required to vanish.
This case is treated in Section~\ref{sec:LC-CAH}.
\end{remark}

\begin{corollary}
\label{cor:Hess-alpha-CAH}
In Setting~\textup{\ref{set:alpha-beta-CAH}}, suppose that both statistical manifolds are Hessian and that $\alpha^{2}=\beta^{2}\ne0$.
Then $\varphi^{(\alpha,\beta)}$ is a local $(\alpha,\beta)$-statisticmorphism if and only if \eqref{eq:alpha-CAH-K} holds for every $v\in\mathcal{O}_{p}^{(\alpha,\beta)}$ and every $t\in[0,1]$.
\end{corollary}

\begin{proof}
For Hessian manifolds, \eqref{eq:Ralpha-Rg} gives
\[
R^{(\alpha)}
=\frac{\alpha^{2}-1}{4}[K,K],
\qquad
\overline{R}^{(\beta)}
=\frac{\beta^{2}-1}{4}[\overline{K},\overline{K}].
\]

Assume that \eqref{eq:alpha-CAH-K} holds.
Fix $v\in\mathcal{O}_{p}^{(\alpha,\beta)}$ and $t\in[0,1]$.
Applying \eqref{eq:alpha-CAH-K} twice gives
\[
\begin{aligned}
\alpha^{2}\,
\Phi^{(\alpha,\beta)}_{v,t}\circ[K_{u},K_{w}]
=
\beta^{2}
\bigl[
\overline{K}_{\Phi^{(\alpha,\beta)}_{v,t}(u)},
\overline{K}_{\Phi^{(\alpha,\beta)}_{v,t}(w)}
\bigr]
\circ\Phi^{(\alpha,\beta)}_{v,t}
\end{aligned}
\]
for all $u,w\in T_{\gamma_{v}^{(\alpha)}(t)}M$.
Since $\alpha^{2}=\beta^{2}\ne0$, the curvature identities above yield
\[
\begin{aligned}
\Phi^{(\alpha,\beta)}_{v,t}\circ R^{(\alpha)}(u,w)
=
\overline{R}^{(\beta)}
\bigl(
\Phi^{(\alpha,\beta)}_{v,t}(u),
\Phi^{(\alpha,\beta)}_{v,t}(w)
\bigr)
\circ\Phi^{(\alpha,\beta)}_{v,t}.
\end{aligned}
\]
Thus, we obtain \eqref{eq:alpha-CAH-R}.
Therefore, Theorem~\ref{thm:alpha-beta-CAH} shows that $\varphi^{(\alpha,\beta)}$ is a local $(\alpha,\beta)$-statisticmorphism.

Conversely, the necessity of \eqref{eq:alpha-CAH-K} follows from the converse part of Theorem~\ref{thm:alpha-beta-CAH}.
\end{proof}

\begin{remark}
Suppose that both statistical manifolds are Hessian, and let $\varphi$ be a local $(\alpha,\beta)$-statisticmorphism with $\alpha\ne0$ and $\beta\ne0$.
If $\alpha^{2}\ne\beta^{2}$, then $g$ and $\overline{g}$ are flat on the domain and image of $\varphi$, respectively.

Indeed, \eqref{eq:family-scaling} gives $\varphi_{*}\nabla = \overline{\nabla}^{(\beta/\alpha)}$.
Since $\nabla$ is flat, $\overline{\nabla}^{(\beta/\alpha)}$ is also flat.
Therefore, on the Hessian target, \eqref{eq:RS-RA-K} and \eqref{eq:Ralpha-Rg} give
\[
0
=
\overline{R}^{(\beta/\alpha)}
=
\left(1-\frac{\beta^{2}}{\alpha^{2}}\right)
\overline{R}^{(0)}.
\]
Since $\alpha^{2}\ne\beta^{2}$, we obtain $\overline{R}^{(0)}=0$ on the image of $\varphi$.
The metric $g$ is also flat on its domain, because $\varphi$ is an isometry.
\end{remark}

%%%%%%%%%%%%%%%%%%%%%%%%%%%%
%%%%%%%%%%%%%%%%%%%%%%%%%%%%
\subsection{A criterion using covariant derivatives of the skewness tensor}
%%%%%%%%%%%%%%%%%%%%%%%%%%%%
%%%%%%%%%%%%%%%%%%%%%%%%%%%%

Theorem~\ref{thm:alpha-beta-CAH} requires the relation \eqref{eq:alpha-CAH-K} between $\alpha K$ and $\beta\overline{K}$ along every pair of corresponding radial geodesics.
The following corollary gives an equivalent criterion by replacing this relation with the condition \eqref{eq:base-scaled-K} at the base points and the comparison of radial covariant derivatives in~\eqref{eq:radial-nabla-K}.

For $X,Y\in\mathfrak{X}(M)$, write $\bigl(\nabla^{(\alpha)}_{X}K\bigr)_{Y} \in\Gamma\bigl(\operatorname{End}(TM)\bigr)$ for the endomorphism field defined by $\bigl(\nabla^{(\alpha)}_{X}K\bigr)_{Y}(Z) := \bigl(\nabla^{(\alpha)}_{X}K\bigr)(Y,Z)$.
The analogous notation is used for $\overline{K}$.

\begin{corollary}
\label{cor:CAH-nabla-K}
In Setting~\textup{\ref{set:alpha-beta-CAH}}, assume that $\alpha\ne0$ and $\beta\ne0$.
Then $\varphi^{(\alpha,\beta)}$ is a local $(\alpha,\beta)$-statisticmorphism if and only if the following conditions hold\textup{:}
\begin{enumerate}[label={\rm(\roman*)}]
\item
The condition \eqref{eq:alpha-CAH-R} on the Jacobi operators holds for every $v\in\mathcal{O}_{p}^{(\alpha,\beta)}$ and every $t\in[0,1]$.

\item
For every $u_{p}\in T_{p}M$,
\begin{equation}
\alpha\,\Phi\circ K_{u_{p}}
=
\beta\,\overline{K}_{\Phi(u_{p})}\circ\Phi.
\label{eq:base-scaled-K}
\end{equation}

\item
For every $v\in\mathcal{O}_{p}^{(\alpha,\beta)}$, every $t\in[0,1]$, and every $u_{t}\in T_{\gamma_{v}^{(\alpha)}(t)}M$,
\begin{equation}
\alpha\,\Phi^{(\alpha,\beta)}_{v,t}\circ
\bigl(
\nabla^{(\alpha)}_{\dot{\gamma}_{v}^{(\alpha)}(t)}K
\bigr)_{u_{t}}
=
\beta\,
\bigl(
\overline{\nabla}^{(\beta)}
_{\dot{\overline{\gamma}}_{v}^{(\beta)}(t)}
\overline{K}
\bigr)_{\Phi^{(\alpha,\beta)}_{v,t}(u_{t})}
\circ\Phi^{(\alpha,\beta)}_{v,t}.
\label{eq:radial-nabla-K}
\end{equation}
\end{enumerate}
\end{corollary}

\begin{proof}
For sufficiency, assume that \textup{(i)}--\textup{(iii)} hold.
Fix $v\in\mathcal{O}_{p}^{(\alpha,\beta)}$ and $t\in[0,1]$, and let $u_{t},w_{t}\in T_{\gamma_{v}^{(\alpha)}(t)}M$.
Put
\[
u_{p}
:=
\bigl(P^{(\alpha)}_{v,t}\bigr)^{-1}(u_{t}),
\qquad
w_{p}
:=
\bigl(P^{(\alpha)}_{v,t}\bigr)^{-1}(w_{t}).
\]
For $0\leq s\leq t$, define
\[
U(s):=P^{(\alpha)}_{v,s}(u_{p}),
\qquad
W(s):=P^{(\alpha)}_{v,s}(w_{p}).
\]
These are $\nabla^{(\alpha)}$-parallel vector fields along $\gamma_{v}^{(\alpha)}$ satisfying
\[
U(0)=u_{p},
\qquad
W(0)=w_{p},
\qquad
U(t)=u_{t},
\qquad
W(t)=w_{t}.
\]
Set
\[
\overline{U}(s)
:=
\Phi^{(\alpha,\beta)}_{v,s}\bigl(U(s)\bigr),
\qquad
\overline{W}(s)
:=
\Phi^{(\alpha,\beta)}_{v,s}\bigl(W(s)\bigr).
\]
The definition of $\Phi^{(\alpha,\beta)}_{v,s}$ shows that $\overline{U}$ and $\overline{W}$ are $\overline{\nabla}^{(\beta)}$-parallel along $\overline{\gamma}_{v}^{(\beta)}$.

Define the vector field $\overline{F}$ along $\overline{\gamma}_{v}^{(\beta)}$ by
\[
\overline{F}(s)
:=
\left(
\alpha\,\Phi^{(\alpha,\beta)}_{v,s}\circ K_{U(s)}
-
\beta\,\overline{K}_{\overline{U}(s)}
\circ\Phi^{(\alpha,\beta)}_{v,s}
\right)\bigl(W(s)\bigr).
\]
By \eqref{eq:Phi-alpha-beta}, every vector field $Y$ along $\gamma_{v}^{(\alpha)}$ satisfies
\[
\overline{\nabla}^{(\beta)}
_{\dot{\overline{\gamma}}_{v}^{(\beta)}(s)}
\left(
\Phi^{(\alpha,\beta)}_{v,s}\bigl(Y(s)\bigr)
\right)
=
\Phi^{(\alpha,\beta)}_{v,s}
\left(
\nabla^{(\alpha)}
_{\dot{\gamma}_{v}^{(\alpha)}(s)}Y
\right).
\]
Since $U,W,\overline{U},\overline{W}$ are parallel, condition~\textup{(iii)}, applied to $U(s)$ and evaluated at $W(s)$, gives
\[
\begin{aligned}
\overline{\nabla}^{(\beta)}
_{\dot{\overline{\gamma}}_{v}^{(\beta)}(s)}
\overline{F}
&=
\left(
\alpha\,\Phi^{(\alpha,\beta)}_{v,s}\circ
\bigl(
\nabla^{(\alpha)}_{\dot{\gamma}_{v}^{(\alpha)}(s)}K
\bigr)_{U(s)}
-
\beta\,
\bigl(
\overline{\nabla}^{(\beta)}
_{\dot{\overline{\gamma}}_{v}^{(\beta)}(s)}
\overline{K}
\bigr)_{\overline{U}(s)}
\circ\Phi^{(\alpha,\beta)}_{v,s}
\right)\bigl(W(s)\bigr)
=0.
\end{aligned}
\]
At $s=0$, one has
\[
\Phi^{(\alpha,\beta)}_{v,0}=\Phi,
\qquad
\overline{U}(0)=\Phi(u_{p}),
\qquad
\overline{W}(0)=\Phi(w_{p}).
\]
Condition~\textup{(ii)} shows that the operator in brackets defining $\overline{F}(0)$ vanishes.
Hence $\overline{F}(0)=0$.
Since $\overline{F}$ is $\overline{\nabla}^{(\beta)}$-parallel,
\[
\overline{F}(s)
=
\overline{P}^{(\beta)}_{v,s}
\bigl(\overline{F}(0)\bigr)
=
0
\]
for every $s\in[0,t]$.
Evaluating at $s=t$ gives
\[
\left(
\alpha\,\Phi^{(\alpha,\beta)}_{v,t}\circ K_{u_{t}}
-
\beta\,\overline{K}_{\Phi^{(\alpha,\beta)}_{v,t}(u_{t})}
\circ\Phi^{(\alpha,\beta)}_{v,t}
\right)(w_{t})
=
0.
\]
Since $w_{t}$ is arbitrary,
\[
\alpha\,\Phi^{(\alpha,\beta)}_{v,t}\circ K_{u_{t}}
=
\beta\,\overline{K}_{\Phi^{(\alpha,\beta)}_{v,t}(u_{t})}
\circ\Phi^{(\alpha,\beta)}_{v,t}.
\]
The arbitrariness of $v$, $t$, and $u_{t}$ gives \eqref{eq:alpha-CAH-K}.
Theorem~\ref{thm:alpha-beta-CAH} proves sufficiency.

Conversely, suppose that $\varphi:=\varphi^{(\alpha,\beta)}$ is a local $(\alpha,\beta)$-statisticmorphism.
Theorem~\ref{thm:alpha-beta-CAH} gives \textup{(i)} and \textup{(ii)}.
Since $\varphi$ intertwines the two connections and $(d\varphi)_{p}=\Phi$, its differential satisfies
\[
(d\varphi)_{\gamma_{v}^{(\alpha)}(t)}
=
\Phi^{(\alpha,\beta)}_{v,t}
\]
along every pair of corresponding radial geodesics.
The connection identity and $\alpha\varphi_{*}K=\beta\overline{K}$ give
\[
\alpha\,\varphi_{*}\bigl(\nabla^{(\alpha)}K\bigr)
=
\overline{\nabla}^{(\beta)}
\bigl(\alpha\varphi_{*}K\bigr)
=
\beta\,\overline{\nabla}^{(\beta)}\overline{K}.
\]
Together with \eqref{eq:radial-tangent}, this identity yields \textup{(iii)}.
\end{proof}

%%%%%%%%%%%%%%%%%%%%%%%%%%%%
%%%%%%%%%%%%%%%%%%%%%%%%%%%%
\subsection{Spray maps between dual statistical structures}
%%%%%%%%%%%%%%%%%%%%%%%%%%%%
%%%%%%%%%%%%%%%%%%%%%%%%%%%%

For $\alpha\ne0$, the following corollary applies Theorem~\ref{thm:alpha-beta-CAH} on the same statistical manifold with $\overline{p}=p$, $\beta=-\alpha$, and $\Phi=-\id_{T_{p}M}$.
It gives necessary and sufficient conditions for the spray map determined by these data to preserve $g$ and interchange $\nabla^{(\alpha)}$ and $\nabla^{(-\alpha)}$.

\begin{corollary}
\label{cor:alpha-dual-CAH}
Let $(M,g,\nabla)$ be a statistical manifold, let $p\in M$, and let $\alpha\ne0$.
In Setting~\textup{\ref{set:alpha-beta-CAH}}, take
\[
(\overline{M},\overline{g},\overline{\nabla})
=
(M,g,\nabla),
\qquad
\overline{p}=p,
\qquad
\beta=-\alpha,
\qquad
\Phi=-\id_{T_{p}M}.
\]
For these data, $\overline{\gamma}_{v}^{(-\alpha)} = \gamma_{-v}^{(-\alpha)}$.
Suppose that, for every $v\in\mathcal{O}_{p}^{(\alpha,-\alpha)}$ and every $t\in[0,1]$, the following two conditions hold\textup{:}
\begin{subequations}
\label{eq:alpha-dual-CAH}
\begin{align}
&\Phi^{(\alpha,-\alpha)}_{v,t}\circ
R^{(\alpha)}_{\dot{\gamma}_{v}^{(\alpha)}(t)}
=
R^{(-\alpha)}_{\dot{\gamma}_{-v}^{(-\alpha)}(t)}
\circ\Phi^{(\alpha,-\alpha)}_{v,t},
\label{eq:alpha-dual-R}\\
&\Phi^{(\alpha,-\alpha)}_{v,t}\circ K_{u}
=
-K_{\Phi^{(\alpha,-\alpha)}_{v,t}(u)}
\circ\Phi^{(\alpha,-\alpha)}_{v,t}
\qquad
\bigl(u\in T_{\gamma_{v}^{(\alpha)}(t)}M\bigr).
\label{eq:alpha-dual-K}
\end{align}
\end{subequations}
Then
\begin{equation}
s_{p}^{(\alpha)}
:=
\operatorname{Exp}_{p}^{(-\alpha)}
\circ
\bigl(-\id_{T_{p}M}\bigr)
\circ
\bigl(\operatorname{Exp}_{p}^{(\alpha)}\bigr)^{-1}
\label{eq:alpha-dual-spray}
\end{equation}
is a local isometry satisfying
\[
s_{p}^{(\alpha)}(p)=p,
\qquad
\bigl(ds_{p}^{(\alpha)}\bigr)_{p}
=
-\id_{T_{p}M},
\qquad
\bigl(s_{p}^{(\alpha)}\bigr)_{*}\nabla^{(\eta)}
=
\nabla^{(-\eta)}
\quad
(\eta\in\mathbb{R}).
\]

Conversely, let $s_{p}$ be a local isometry defined on a neighborhood of $p$ such that
\[
s_{p}(p)=p,
\qquad
(ds_{p})_{p}=-\id_{T_{p}M},
\qquad
(s_{p})_{*}\nabla^{(\alpha)}=\nabla^{(-\alpha)}.
\]
After restricting the domains, $s_{p}=s_{p}^{(\alpha)}$, and the two conditions in \eqref{eq:alpha-dual-CAH} hold.
\end{corollary}

\begin{proof}
For the specified data, $\varphi^{(\alpha,-\alpha)}=s_{p}^{(\alpha)}$.
Equation~\eqref{eq:alpha-dual-R} is the condition \eqref{eq:alpha-CAH-R} on the Jacobi operators.
Since $\overline{K}=K$, $\beta=-\alpha$, and $\alpha\ne0$, equation~\eqref{eq:alpha-dual-K} is equivalent to the relation \eqref{eq:alpha-CAH-K} between the scaled skewness operators.
Theorem~\ref{thm:alpha-beta-CAH} shows that $s_{p}^{(\alpha)}$ is a local $(\alpha,-\alpha)$-statisticmorphism and hence a local isometry.
The identities at $p$ follow from the definition of $s_{p}^{(\alpha)}$.
Finally, \eqref{eq:family-scaling} with $\beta=-\alpha$ shows that $s_{p}^{(\alpha)}$ interchanges $\nabla^{(\eta)}$ and $\nabla^{(-\eta)}$ for every $\eta\in\mathbb{R}$.

Conversely, the assumptions on $s_{p}$ say that it is a local $(\alpha,-\alpha)$-statisticmorphism with initial data $s_{p}(p)=p$ and $(ds_{p})_{p}=-\id_{T_{p}M}$.
The converse assertion of Theorem~\ref{thm:alpha-beta-CAH} gives, after restricting the domains, $s_{p}=\varphi^{(\alpha,-\alpha)}=s_{p}^{(\alpha)}$ and the two conditions in \eqref{eq:alpha-CAH}.
For the specified data, these are precisely \eqref{eq:alpha-dual-R} and \eqref{eq:alpha-dual-K}.
\end{proof}

\begin{remark}
Since $K^{*}=-K$, equation~\eqref{eq:alpha-dual-K} can equivalently be written as
\[
\Phi^{(\alpha,-\alpha)}_{v,t}\circ K_{u}
=
K^{*}_{\Phi^{(\alpha,-\alpha)}_{v,t}(u)}
\circ\Phi^{(\alpha,-\alpha)}_{v,t}.
\]
Thus \eqref{eq:alpha-dual-K} states that the linear isomorphisms $\Phi^{(\alpha,-\alpha)}_{v,t}$ intertwine the skewness operators of the original and dual statistical structures.
\end{remark}

%%%%%%%%%%%%%%%%%%%%%%%%%%%%
%%%%%%%%%%%%%%%%%%%%%%%%%%%%
%%%%%%%%%%%%%%%%%%%%%%%%%%%%
\section{The Levi--Civita formulation of the statistical CAH theorem}
\label{sec:LC-CAH}
%%%%%%%%%%%%%%%%%%%%%%%%%%%%
%%%%%%%%%%%%%%%%%%%%%%%%%%%%
%%%%%%%%%%%%%%%%%%%%%%%%%%%%

This section treats Setting~\ref{set:alpha-beta-CAH} with $\alpha=\beta=0$, a case excluded from Theorem~\ref{thm:alpha-beta-CAH}.
In this case, \eqref{eq:alpha-CAH-K} is vacuous, and \eqref{eq:alpha-CAH-R} compares only the Levi--Civita Jacobi operators.
To obtain a statisticmorphism between $(M,g,\nabla)$ and $(\overline{M},\overline{g},\overline{\nabla})$, the intertwining of $K$ with $\overline{K}$ must be imposed separately.

Throughout this section, $\varphi^{(0)}$, $\Phi^{(0)}_{v,t}$, $\mathcal{O}_{p}^{(0)}$, and $O^{(0)}(p)$ abbreviate $\varphi^{(0,0)}$, $\Phi^{(0,0)}_{v,t}$, $\mathcal{O}_{p}^{(0,0)}$, and $O^{(0,0)}(p)$, respectively.

%%%%%%%%%%%%%%%%%%%%%%%%%%%%
%%%%%%%%%%%%%%%%%%%%%%%%%%%%
\subsection{The local statistical Cartan--Ambrose--Hicks theorem}
%%%%%%%%%%%%%%%%%%%%%%%%%%%%
%%%%%%%%%%%%%%%%%%%%%%%%%%%%

For this specialization, the spray map in \eqref{eq:varphi-alpha-beta} is
\begin{equation}
\varphi^{(0)}
:=
\overline{\operatorname{Exp}}_{\overline{p}}^{\,\overline{g}}
\circ\Phi\circ
\bigl(\operatorname{Exp}_{p}^{g}\bigr)^{-1}.
\label{eq:varphi-LC}
\end{equation}
For every $v\in\mathcal{O}_{p}^{(0)}$ and every $t\in[0,1]$, the corresponding radial geodesics are
\[
\gamma_{v}^{(0)}(t)
=
\operatorname{Exp}_{p}^{g}(tv),
\qquad
\overline{\gamma}_{v}^{(0)}(t)
=
\overline{\operatorname{Exp}}_{\overline{p}}^{\,\overline{g}}
\bigl(t\Phi(v)\bigr)
=
\varphi^{(0)}\bigl(\gamma_{v}^{(0)}(t)\bigr).
\]
The Levi--Civita parallel displacements along these geodesics are
\[
P^{(0)}_{v,t}
:=
P^{\nabla^{g}}_{\gamma_{v}^{(0)},0\to t},
\qquad
\overline{P}^{(0)}_{v,t}
:=
P^{\nabla^{\overline{g}}}_{\overline{\gamma}_{v}^{(0)},0\to t}.
\]
The radial Levi--Civita parallel identification
\[
\Phi^{(0)}_{v,t}:
T_{\gamma_{v}^{(0)}(t)}M
\longrightarrow
T_{\overline{\gamma}_{v}^{(0)}(t)}\overline{M}
\]
is defined by
\begin{equation}
\Phi^{(0)}_{v,t}
:=
\overline{P}^{(0)}_{v,t}
\circ\Phi\circ
\bigl(P^{(0)}_{v,t}\bigr)^{-1}.
\label{eq:Phi-LC}
\end{equation}
Since $\Phi$ and the two Levi--Civita parallel displacements are linear isometries, $\Phi^{(0)}_{v,t}$ is a linear isometry.

\begin{theorem}
\label{thm:LC-CAH}
In Setting~\textup{\ref{set:alpha-beta-CAH}} with $\alpha=\beta=0$, suppose that, for every $v\in\mathcal{O}_{p}^{(0)}$ and every $t\in[0,1]$, the following two conditions hold\textup{:}
\begin{subequations}
\label{eq:LC-CAH}
\begin{align}
&\Phi^{(0)}_{v,t}\circ
R^{g}_{\dot{\gamma}_{v}^{(0)}(t)}
=
R^{\overline{g}}_{\dot{\overline{\gamma}}_{v}^{(0)}(t)}
\circ\Phi^{(0)}_{v,t},
\label{eq:LC-CAH-Rg}\\
&\Phi^{(0)}_{v,t}\circ K_{u}
=
\overline{K}_{\Phi^{(0)}_{v,t}(u)}
\circ\Phi^{(0)}_{v,t}
\qquad
\bigl(u\in T_{\gamma_{v}^{(0)}(t)}M\bigr).
\label{eq:LC-CAH-K}
\end{align}
\end{subequations}
Then $\varphi^{(0)}$ is a local statisticmorphism\textup{:}
\begin{equation}
(\varphi^{(0)})^{*}\overline{g}=g,
\qquad
(\varphi^{(0)})_{*}\nabla=\overline{\nabla}.
\label{eq:LC-CAH-morph}
\end{equation}

Conversely, let $\varphi$ be a local statisticmorphism satisfying $\varphi(p)=\overline{p}$ and $(d\varphi)_{p}=\Phi$.
After restricting the domains, $\varphi=\varphi^{(0)}$, and the two conditions in \eqref{eq:LC-CAH} hold.
\end{theorem}

\begin{proof}
Assume \eqref{eq:LC-CAH}.
Equation~\eqref{eq:LC-CAH-Rg} is the hypothesis of Lemma~\ref{lem:radial-Jacobi} with $\alpha=\beta=0$.
Hence
\[
(d\varphi^{(0)})_{\gamma_{v}^{(0)}(t)}
=
\Phi^{(0)}_{v,t}.
\]
Since every $\Phi^{(0)}_{v,t}$ is a linear isometry, $\varphi^{(0)}$ is a local isometry.

At each point $q=\gamma_{v}^{(0)}(t)$, equations~\eqref{eq:LC-CAH-K} and \eqref{eq:dvarphi-Phi} give
\[
(d\varphi^{(0)})_{q}\circ K_{u}
=
\overline{K}_{(d\varphi^{(0)})_{q}(u)}
\circ(d\varphi^{(0)})_{q}
\qquad
(u\in T_{q}M).
\]
Proposition~\ref{prop:stat-morph-equiv}, with $\alpha=\beta=1$, shows that $\varphi^{(0)}$ is a local statisticmorphism.

Conversely, let $\varphi$ be a local statisticmorphism satisfying $\varphi(p)=\overline{p}$ and $(d\varphi)_{p}=\Phi$.
Since $\varphi$ is a local isometry, it intertwines the Levi--Civita connections.
Uniqueness of the corresponding radial geodesics gives $\varphi=\varphi^{(0)}$ after restricting the domains.
Moreover, preservation of Levi--Civita parallel displacement and $(d\varphi)_{p}=\Phi$ give
\[
(d\varphi)_{\gamma_{v}^{(0)}(t)}\circ P^{(0)}_{v,t}
=
\overline{P}^{(0)}_{v,t}\circ\Phi.
\]
By \eqref{eq:Phi-LC}, this identity gives $(d\varphi)_{\gamma_{v}^{(0)}(t)}=\Phi^{(0)}_{v,t}$.
The Levi--Civita connection-intertwining relation gives
\[
(d\varphi)_{q}\circ R^{g}_{z}
=
R^{\overline{g}}_{(d\varphi)_{q}(z)}
\circ(d\varphi)_{q}
\qquad
(z\in T_{q}M).
\]
Taking $q=\gamma_{v}^{(0)}(t)$ and $z=\dot{\gamma}_{v}^{(0)}(t)$ gives \eqref{eq:LC-CAH-Rg}.
Proposition~\ref{prop:stat-morph-equiv} gives $\varphi_{*}K=\overline{K}$, which yields \eqref{eq:LC-CAH-K} under the same identification of $d\varphi$ with $\Phi^{(0)}_{v,t}$.
\end{proof}

\begin{remark}
Theorem~\ref{thm:LC-CAH} recovers \cite[Theorem~3.2]{Satoh2019} and refines its curvature hypothesis by replacing the comparison of the full Riemannian curvature tensors with the comparison \eqref{eq:LC-CAH-Rg} of their Jacobi operators along the corresponding radial geodesics.
\end{remark}

%%%%%%%%%%%%%%%%%%%%%%%%%%%%
%%%%%%%%%%%%%%%%%%%%%%%%%%%%
\subsection{Criteria using statistical curvature or the curvature of the statistical connection}
%%%%%%%%%%%%%%%%%%%%%%%%%%%%
%%%%%%%%%%%%%%%%%%%%%%%%%%%%

The following theorem expresses the curvature criterion of Theorem~\ref{thm:LC-CAH} in terms of $R^{\mathsf{S}}$ or $R$, while retaining the relation \eqref{eq:LC-CAH-K} between the skewness operators.

On $\overline{M}$, write
\[
\overline{R}:=\overline{R}^{(1)},
\qquad
\overline{R}^{*}:=\overline{R}^{(-1)},
\qquad
\overline{R}^{\mathsf{S}}
:=
\frac{1}{2}
\bigl(\overline{R}+\overline{R}^{*}\bigr).
\]

\begin{theorem}
\label{thm:LC-CAH-curv}
In Setting~\textup{\ref{set:alpha-beta-CAH}} with $\alpha=\beta=0$, the following conditions are equivalent.
\begin{enumerate}[label={\rm(\roman*)}]
\item
The map $\varphi^{(0)}$ defined by \eqref{eq:varphi-LC} is a local statisticmorphism.

\item
The relation \eqref{eq:LC-CAH-Rg} between the Jacobi operators of the Riemannian curvature tensors and the relation \eqref{eq:LC-CAH-K} between the skewness operators hold.

\item
The relation \eqref{eq:LC-CAH-K} between the skewness operators holds.
In addition, for every $v\in\mathcal{O}_{p}^{(0)}$ and every $t\in[0,1]$,
\begin{equation}
\Phi^{(0)}_{v,t}\circ
R^{\mathsf{S}}_{\dot{\gamma}_{v}^{(0)}(t)}
=
\overline{R}^{\mathsf{S}}_{\dot{\overline{\gamma}}_{v}^{(0)}(t)}
\circ\Phi^{(0)}_{v,t}.
\label{eq:LC-CAH-RS}
\end{equation}

\item
The relation \eqref{eq:LC-CAH-K} between the skewness operators holds.
In addition, for every $v\in\mathcal{O}_{p}^{(0)}$, every $t\in[0,1]$, and every $u\in T_{\gamma_{v}^{(0)}(t)}M$,
\begin{equation}
\Phi^{(0)}_{v,t}\circ
R\bigl(u,\dot{\gamma}_{v}^{(0)}(t)\bigr)
=
\overline{R}\bigl(
\Phi^{(0)}_{v,t}(u),
\dot{\overline{\gamma}}_{v}^{(0)}(t)
\bigr)
\circ\Phi^{(0)}_{v,t}.
\label{eq:LC-CAH-mixed-R}
\end{equation}
\end{enumerate}
\end{theorem}

\begin{proof}
It suffices to prove
\[
\textup{(i)}
\Longrightarrow
\textup{(iv)}
\Longrightarrow
\textup{(iii)}
\Longrightarrow
\textup{(ii)}
\Longrightarrow
\textup{(i)}.
\]

Assume \textup{(i)}.
Since $\varphi^{(0)}$ intertwines $\nabla$ and $\overline{\nabla}$, its differential intertwines their curvature endomorphisms:
\[
(d\varphi^{(0)})_{q}\circ R(u,z)
=
\overline{R}\bigl(
(d\varphi^{(0)})_{q}(u),
(d\varphi^{(0)})_{q}(z)
\bigr)
\circ(d\varphi^{(0)})_{q}
\]
for every $q\in O^{(0)}(p)$ and all $u,z\in T_{q}M$.
Since $\varphi^{(0)}$ is a local isometry, its differential intertwines the Levi--Civita parallel displacements.
The identity $(d\varphi^{(0)})_{p}=\Phi$ gives
\[
(d\varphi^{(0)})_{\gamma_{v}^{(0)}(t)}
=
\overline{P}^{(0)}_{v,t}
\circ\Phi\circ
\bigl(P^{(0)}_{v,t}\bigr)^{-1}
=
\Phi^{(0)}_{v,t}.
\]
Substituting $q=\gamma_{v}^{(0)}(t)$ and $z=\dot{\gamma}_{v}^{(0)}(t)$ into the curvature-intertwining relation, and using \eqref{eq:radial-tangent} with $\alpha=\beta=0$, gives \eqref{eq:LC-CAH-mixed-R}.
Proposition~\ref{prop:stat-morph-equiv} gives $(\varphi^{(0)})_{*}K=\overline{K}$.
Together with $(d\varphi^{(0)})_{\gamma_{v}^{(0)}(t)}=\Phi^{(0)}_{v,t}$, this gives \eqref{eq:LC-CAH-K}.
Thus \textup{(iv)} holds.

Assume \textup{(iv)}.
Fix $v\in\mathcal{O}_{p}^{(0)}$, $t\in[0,1]$, and $u\in T_{\gamma_{v}^{(0)}(t)}M$.
Taking the $g$-adjoint of the source curvature endomorphism and the $\overline{g}$-adjoint of the target curvature endomorphism in \eqref{eq:LC-CAH-mixed-R}, and using the isometry of $\Phi^{(0)}_{v,t}$ together with \eqref{eq:dual-curv}, gives
\[
\Phi^{(0)}_{v,t}\circ
R^{*}\bigl(u,\dot{\gamma}_{v}^{(0)}(t)\bigr)
=
\overline{R}^{*}\bigl(
\Phi^{(0)}_{v,t}(u),
\dot{\overline{\gamma}}_{v}^{(0)}(t)
\bigr)
\circ\Phi^{(0)}_{v,t}.
\]
Averaging this identity with \eqref{eq:LC-CAH-mixed-R} gives
\[
\Phi^{(0)}_{v,t}\circ
R^{\mathsf{S}}\bigl(u,\dot{\gamma}_{v}^{(0)}(t)\bigr)
=
\overline{R}^{\mathsf{S}}\bigl(
\Phi^{(0)}_{v,t}(u),
\dot{\overline{\gamma}}_{v}^{(0)}(t)
\bigr)
\circ\Phi^{(0)}_{v,t}.
\]
Applying both sides to $\dot{\gamma}_{v}^{(0)}(t)$ and using \eqref{eq:radial-tangent} with $\alpha=\beta=0$ gives \eqref{eq:LC-CAH-RS}.
Together with \eqref{eq:LC-CAH-K}, this proves \textup{(iii)}.

Assume \textup{(iii)}.
Fix $v\in\mathcal{O}_{p}^{(0)}$, $t\in[0,1]$, and $u\in T_{\gamma_{v}^{(0)}(t)}M$.
Applying \eqref{eq:LC-CAH-K} with $u$ and with $\dot{\gamma}_{v}^{(0)}(t)$, and using \eqref{eq:radial-tangent} with $\alpha=\beta=0$, gives
\[
\Phi^{(0)}_{v,t}\circ
[K_{u},K_{\dot{\gamma}_{v}^{(0)}(t)}]
=
[\overline{K}_{\Phi^{(0)}_{v,t}(u)},
 \overline{K}_{\dot{\overline{\gamma}}_{v}^{(0)}(t)}]
\circ\Phi^{(0)}_{v,t}.
\]
Since $u$ is arbitrary, applying this identity to $\dot{\gamma}_{v}^{(0)}(t)$ and using \eqref{eq:radial-tangent} with $\alpha=\beta=0$ gives
\[
\Phi^{(0)}_{v,t}\circ
[K,K]_{\dot{\gamma}_{v}^{(0)}(t)}
=
[\overline{K},\overline{K}]_{\dot{\overline{\gamma}}_{v}^{(0)}(t)}
\circ\Phi^{(0)}_{v,t}.
\]
Equation~\eqref{eq:RS-RA-K} gives
\[
R^{g}_{\dot{\gamma}_{v}^{(0)}(t)}
=
R^{\mathsf{S}}_{\dot{\gamma}_{v}^{(0)}(t)}
-\frac{1}{4}[K,K]_{\dot{\gamma}_{v}^{(0)}(t)}
\]
and
\[
R^{\overline{g}}_{\dot{\overline{\gamma}}_{v}^{(0)}(t)}
=
\overline{R}^{\mathsf{S}}_{\dot{\overline{\gamma}}_{v}^{(0)}(t)}
-\frac{1}{4}
[\overline{K},\overline{K}]_{\dot{\overline{\gamma}}_{v}^{(0)}(t)}.
\]
Combining these identities with \eqref{eq:LC-CAH-RS} gives \eqref{eq:LC-CAH-Rg}.
Together with \eqref{eq:LC-CAH-K}, this proves \textup{(ii)}.

Assume \textup{(ii)}.
Theorem~\ref{thm:LC-CAH} proves \textup{(i)}.
\end{proof}

\begin{remark}
Corollary~3.4 of \cite{Satoh2019} compares the full tensors $R$, $K$, and $\nabla K$ along Levi--Civita geodesics using Levi--Civita parallel displacement.
Theorem~\ref{thm:LC-CAH-curv} refines this criterion: with the same geodesics and parallel identifications, condition~\textup{(iv)} requires only the comparison of $R(u,\dot{\gamma}_{v}^{(0)}(t))w$ and $K(u,w)$ for arbitrary tangent vectors $u,w$ at each point of the geodesics, without a separate comparison of $\nabla K$.
\end{remark}

\begin{corollary}
\label{cor:LC-CAH-sec-curv}
In the setting of Theorem~\textup{\ref{thm:LC-CAH-curv}}, assume that $\dim M\geq2$.
Then condition \textup{(iii)} of that theorem is equivalent to the following two conditions.
\begin{enumerate}[label={\rm(\alph*)}]
\item
The relation \eqref{eq:LC-CAH-K} between the skewness operators holds.

\item
For every $v\in\mathcal{O}_{p}^{(0)}\setminus\{0\}$, every $t\in[0,1]$, and every two-plane $\Pi\subset T_{\gamma_{v}^{(0)}(t)}M$ containing $\dot{\gamma}_{v}^{(0)}(t)$, the statistical sectional curvatures satisfy
\[
\operatorname{Sec}_{R^{\mathsf{S}}}(\Pi)
=
\operatorname{Sec}_{\overline{R}^{\mathsf{S}}}
\bigl(\Phi^{(0)}_{v,t}(\Pi)\bigr).
\]
\end{enumerate}
\end{corollary}

\begin{proof}
Assume condition \textup{(iii)} of Theorem~\ref{thm:LC-CAH-curv}.
Equation~\eqref{eq:LC-CAH-K} gives condition \textup{(a)}.
To prove condition \textup{(b)}, fix $v\in\mathcal{O}_{p}^{(0)}\setminus\{0\}$, $t\in[0,1]$, and a two-plane $\Pi$ containing $\dot{\gamma}_{v}^{(0)}(t)$.
Choose $x\in T_{\gamma_{v}^{(0)}(t)}M$ such that $\bigl\{x,\dot{\gamma}_{v}^{(0)}(t)\bigr\}$ is a basis of $\Pi$.
By \eqref{eq:radial-tangent} with $\alpha=\beta=0$, $\bigl\{ \Phi^{(0)}_{v,t}(x), \dot{\overline{\gamma}}_{v}^{(0)}(t) \bigr\}$ is a basis of $\Phi^{(0)}_{v,t}(\Pi)$.
Applying \eqref{eq:LC-CAH-RS} to $x$ and taking the $\overline{g}$-inner product with $\Phi^{(0)}_{v,t}(x)$ gives
\[
g\bigl(
R^{\mathsf{S}}_{\dot{\gamma}_{v}^{(0)}(t)}(x),x
\bigr)
=
\overline{g}\bigl(
\overline{R}^{\mathsf{S}}_{\dot{\overline{\gamma}}_{v}^{(0)}(t)}
\bigl(\Phi^{(0)}_{v,t}(x)\bigr),
\Phi^{(0)}_{v,t}(x)
\bigr).
\]
Moreover, the isometry of $\Phi^{(0)}_{v,t}$ and \eqref{eq:radial-tangent} give
\begin{align*}
&\overline{g}\bigl(
\Phi^{(0)}_{v,t}(x),\Phi^{(0)}_{v,t}(x)
\bigr)
\overline{g}\bigl(
\dot{\overline{\gamma}}_{v}^{(0)}(t),
\dot{\overline{\gamma}}_{v}^{(0)}(t)
\bigr)
-
\overline{g}\bigl(
\Phi^{(0)}_{v,t}(x),
\dot{\overline{\gamma}}_{v}^{(0)}(t)
\bigr)^{2}
\\
&\qquad=
g(x,x)
g\bigl(
\dot{\gamma}_{v}^{(0)}(t),
\dot{\gamma}_{v}^{(0)}(t)
\bigr)
-
g\bigl(
x,\dot{\gamma}_{v}^{(0)}(t)
\bigr)^{2}.
\end{align*}
Thus the numerators and denominators in the definitions of the two sectional curvatures are equal.
Hence condition \textup{(b)} holds.

Conversely, assume conditions \textup{(a)} and \textup{(b)}.
Fix $v\in\mathcal{O}_{p}^{(0)}\setminus\{0\}$ and $t\in[0,1]$.
Since $\dot{\gamma}_{v}^{(0)}$ is Levi--Civita parallel and $\dot{\gamma}_{v}^{(0)}(0)=v\ne0$, one has $\dot{\gamma}_{v}^{(0)}(t)\ne0$.
Moreover, \eqref{eq:radial-tangent} with $\alpha=\beta=0$ gives
\[
\Phi^{(0)}_{v,t}
\bigl(\dot{\gamma}_{v}^{(0)}(t)\bigr)
=
\dot{\overline{\gamma}}_{v}^{(0)}(t).
\]

Let $x\in T_{\gamma_{v}^{(0)}(t)}M$ be linearly independent of $\dot{\gamma}_{v}^{(0)}(t)$.
Applying condition \textup{(b)} to the two-plane $\operatorname{span} \bigl\{x,\dot{\gamma}_{v}^{(0)}(t)\bigr\}$ and using the definition of statistical sectional curvature gives
\begin{equation}
g\bigl(
R^{\mathsf{S}}_{\dot{\gamma}_{v}^{(0)}(t)}(x),x
\bigr)
=
\overline{g}\bigl(
\overline{R}^{\mathsf{S}}_{\dot{\overline{\gamma}}_{v}^{(0)}(t)}
\bigl(\Phi^{(0)}_{v,t}(x)\bigr),
\Phi^{(0)}_{v,t}(x)
\bigr),
\label{eq:radial-RS-quad}
\end{equation}
because the isometry of $\Phi^{(0)}_{v,t}$ makes the denominators in the two sectional curvature formulas equal.
Equation~\eqref{eq:radial-RS-quad} also holds when $x$ is proportional to $\dot{\gamma}_{v}^{(0)}(t)$, since the two Jacobi operators annihilate the corresponding radial tangent vectors.
Thus it holds for every $x\in T_{\gamma_{v}^{(0)}(t)}M$.

Since $R^{\mathsf{S}}$ and $\overline{R}^{\mathsf{S}}$ are curvature-like, their Jacobi operators are self-adjoint.
Applying \eqref{eq:radial-RS-quad} to $x+z$ and subtracting its instances for $x$ and $z$ gives
\[
2g\bigl(
R^{\mathsf{S}}_{\dot{\gamma}_{v}^{(0)}(t)}(x),z
\bigr)
=
2\overline{g}\bigl(
\overline{R}^{\mathsf{S}}_{\dot{\overline{\gamma}}_{v}^{(0)}(t)}
\bigl(\Phi^{(0)}_{v,t}(x)\bigr),
\Phi^{(0)}_{v,t}(z)
\bigr)
\]
for all $x,z\in T_{\gamma_{v}^{(0)}(t)}M$.
Dividing by $2$ and using the isometry of $\Phi^{(0)}_{v,t}$ gives
\[
\overline{g}\Bigl(
\Phi^{(0)}_{v,t}\bigl(
R^{\mathsf{S}}_{\dot{\gamma}_{v}^{(0)}(t)}(x)
\bigr)
-
\overline{R}^{\mathsf{S}}_{\dot{\overline{\gamma}}_{v}^{(0)}(t)}
\bigl(\Phi^{(0)}_{v,t}(x)\bigr),
\Phi^{(0)}_{v,t}(z)
\Bigr)
=0.
\]
As $z$ ranges over $T_{\gamma_{v}^{(0)}(t)}M$, $\Phi^{(0)}_{v,t}(z)$ ranges over $T_{\overline{\gamma}_{v}^{(0)}(t)}\overline{M}$.
The nondegeneracy of $\overline{g}$ gives
\[
\Phi^{(0)}_{v,t}\circ
R^{\mathsf{S}}_{\dot{\gamma}_{v}^{(0)}(t)}
=
\overline{R}^{\mathsf{S}}_{\dot{\overline{\gamma}}_{v}^{(0)}(t)}
\circ\Phi^{(0)}_{v,t}.
\]
This is \eqref{eq:LC-CAH-RS}.

For $v=0$, both radial tangent vectors vanish, so both Jacobi operators vanish.
Together with condition \textup{(a)}, this proves condition \textup{(iii)}.
\end{proof}

\begin{corollary}
\label{cor:LC-CAH-conj-Hess}
In the setting of Theorem~\textup{\ref{thm:LC-CAH-curv}}, the following assertions hold.
\begin{enumerate}[label=\textup{(\arabic*)}]
\item
If both statistical manifolds are conjugate symmetric, then $\varphi^{(0)}$ is a local statisticmorphism if and only if the relation \eqref{eq:LC-CAH-K} between the skewness operators holds and, for every $v\in\mathcal{O}_{p}^{(0)}$ and every $t\in[0,1]$,
\[
\Phi^{(0)}_{v,t}\circ
R_{\dot{\gamma}_{v}^{(0)}(t)}
=
\overline{R}_{\dot{\overline{\gamma}}_{v}^{(0)}(t)}
\circ\Phi^{(0)}_{v,t}.
\]

\item
If both statistical manifolds are Hessian, then $\varphi^{(0)}$ is a local statisticmorphism if and only if the relation \eqref{eq:LC-CAH-K} between the skewness operators holds.
\end{enumerate}
\end{corollary}

\begin{proof}
Assume that both statistical manifolds are conjugate symmetric.
Then
\[
R=R^{*}=R^{\mathsf{S}},
\qquad
\overline{R}
=
\overline{R}^{*}
=
\overline{R}^{\mathsf{S}}.
\]
Consequently, condition \textup{(iii)} of Theorem~\ref{thm:LC-CAH-curv} is precisely the conjunction of the two relations in assertion \textup{(1)}.
The equivalence of conditions \textup{(i)} and \textup{(iii)} in that theorem proves assertion \textup{(1)}.

In the Hessian case, $R=R^{*}=R^{\mathsf{S}}=0$ and $\overline{R} = \overline{R}^{*} = \overline{R}^{\mathsf{S}}=0$.
Hence the relation between the Jacobi operators in condition \textup{(iii)} is automatic, and that condition reduces to \eqref{eq:LC-CAH-K}.
The equivalence of conditions \textup{(i)} and \textup{(iii)} proves assertion \textup{(2)}.
\end{proof}

%%%%%%%%%%%%%%%%%%%%%%%%%%%%
%%%%%%%%%%%%%%%%%%%%%%%%%%%%
\subsection{Parallel Riemannian curvature and skewness tensors}
%%%%%%%%%%%%%%%%%%%%%%%%%%%%
%%%%%%%%%%%%%%%%%%%%%%%%%%%%

Theorem~\ref{thm:LC-CAH} requires relations between the Jacobi operators and between the skewness operators along corresponding radial Levi--Civita geodesics.
If $R^{g}$ and $K$ are $\nabla^{g}$-parallel and $R^{\overline{g}}$ and $\overline{K}$ are $\nabla^{\overline{g}}$-parallel, these relations hold along the geodesics whenever $\Phi$ intertwines the corresponding operators at the base points.
The following corollary recovers \cite[Corollary~3.5]{Satoh2019}.

\begin{corollary}
\label{cor:LC-parallel-CAH}
In Setting~\textup{\ref{set:alpha-beta-CAH}} with $\alpha=\beta=0$, assume that
\[
\nabla^{g}R^{g}=0,
\qquad
\nabla^{g}K=0,
\qquad
\nabla^{\overline{g}}R^{\overline{g}}=0,
\qquad
\nabla^{\overline{g}}\overline{K}=0.
\]
Suppose that the base-point isometry $\Phi:T_{p}M\to T_{\overline{p}}\overline{M}$ satisfies
\[
\Phi\circ R^{g}_{v}
=
R^{\overline{g}}_{\Phi(v)}\circ\Phi,
\qquad
\Phi\circ K_{u}
=
\overline{K}_{\Phi(u)}\circ\Phi
\qquad
(u,v\in T_{p}M).
\]
Then, after restricting the normal neighborhoods, $\varphi^{(0)}$ is a local statisticmorphism.
Moreover, any local statisticmorphism $\varphi$ satisfying
\[
\varphi(p)=\overline{p},
\qquad
(d\varphi)_{p}=\Phi
\]
agrees with $\varphi^{(0)}$ on a neighborhood of $p$.
\end{corollary}

\begin{proof}
Fix $v\in\mathcal{O}_{p}^{(0)}$ and $t\in[0,1]$.
Since $R^{g}$ and $K$ are $\nabla^{g}$-parallel and $\dot{\gamma}_{v}^{(0)}(t)=P^{(0)}_{v,t}(v)$, one has
\[
R^{g}_{\dot{\gamma}_{v}^{(0)}(t)}\circ P^{(0)}_{v,t}
=
P^{(0)}_{v,t}\circ R^{g}_{v}
\]
and
\[
K_{P^{(0)}_{v,t}(w)}\circ P^{(0)}_{v,t}
=
P^{(0)}_{v,t}\circ K_{w}
\qquad
(w\in T_{p}M).
\]
The analogous identities hold for $R^{\overline{g}}$ and $\overline{K}$ with $\overline{P}^{(0)}_{v,t}$ and initial vectors $\Phi(v)$ and $\Phi(w)$.
Since $P^{(0)}_{v,t}$ is a linear isomorphism, every $u\in T_{\gamma_{v}^{(0)}(t)}M$ has the form $u=P^{(0)}_{v,t}(w)$ for some $w\in T_{p}M$.
Combining these transport identities with the two base-point relations and \eqref{eq:Phi-LC} gives
\[
\Phi^{(0)}_{v,t}\circ
R^{g}_{\dot{\gamma}_{v}^{(0)}(t)}
=
R^{\overline{g}}_{\dot{\overline{\gamma}}_{v}^{(0)}(t)}
\circ\Phi^{(0)}_{v,t}
\]
and
\[
\Phi^{(0)}_{v,t}\circ K_{u}
=
\overline{K}_{\Phi^{(0)}_{v,t}(u)}
\circ\Phi^{(0)}_{v,t}
\qquad
\bigl(u\in T_{\gamma_{v}^{(0)}(t)}M\bigr).
\]
Thus the two conditions in \eqref{eq:LC-CAH} hold.
Theorem~\ref{thm:LC-CAH} shows that $\varphi^{(0)}$ is a local statisticmorphism with the prescribed initial data.

Let $\psi$ be another local statisticmorphism satisfying
\[
\psi(p)=\overline{p},
\qquad
(d\psi)_{p}=\Phi.
\]
The converse assertion in Theorem~\ref{thm:LC-CAH} shows that $\psi=\varphi^{(0)}$ on a neighborhood of $p$.
\end{proof}

%%%%%%%%%%%%%%%%%%%%%%%%%%%%
%%%%%%%%%%%%%%%%%%%%%%%%%%%%
\subsection{Levi--Civita spray maps between dual statistical structures}
\label{subsec:LC-dual-spray}
%%%%%%%%%%%%%%%%%%%%%%%%%%%%
%%%%%%%%%%%%%%%%%%%%%%%%%%%%

Theorem~\ref{thm:LC-CAH-curv} admits the following specialization to the dual statistical structure.

\begin{corollary}
\label{cor:LC-dual-CAH}
Let $(M,g,\nabla)$ be a statistical manifold, and let $p\in M$.
In Setting~\textup{\ref{set:alpha-beta-CAH}} with $\alpha=\beta=0$, take
\[
(\overline{M},\overline{g},\overline{\nabla})
=
(M,g,\nabla^{*}),
\qquad
\overline{p}=p,
\qquad
\Phi=-\id_{T_{p}M}.
\]
For these data, $\overline{\gamma}_{v}^{(0)}=\gamma_{-v}^{(0)}$ for every $v\in\mathcal{O}_{p}^{(0)}$.
Moreover, for every such $v$ and every $t\in[0,1]$,
\begin{equation}
\Phi_{v,t}^{(0)}
=
P_{-v,t}^{(0)}
\circ
\bigl(-\id_{T_{p}M}\bigr)
\circ
\bigl(P_{v,t}^{(0)}\bigr)^{-1}.
\label{eq:LC-dual-Phi}
\end{equation}
The corresponding spray map is
\begin{equation}
s_{p}^{(0)}
:=
\varphi^{(0)}
=
\operatorname{Exp}_{p}^{g}
\circ
\bigl(-\id_{T_{p}M}\bigr)
\circ
\bigl(\operatorname{Exp}_{p}^{g}\bigr)^{-1}.
\label{eq:LC-dual-spray}
\end{equation}
The following conditions are equivalent.
\begin{enumerate}[label={\rm(\roman*)}]
\item
The map $s_{p}^{(0)}$ in \eqref{eq:LC-dual-spray} is a local statisticmorphism from $(M,g,\nabla)$ to $(M,g,\nabla^{*})$.

\item
For every $v\in\mathcal{O}_{p}^{(0)}$ and every $t\in[0,1]$,
\begin{equation}
\Phi_{v,t}^{(0)}
\circ
R^{g}_{\dot{\gamma}_{v}^{(0)}(t)}
=
R^{g}_{\dot{\gamma}_{-v}^{(0)}(t)}
\circ
\Phi_{v,t}^{(0)}.
\label{eq:LC-dual-Rg}
\end{equation}
In addition, for every $u\in T_{\gamma_{v}^{(0)}(t)}M$,
\begin{equation}
\Phi_{v,t}^{(0)}
\circ K_{u}
=
-K_{\Phi_{v,t}^{(0)}(u)}
\circ
\Phi_{v,t}^{(0)}.
\label{eq:LC-dual-K}
\end{equation}

\item
The relation \eqref{eq:LC-dual-K} between the skewness operators holds.
In addition, for every $v\in\mathcal{O}_{p}^{(0)}$ and every $t\in[0,1]$,
\begin{equation}
\Phi_{v,t}^{(0)}
\circ
R^{\mathsf{S}}_{\dot{\gamma}_{v}^{(0)}(t)}
=
R^{\mathsf{S}}_{\dot{\gamma}_{-v}^{(0)}(t)}
\circ
\Phi_{v,t}^{(0)}.
\label{eq:LC-dual-RS}
\end{equation}

\item
The relation \eqref{eq:LC-dual-K} between the skewness operators holds.
In addition, for every $v\in\mathcal{O}_{p}^{(0)}$, every $t\in[0,1]$, and every $u\in T_{\gamma_{v}^{(0)}(t)}M$,
\begin{equation}
\Phi_{v,t}^{(0)}
\circ
R\bigl(u,\dot{\gamma}_{v}^{(0)}(t)\bigr)
=
R^{*}\bigl(
\Phi_{v,t}^{(0)}(u),
\dot{\gamma}_{-v}^{(0)}(t)
\bigr)
\circ
\Phi_{v,t}^{(0)}.
\label{eq:LC-dual-mixed-R}
\end{equation}
\end{enumerate}

By construction, $s_{p}^{(0)}(p)=p$ and $(ds_{p}^{(0)})_{p}=-\id_{T_{p}M}$.
Under the equivalent conditions above, the same map also satisfies
\begin{equation}
(s_{p}^{(0)})_{*}\nabla^{(\eta)}
=
\nabla^{(-\eta)}
\qquad
(\eta\in\mathbb{R}).
\label{eq:LC-dual-family}
\end{equation}

Conversely, let $s_{p}$ be a local isometry such that $s_{p}(p)=p$ and $(ds_{p})_{p}=-\id_{T_{p}M}$, and suppose that $(s_{p})_{*}\nabla^{(\eta_{0})}=\nabla^{(-\eta_{0})}$ for some $\eta_{0}\ne0$.
After restriction to a common neighborhood of $p$, one has $s_{p}=s_{p}^{(0)}$, and all four equivalent conditions above hold.
\end{corollary}

\begin{proof}
For the target statistical manifold $(M,g,\nabla^{*})$, the Levi--Civita connection remains $\nabla^{g}$, whereas the skewness tensor is $-K$.
Since $\overline{g}=g$ and $\Phi(v)=-v$, uniqueness of Levi--Civita geodesics gives $\overline{\gamma}_{v}^{(0)}=\gamma_{-v}^{(0)}$, and the corresponding parallel displacements satisfy $\overline{P}_{v,t}^{(0)}=P_{-v,t}^{(0)}$.
Equations~\eqref{eq:varphi-LC} and \eqref{eq:Phi-LC} give \eqref{eq:LC-dual-spray} and \eqref{eq:LC-dual-Phi}, respectively.

Moreover, equation~\eqref{eq:alpha-family} and $\overline{K}=-K$ give $\overline{\nabla}^{(\eta)} =\nabla^{g}-(\eta/2)\overline{K} =\nabla^{(-\eta)}$ for every $\eta\in\mathbb{R}$.
Dualization interchanges $R$ and $R^{*}$ and leaves their average invariant.
Hence $\overline{R}=R^{*}$, $\overline{R}^{*}=R$, and $\overline{R}^{\mathsf{S}}=R^{\mathsf{S}}$.

Under these identifications, conditions \textup{(i)}--\textup{(iv)} are precisely conditions \textup{(i)}--\textup{(iv)} of Theorem~\ref{thm:LC-CAH-curv}.
Consequently, their equivalence follows from that theorem.

Furthermore, equation~\eqref{eq:LC-dual-spray} implies $s_{p}^{(0)}(p)=p$ and $(ds_{p}^{(0)})_{p}=-\id_{T_{p}M}$.
Applying \eqref{eq:family-scaling} with $\alpha=\beta=1$ to the dual target, for which $\overline{\nabla}^{(\eta)}=\nabla^{(-\eta)}$, gives \eqref{eq:LC-dual-family}.

Conversely, the assumed relation shows that condition \textup{(v)} of Proposition~\ref{prop:stat-morph-equiv} holds for the statistical manifolds $(M,g,\nabla)$ and $(M,g,\nabla^{*})$, with $\alpha=\beta=1$.
Hence $s_{p}$ is a local statisticmorphism from $(M,g,\nabla)$ to $(M,g,\nabla^{*})$.
The converse assertion in Theorem~\ref{thm:LC-CAH} yields $s_{p}=s_{p}^{(0)}$ after restriction to a common neighborhood of $p$, together with condition \textup{(ii)}.
Finally, conditions \textup{(iii)} and \textup{(iv)} follow from the equivalence already established.
\end{proof}

For the dual target with $\overline{p}=p$ and $\Phi=-\id_{T_{p}M}$, the base-point conditions in Corollary~\ref{cor:LC-parallel-CAH} hold automatically.

\begin{corollary}
\label{cor:LC-parallel-dual}
Let $(M,g,\nabla)$ be a statistical manifold satisfying $\nabla^{g}R^{g}=0$ and $\nabla^{g}K=0$.
Then, for every $p\in M$, the equivalent conditions in Corollary~\textup{\ref{cor:LC-dual-CAH}} hold on sufficiently small normal neighborhoods.
In particular, the map $s_{p}^{(0)}$ in \eqref{eq:LC-dual-spray} is a local statisticmorphism from $(M,g,\nabla)$ to $(M,g,\nabla^{*})$.
\end{corollary}

\begin{proof}
Fix $p\in M$ and take the source and target data specified in Corollary~\ref{cor:LC-dual-CAH}.
Since $\overline{g}=g$ and $\overline{K}=-K$, the dual target satisfies $\nabla^{\overline{g}}R^{\overline{g}}=0$ and $\nabla^{\overline{g}}\overline{K}=0$.
Moreover, by multilinearity, the base-point isometry $\Phi=-\id_{T_{p}M}$ satisfies
\[
\Phi\circ R^{g}_{v}
=
R^{g}_{\Phi(v)}\circ\Phi,
\qquad
\Phi\circ K_{u}
=
\overline{K}_{\Phi(u)}\circ\Phi
\qquad
(u,v\in T_{p}M).
\]
Thus all hypotheses of Corollary~\ref{cor:LC-parallel-CAH} are satisfied, and $s_{p}^{(0)}$ is a local statisticmorphism from $(M,g,\nabla)$ to $(M,g,\nabla^{*})$.
Finally, the remaining assertions follow from Corollary~\ref{cor:LC-dual-CAH}.
\end{proof}

%%%%%%%%%%%%%%%%%%%%%%%%%%%%
%%%%%%%%%%%%%%%%%%%%%%%%%%%%
%%%%%%%%%%%%%%%%%%%%%%%%%%%%
\section{Dual geodesic symmetries and locally statistical symmetric spaces}
\label{sec:local-stat-sym}
%%%%%%%%%%%%%%%%%%%%%%%%%%%%
%%%%%%%%%%%%%%%%%%%%%%%%%%%%
%%%%%%%%%%%%%%%%%%%%%%%%%%%%

This section introduces dual geodesic symmetries and compares the spray maps in Corollaries~\ref{cor:alpha-dual-CAH} and~\ref{cor:LC-dual-CAH}.
For each fixed $\alpha\ne0$, the two maps agree on a sufficiently small common neighborhood whenever either map is a dual geodesic symmetry.
The existence of local dual geodesic symmetries at every point is characterized by the conditions $\nabla^{g}R^{g}=0$ and $\nabla^{g}K=0$.

%%%%%%%%%%%%%%%%%%%%%%%%%%%%
%%%%%%%%%%%%%%%%%%%%%%%%%%%%
\subsection{Dual geodesic symmetries on normal neighborhoods}
%%%%%%%%%%%%%%%%%%%%%%%%%%%%
%%%%%%%%%%%%%%%%%%%%%%%%%%%%

For $\alpha\ne0$, the spray map in Corollary~\ref{cor:alpha-dual-CAH} has the form
\begin{equation}
s_{p}^{(\alpha)}
=
\operatorname{Exp}_{p}^{(-\alpha)}
\circ\bigl(-\id_{T_{p}M}\bigr)
\circ\bigl(\operatorname{Exp}_{p}^{(\alpha)}\bigr)^{-1}.
\label{eq:alpha-sym-spray}
\end{equation}
This construction uses the dual connections $\nabla^{(\alpha)}$ and $\nabla^{(-\alpha)}$ on the same statistical manifold.
In Corollary~\ref{cor:LC-dual-CAH}, both connection parameters are zero, and the target statistical manifold is $(M,g,\nabla^{*})$.
The corresponding Levi--Civita spray map in \eqref{eq:LC-dual-spray} is
\begin{equation}
s_{p}^{(0)}
=
\operatorname{Exp}_{p}^{g}
\circ\bigl(-\id_{T_{p}M}\bigr)
\circ\bigl(\operatorname{Exp}_{p}^{g}\bigr)^{-1}.
\label{eq:LC-sym-spray}
\end{equation}
Its exponential formula is obtained by formally setting $\alpha=0$ in \eqref{eq:alpha-sym-spray}.
However, the corresponding CAH criterion does not follow by this substitution: at zero parameter the condition comparing the scaled skewness tensors is vacuous, whereas the Levi--Civita criterion imposes the relation with the skewness tensor of the dual target independently.

Write $\mathcal{O}_{p}^{(0)}:=\mathcal{O}_{p}^{(0,0)}$, $O^{(0)}(p):=O^{(0,0)}(p)$, and $\overline{O}^{(0)}(p):=\overline{O}^{(0,0)}(p)$ for the tangent-space, source, and target neighborhoods, respectively, in the specialization used in Corollary~\ref{cor:LC-dual-CAH}.
After replacing the corresponding convex neighborhood $\mathcal{O}_{p}^{(0)}$ by $\mathcal{O}_{p}^{(0)}\cap\bigl(-\mathcal{O}_{p}^{(0)}\bigr)$ and replacing its exponential images accordingly, one may assume that $\mathcal{O}_{p}^{(0)}=-\mathcal{O}_{p}^{(0)}$ and $O^{(0)}(p)=\overline{O}^{(0)}(p)$.
Consequently, the map in \eqref{eq:LC-sym-spray} is a self-diffeomorphism of $O^{(0)}(p)$.
Since the statistical structure is fixed throughout this section, put
\begin{equation}
s_{p}:=s_{p}^{(0)}:
O^{(0)}(p)\longrightarrow O^{(0)}(p).
\label{eq:sp-domain}
\end{equation}
By construction,
\begin{equation}
s_{p}(p)=p,
\qquad
\bigl(ds_{p}\bigr)_{p}=-\id_{T_{p}M},
\qquad
s_{p}^{2}=\id_{O^{(0)}(p)}.
\label{eq:sp-properties}
\end{equation}
The map $s_{p}$ is the Riemannian geodesic symmetry determined by $g$ on this normal neighborhood.
Whether it is a local isometry and interchanges the statistical connections is an additional condition.

\begin{definition}
\label{def:dual-geod-sym}
The map $s_{p}$ in \eqref{eq:sp-domain} is called a \emph{dual geodesic symmetry of $(M,g,\nabla)$ at $p$} if it is a local statisticmorphism
\[
s_{p}:
(O^{(0)}(p),g|_{O^{(0)}(p)},\nabla|_{O^{(0)}(p)})
\longrightarrow
(O^{(0)}(p),g|_{O^{(0)}(p)},\nabla^{*}|_{O^{(0)}(p)}).
\]
Equivalently,
\begin{equation}
s_{p}^{*}g=g,
\qquad
(s_{p})_{*}\nabla=\nabla^{*}
\quad\text{on }O^{(0)}(p).
\label{eq:dual-stat-morph}
\end{equation}
The statistical manifold $(M,g,\nabla)$ is called a \emph{locally statistical symmetric space} if, for every $p\in M$, the map $s_{p}$ is a dual geodesic symmetry on a sufficiently small normal neighborhood $O^{(0)}(p)$ as above.
\end{definition}

A dual geodesic symmetry restricts to a dual geodesic symmetry on every smaller $s_{p}$-invariant normal neighborhood $O(p)\subset O^{(0)}(p)$.
As a local isometry, it is precisely the local Riemannian geodesic symmetry.
In particular, the underlying Riemannian manifold of a locally statistical symmetric space is locally symmetric.
By Corollary~\ref{cor:LC-dual-CAH}, the condition in Definition~\ref{def:dual-geod-sym} is equivalent to the Levi--Civita conditions \eqref{eq:LC-dual-Rg}--\eqref{eq:LC-dual-K} on the corresponding radial geodesics.

\begin{lemma}
\label{lem:dual-sym-equiv}
Let $p\in M$, and let $s_{p}:O^{(0)}(p)\to O^{(0)}(p)$ be the map in \eqref{eq:sp-domain}.
The following conditions are equivalent.
\begin{enumerate}[label={\rm(\roman*)}]
\item
$s_{p}$ is a dual geodesic symmetry at $p$.

\item
$s_{p}$ is a local isometry satisfying
\begin{equation}
(s_{p})_{*}K=K^{*}=-K.
\label{eq:dual-sym-K}
\end{equation}
Equivalently, $s_{p}$ is a local isometry satisfying $s_{p}^{*}C=-C$.

\item
For some $\alpha\ne0$, the map $s_{p}$ is a local isometry satisfying
\[
(s_{p})_{*}\nabla^{(\alpha)}=\nabla^{(-\alpha)}.
\]

\item
The map $s_{p}$ is a local isometry satisfying
\begin{equation}
(s_{p})_{*}\nabla^{(\eta)}=\nabla^{(-\eta)}
\label{eq:dual-sym-alpha}
\end{equation}
for every $\eta\in\mathbb{R}$.
\end{enumerate}
All tensor and connection identities are on $O^{(0)}(p)$.
\end{lemma}

\begin{proof}
Each condition requires $s_{p}$ to be a local isometry.
For the dual target $(M,g,\nabla^{*})$, one has $\overline{K}=-K$, $\overline{C}=-C$, and $\overline{\nabla}^{(\eta)}=\nabla^{(-\eta)}$.
Applying Proposition~\ref{prop:stat-morph-equiv} with $\alpha=\beta=1$ to this target gives the stated equivalences.
\end{proof}

%%%%%%%%%%%%%%%%%%%%%%%%%%%%
%%%%%%%%%%%%%%%%%%%%%%%%%%%%
\subsection{Dual geodesic symmetries and exponential maps of alpha-connections}
%%%%%%%%%%%%%%%%%%%%%%%%%%%%
%%%%%%%%%%%%%%%%%%%%%%%%%%%%

Lemma~\ref{lem:dual-sym-equiv} concerns the Levi--Civita spray map $s_{p}=s_{p}^{(0)}$.
For a fixed $\alpha\ne0$, Corollary~\ref{cor:alpha-dual-CAH} shows that conditions \eqref{eq:alpha-dual-R} and \eqref{eq:alpha-dual-K} hold if and only if $s_{p}^{(\alpha)}$ is a local isometry satisfying $(s_{p}^{(\alpha)})_{*}\nabla^{(\alpha)}=\nabla^{(-\alpha)}$.
The following proposition shows that these conditions imply $s_{p}^{(\alpha)}=s_{p}^{(0)}$ on a sufficiently small common neighborhood of $p$.

\begin{proposition}
\label{prop:dual-sym-sprays}
Let $p\in M$.
The following conditions are equivalent, after restricting the domains as necessary.
\begin{enumerate}[label={\rm(\roman*)}]
\item
The Levi--Civita spray map $s_{p}=s_{p}^{(0)}$ is a dual geodesic symmetry at $p$.

\item
For some $\alpha\ne0$, the spray map $s_{p}^{(\alpha)}$ is a local isometry satisfying
\[
(s_{p}^{(\alpha)})_{*}\nabla^{(\alpha)}
=
\nabla^{(-\alpha)}.
\]
\end{enumerate}

If these conditions hold, then \textup{(ii)} holds for every $\alpha\ne0$.
For each such $\alpha$, the domains of $s_{p}^{(\alpha)}$ and $s_{p}^{(0)}$ may be restricted to a common neighborhood $O$ of $p$ on which
\begin{equation}
s_{p}^{(\alpha)}=s_{p}^{(0)}=s_{p}.
\label{eq:dual-sym-sprays}
\end{equation}
Moreover, for each $\eta\in\mathbb{R}$,
\begin{equation}
s_{p}\bigl(\operatorname{Exp}_{p}^{(\eta)}(v)\bigr)
=
\operatorname{Exp}_{p}^{(-\eta)}(-v)
\label{eq:dual-sym-exp}
\end{equation}
for all sufficiently small $v\in T_{p}M$.
\end{proposition}

\begin{proof}
Assume \textup{(i)}.
Lemma~\ref{lem:dual-sym-equiv} gives \eqref{eq:dual-sym-alpha} for every $\eta\in\mathbb{R}$.
Hence $s_{p}$ sends every sufficiently short $\nabla^{(\eta)}$-geodesic issuing from $p$ with initial velocity $v$ to the $\nabla^{(-\eta)}$-geodesic issuing from $p$ with initial velocity $-v$.
Uniqueness of geodesics gives \eqref{eq:dual-sym-exp}.
Taking $\eta=\alpha\ne0$ and using \eqref{eq:alpha-sym-spray} gives \eqref{eq:dual-sym-sprays}.
The local isometry property of $s_{p}$, together with \eqref{eq:dual-sym-alpha}, gives \textup{(ii)} for every $\alpha\ne0$.

Conversely, assume \textup{(ii)} for some $\alpha\ne0$.
By the definition of $s_{p}^{(\alpha)}$, we have $s_{p}^{(\alpha)}(p)=p$ and $\bigl(ds_{p}^{(\alpha)}\bigr)_{p}=-\id_{T_{p}M}$.
Since $s_{p}^{(\alpha)}$ is a local isometry, it preserves the Levi--Civita connection and its geodesics.
Together with the preceding identities at $p$, uniqueness of geodesics gives $s_{p}^{(\alpha)}(\operatorname{Exp}_{p}^{g}(v)) =\operatorname{Exp}_{p}^{g}(-v)$ for every sufficiently small $v\in T_{p}M$.
Comparing the preceding equality with \eqref{eq:LC-sym-spray} gives \eqref{eq:dual-sym-sprays} after restriction to a common neighborhood of $p$.
Thus $s_{p}$ satisfies condition \textup{(iii)} of Lemma~\ref{lem:dual-sym-equiv}, which proves \textup{(i)}.
\end{proof}

%%%%%%%%%%%%%%%%%%%%%%%%%%%%
%%%%%%%%%%%%%%%%%%%%%%%%%%%%
\subsection{Characterization of locally statistical symmetric spaces}
%%%%%%%%%%%%%%%%%%%%%%%%%%%%
%%%%%%%%%%%%%%%%%%%%%%%%%%%%

If local dual geodesic symmetries exist at every point, then $\nabla^{g}R^{g}=0$ and $\nabla^{g}K=0$.
Conversely, Corollary~\ref{cor:LC-parallel-dual} gives local dual geodesic symmetries at every point under these conditions.

\begin{proposition}
\label{prop:dual-sym-parity}
If a dual geodesic symmetry $s_{p}$ exists at $p$, then
\begin{equation}
(\nabla^{g}R^{g})_{p}=0,
\qquad
(\nabla^{g}K)_{p}=0.
\label{eq:dual-sym-parity}
\end{equation}
\end{proposition}

\begin{proof}
Since $s_{p}$ is a local isometry,
\[
(s_{p})_{*}(\nabla^{g}R^{g})=\nabla^{g}R^{g}.
\]
At $p$, the linear map $(ds_{p})_{p}=-\id_{T_{p}M}$ acts by multiplication by $-1$ on tensors of type $(1,4)$.
Hence $(\nabla^{g}R^{g})_{p}=0$.

Lemma~\ref{lem:dual-sym-equiv} gives $(s_{p})_{*}K=-K$.
Since $s_{p}$ preserves $\nabla^{g}$,
\[
(s_{p})_{*}(\nabla^{g}K)=-\nabla^{g}K.
\]
The map $-\id_{T_{p}M}$ acts as the identity on tensors of type $(1,3)$.
Evaluation at $p$ gives $(\nabla^{g}K)_{p}=0$.
\end{proof}

\begin{theorem}
\label{thm:local-stat-sym}
Let $(M,g,\nabla)$ be a statistical manifold.
The following conditions are equivalent.
\begin{enumerate}[label={\rm(\arabic*)}]
\item
The statistical manifold $(M,g,\nabla)$ is a locally statistical symmetric space.

\item
The Riemannian manifold $(M,g)$ is locally symmetric, and the statistical structure $(g,\nabla)$ has a Levi--Civita parallel cubic form.

\item
The following two conditions hold\textup{:}
\begin{equation}
\nabla^{g}R^{g}=0,
\qquad
\nabla^{g}K=0.
\label{eq:parallel-R-K}
\end{equation}
\end{enumerate}

Under these conditions, the dual geodesic symmetry at every $p\in M$ is the local Riemannian geodesic symmetry $s_{p}=s_{p}^{(0)}$ and satisfies $(s_{p})_{*}\nabla=\nabla^{*}$.
It is unique after restriction to a common neighborhood and satisfies \eqref{eq:dual-sym-sprays}, \eqref{eq:dual-sym-alpha}, and \eqref{eq:dual-sym-exp}.
\end{theorem}

\begin{proof}
By \eqref{eq:parallel-K-C}, the parallelism condition in \textup{(2)} is equivalent to $\nabla^{g}K=0$.

Assume \textup{(1)}.
Each $s_{p}$ is the Riemannian geodesic symmetry at $p$ and is a local isometry.
Hence $(M,g)$ is locally symmetric.
Proposition~\ref{prop:dual-sym-parity}, applied at every point, gives $\nabla^{g}K=0$.
Thus \textup{(2)} holds.

Conditions \textup{(2)} and \textup{(3)} are equivalent because a Riemannian manifold is locally symmetric if and only if $\nabla^{g}R^{g}=0$.

Assume \textup{(3)}.
Corollary~\ref{cor:LC-parallel-dual} shows that, for every $p\in M$, the Levi--Civita spray map $s_{p}^{(0)}$ is a local statisticmorphism from $(M,g,\nabla)$ to $(M,g,\nabla^{*})$ on a sufficiently small normal neighborhood.
Shrinking the tangent-space neighborhood to a centrally symmetric convex neighborhood as above makes $s_{p}^{(0)}$ a self-diffeomorphism of $O^{(0)}(p)$.
Thus \textup{(1)} holds.
Uniqueness follows from the converse assertion of Corollary~\ref{cor:LC-dual-CAH}.
The remaining identities follow from Lemma~\ref{lem:dual-sym-equiv} and Proposition~\ref{prop:dual-sym-sprays}.
\end{proof}

\begin{corollary}
\label{cor:parallel-Ralpha}
Under the equivalent conditions of Theorem~\textup{\ref{thm:local-stat-sym}},
\begin{equation}
R^{(\eta)}
=R^{g}+\frac{\eta^{2}}{4}[K,K],
\qquad
R^{(\eta)}=R^{(-\eta)}
\label{eq:even-Ralpha}
\end{equation}
for every $\eta\in\mathbb{R}$.
Moreover, $\nabla^{g}R^{(\eta)}=0$ for every $\eta\in\mathbb{R}$.
\end{corollary}

\begin{proof}
Equation~\eqref{eq:Ralpha-Rg} and $\nabla^{g}K=0$ give \eqref{eq:even-Ralpha}.
Since $R^{g}$ and $K$ are $\nabla^{g}$-parallel, so is $[K,K]$.
Hence $\nabla^{g}R^{(\eta)}=0$ for every $\eta\in\mathbb{R}$.
\end{proof}

%%%%%%%%%%%%%%%%%%%%%%%%%%%%
%%%%%%%%%%%%%%%%%%%%%%%%%%%%
%%%%%%%%%%%%%%%%%%%%%%%%%%%%
\section{Statistical transvections and statistical symmetric spaces}
\label{sec:transv-stat-sym}
%%%%%%%%%%%%%%%%%%%%%%%%%%%%
%%%%%%%%%%%%%%%%%%%%%%%%%%%%
%%%%%%%%%%%%%%%%%%%%%%%%%%%%

This section constructs local statistical transvections and their infinitesimal generators on locally statistical symmetric spaces.
The local dual geodesic symmetries extend uniquely to global dual geodesic symmetries when $g$ is complete and $M$ is simply connected.
For connected statistical manifolds, the existence of a global dual geodesic symmetry at every point is equivalent to the homogeneous definition of Inoguchi and Ohno \cite{InoguchiOhno}.
For Hessian manifolds, local dual geodesic symmetries exist at every point if and only if $Q=Q^{*}$.
For standard results on Riemannian symmetric spaces and invariant connections on homogeneous spaces, see, for example, \cite{Helgason,KN2,Wolf}.

%%%%%%%%%%%%%%%%%%%%%%%%%%%%
%%%%%%%%%%%%%%%%%%%%%%%%%%%%
\subsection{Local statistical transvections and their infinitesimal generators}
\label{subsec:local-stat-transv}
%%%%%%%%%%%%%%%%%%%%%%%%%%%%
%%%%%%%%%%%%%%%%%%%%%%%%%%%%

Assume that $(M,g,\nabla)$ is a locally statistical symmetric space.
By Theorem~\ref{thm:local-stat-sym}, $(M,g)$ is locally symmetric.
Products of two dual geodesic symmetries generate a locally transitive pseudogroup of local statistical automorphisms.
Compositions of local maps in this subsection are defined on sufficiently small common open domains.

The Riemannian geodesic symmetry $s_{p}$ is the map on a normal neighborhood of $p$ introduced in \eqref{eq:sp-domain}; it satisfies $s_{p}^{2}=\id$ and $(ds_{p})_{p}=-\id$.
The map $(p,q)\mapsto s_{p}(q)$ is smooth on a neighborhood of the diagonal in $M\times M$.
For a local isometry $f$ defined on a neighborhood of $q$, uniqueness of Riemannian geodesic symmetries gives
\begin{equation}
f\circ s_{q}\circ f^{-1}=s_{f(q)}
\label{eq:sym-conjugation}
\end{equation}
on a neighborhood of $f(q)$.
In particular, taking $f=s_{p}$ gives the reflection identity
\begin{equation}
s_{p}\circ s_{q}\circ s_{p}=s_{s_{p}(q)}
\label{eq:sym-reflection}
\end{equation}
on a common neighborhood of $s_{p}(q)$.
By Lemma~\ref{lem:dual-sym-equiv}, every $s_{p}$ preserves the Levi--Civita connection and exchanges each connection $\nabla^{(\eta)}$ with its dual:
\[
(s_{p})_{*}\nabla^{g}=\nabla^{g},\qquad
(s_{p})_{*}\nabla^{(\eta)}=\nabla^{(-\eta)}
\quad(\eta\in\mathbb{R}).
\]

Let $\gamma^{g}$ be a $\nabla^{g}$-geodesic with $\gamma^{g}(0)=p$, and write $P^{\nabla^{g}}_{\gamma^{g},a\to b}$ for Levi--Civita parallel displacement along $\gamma^{g}$ from $\gamma^{g}(a)$ to $\gamma^{g}(b)$.
The \emph{local Riemannian transvections} along $\gamma^{g}$ are defined by
\begin{equation}
\tau_{t}:=s_{\gamma^{g}(t/2)}\circ s_{p}.
\label{eq:transv}
\end{equation}
Each $\tau_{t}$ moves points along $\gamma^{g}$, and its differential is Levi--Civita parallel displacement along $\gamma^{g}$:
\begin{equation}
\tau_{t}(\gamma^{g}(r))=\gamma^{g}(r+t),\qquad
(d\tau_{t})_{\gamma^{g}(r)}
=P^{\nabla^{g}}_{\gamma^{g},r\to r+t}.
\label{eq:transv-parallel}
\end{equation}

The maps $\tau_{t}$ form a smooth local one-parameter group of local isometries: there are an open interval $I$ containing $0$ and a normal neighborhood $O$ of $p$ such that $(t,q)\mapsto\tau_{t}(q)$ is smooth on $I\times O$ and
\[
\tau_{0}=\id,\qquad \tau_{t}\circ\tau_{r}=\tau_{t+r}
\]
whenever $r,t,r+t\in I$ and the composition is defined.
The interval $I$ is chosen so that \eqref{eq:transv-parallel} also holds whenever $r,t,r+t\in I$.

\begin{proposition}
\label{prop:local-transv}
Let $\mathcal{T}$ be the pseudogroup of local diffeomorphisms of $M$ generated by the maps $s_{p}\circ s_{q}$.
Every $\varphi\in\mathcal{T}$ is a local statistical automorphism\textup{:}
\begin{equation}
\varphi^{*}g=g,\qquad
\varphi_{*}\nabla=\nabla.
\label{eq:transv-preserves}
\end{equation}
Furthermore, $\mathcal{T}$ acts locally transitively on $M$\textup{:} for every $p\in M$, there is a normal neighborhood $O$ of $p$ such that every $q\in O$ has the form $q=\varphi(p)$ for some $\varphi\in\mathcal{T}$.
In particular, the local transvections \eqref{eq:transv} are local statistical automorphisms.
\end{proposition}

\begin{proof}
For each Riemannian geodesic symmetry, the metric and the skewness tensor satisfy
\[
s_{p}^{*}g=g,\qquad (s_{p})_{*}K=-K.
\]
Therefore, every composition $s_{p}\circ s_{q}$ preserves $g$ and $K$, and is a local statisticmorphism by Proposition~\ref{prop:stat-morph-equiv}.
Since local statisticmorphisms form a pseudogroup, every element of $\mathcal{T}$ is a local statisticmorphism.
By \eqref{eq:alpha-family}, the entire family of statistical connections is preserved:
\[
\varphi_{*}\nabla^{(\eta)}=\nabla^{(\eta)}
\qquad(\varphi\in\mathcal{T},\ \eta\in\mathbb{R}).
\]

Fix $p\in M$ and choose a normal neighborhood $O$ of $p$ small enough that, for every $q\in O$, the point $p$ is in the domain of the symmetry at the midpoint $m$ of the radial $\nabla^{g}$-geodesic from $p$ to $q$.
Such a choice is possible by the joint smoothness and common domain of the local Riemannian geodesic symmetries.
Then $(s_{m}\circ s_{p})(p)=s_{m}(p)=q$.
Therefore, $\mathcal{T}$ acts locally transitively on $M$.

Finally, since $\tau_{t}=s_{\gamma^{g}(t/2)}\circ s_{p}\in\mathcal{T}$, each $\tau_{t}$ is a local statistical automorphism.
\end{proof}

For $w\in T_{p}M$, put $W(t):=(d\tau_{t})_{p}(w)$.
By \eqref{eq:transv-parallel}, $W$ is $\nabla^{g}$-parallel along $\gamma^{g}$.
Equation~\eqref{eq:alpha-family} gives
\[
\nabla^{(\eta)}_{\dot{\gamma}^{g}}W
=-\frac{\eta}{2}K(\dot{\gamma}^{g},W).
\]
Although each $\tau_{t}$ preserves $\nabla^{(\eta)}$, for $\eta\ne0$ the field $W$ is $\nabla^{(\eta)}$-parallel if and only if $K(\dot{\gamma}^{g},W)=0$ along $\gamma^{g}$.

Fix $p\in M$.
The local transvection actions at $p$ have infinitesimal generators on a common connected, $s_{p}$-invariant normal neighborhood $O$ of $p$.
The joint smoothness of $(x,q)\mapsto s_{x}(q)$ permits $O$ to be chosen independently of $v\in T_{p}M$; the time interval of the corresponding local flow may depend on $v$.
For $v\in T_{p}M$, let $\gamma^{g}$ be the $\nabla^{g}$-geodesic with initial data $\gamma^{g}(0)=p$ and $\dot{\gamma}^{g}(0)=v$.
The \emph{infinitesimal transvection at $p$ corresponding to $v$} is the generator of the corresponding local action \eqref{eq:transv},
\[
X_{v}(q):=\left.\frac{d}{dt}\right|_{t=0}\tau_{t}(q).
\]
By Proposition~\ref{prop:local-transv} and $\tau_{t}(p)=\gamma^{g}(t)$,
\[
X_{v}\in\mathfrak{aut}(O,g,\nabla),\qquad X_{v}(p)=v.
\]

Let $\mathfrak{g}_{p}$ be the Lie subalgebra of $\mathfrak{aut}(O,g,\nabla)$ generated by these vector fields.
By Proposition~\ref{prop:inf-stat-aut} \textup{(ii)}, $\mathfrak{g}_{p}$ is finite-dimensional.
Since $\mathfrak{g}_{p}\subset\mathfrak{aut}(O,g,\nabla)$, the local flow of every $X\in\mathfrak{g}_{p}$ consists of local statisticmorphisms.

The Riemannian geodesic symmetry $s_{p}$ induces an involutive automorphism of $\mathfrak{g}_{p}$.
Indeed, \eqref{eq:sym-reflection} gives
\[
s_{p}\circ\tau_{t}\circ s_{p}^{-1}=\tau_{-t},\qquad
(s_{p})_{*}X_{v}=-X_{v}.
\]
Pushforward preserves Lie brackets:
\[
(s_{p})_{*}[X,Y]=[(s_{p})_{*}X,(s_{p})_{*}Y]
\qquad(X,Y\in\mathfrak{g}_{p}).
\]
Since the $X_{v}$ generate $\mathfrak{g}_{p}$ and $s_{p}^{2}=\id$,
\[
(s_{p})_{*}\mathfrak{g}_{p}=\mathfrak{g}_{p},\qquad
\bigl((s_{p})_{*}\bigr)^{2}=\id_{\mathfrak{g}_{p}}.
\]
Therefore, $(s_{p})_{*}$ restricts to an involutive automorphism of $\mathfrak{g}_{p}$.

Let $\mathfrak{h}_{p}$ and $\mathfrak{m}_{p}$ be the $+1$- and $-1$-eigenspaces of $(s_{p})_{*}$, respectively.
Since $\bigl((s_{p})_{*}\bigr)^{2}=\id_{\mathfrak{g}_{p}}$, these eigenspaces give the direct-sum decomposition
\[
\mathfrak{g}_{p}=\mathfrak{h}_{p}\oplus\mathfrak{m}_{p}.
\]
Since $(s_{p})_{*}$ preserves Lie brackets,
\begin{equation}
[\mathfrak{h}_{p},\mathfrak{h}_{p}]\subset\mathfrak{h}_{p},
\qquad
[\mathfrak{h}_{p},\mathfrak{m}_{p}]\subset\mathfrak{m}_{p},
\qquad
[\mathfrak{m}_{p},\mathfrak{m}_{p}]\subset\mathfrak{h}_{p}.
\label{eq:local-sym-pair}
\end{equation}
Thus $(\mathfrak{g}_{p},\mathfrak{h}_{p})$ is a symmetric pair.
For $X\in\mathfrak{g}_{p}$, preservation of $\nabla^{g}$ and $(ds_{p})_{p}=-\id$ give
\[
\begin{aligned}
\bigl((s_{p})_{*}X\bigr)_{p}
&=(ds_{p})_{p}(X_{p})=-X_{p},\\
\bigl(\nabla^{g}((s_{p})_{*}X)\bigr)_{p}
&=(ds_{p})_{p}\circ(\nabla^{g}X)_{p}
\circ\bigl((ds_{p})_{p}\bigr)^{-1}
=(\nabla^{g}X)_{p}.
\end{aligned}
\]
For $X\in\mathfrak{h}_{p}$, the first identity gives $X_{p}=0$.
For $X\in\mathfrak{m}_{p}$, the second identity gives $(\nabla^{g}X)_{p}=0$.
Conversely, if $X_{p}=0$, then $(s_{p})_{*}X$ and $X$ have the same value and covariant derivative at $p$.
If $(\nabla^{g}X)_{p}=0$, then $(s_{p})_{*}X$ and $-X$ have the same value and covariant derivative at $p$.
By the uniqueness assertion in Proposition~\ref{prop:inf-stat-aut} \textup{(ii)}, we obtain $(s_{p})_{*}X=X$ in the first case and $(s_{p})_{*}X=-X$ in the second case.
Thus
\[
\mathfrak{h}_{p}=\{X\in\mathfrak{g}_{p}\mid X_{p}=0\},\qquad
\mathfrak{m}_{p}=\{X\in\mathfrak{g}_{p}\mid(\nabla^{g}X)_{p}=0\}.
\]
In particular, $\mathfrak{h}_{p}$ is the isotropy subalgebra at $p$.

Evaluation at $p$, which assigns to each vector field $X$ its value $X_{p}$, identifies $\mathfrak{m}_{p}$ with the tangent space $T_{p}M$ via the linear isomorphism
\[
\operatorname{ev}_{p}:\mathfrak{m}_{p}\longrightarrow T_{p}M,
\qquad X\longmapsto X_{p},
\]
with inverse $v\mapsto X_{v}$.
Indeed, the same uniqueness result makes evaluation injective on $\mathfrak{m}_{p}$.
Moreover, $(s_{p})_{*}X_{v}=-X_{v}$ implies $X_{v}\in\mathfrak{m}_{p}$ and hence $(\nabla^{g}X_{v})_{p}=0$.
Together with $X_{v}(p)=v$, this proves surjectivity.
By Proposition~\ref{prop:inf-stat-aut} \textup{(ii)}, $X_{v}$ is the unique Killing vector field $X$ on $O$ satisfying $X_{p}=v$ and $(\nabla^{g}X)_{p}=0$.

%%%%%%%%%%%%%%%%%%%%%%%%%%%%
%%%%%%%%%%%%%%%%%%%%%%%%%%%%
\subsection{Global dual geodesic symmetries and the transvection group}
%%%%%%%%%%%%%%%%%%%%%%%%%%%%
%%%%%%%%%%%%%%%%%%%%%%%%%%%%

This subsection characterizes statistical symmetric spaces, gives a globalization criterion, and shows that the Riemannian transvection group preserves the statistical structure.
All manifolds considered in the remainder of this section are connected.

\begin{definition}
\label{def:stat-sym-space}
A \emph{global dual geodesic symmetry at $p\in M$} is a statisticmorphism
\[
s_{p}:(M,g,\nabla)\longrightarrow(M,g,\nabla^{*})
\]
satisfying
\[
s_{p}(p)=p,\qquad (ds_{p})_{p}=-\id_{T_{p}M}.
\]
The statistical manifold $(M,g,\nabla)$ is a \emph{statistical symmetric space} if such a symmetry exists at every point $p\in M$.
\end{definition}

Since $s_{p}$ is an isometry fixing $p$ with differential $-\id$, uniqueness of Levi--Civita geodesics shows that its restriction to a sufficiently small normal neighborhood of $p$ is the map in \eqref{eq:sp-domain}.
Thus every statistical symmetric space is a locally statistical symmetric space.
Theorem~\ref{thm:homo-stat-sym} shows that $(M,g,\nabla)$ is a statistical symmetric space if and only if it admits a homogeneous presentation satisfying the conditions in \cite[Definition~10.3]{InoguchiOhno}.

\begin{proposition}
\label{prop:global-stat-sym}
The following conditions are equivalent.
\begin{enumerate}[label={\rm(\roman*)}]
\item
$(M,g,\nabla)$ is a statistical symmetric space.

\item
$(M,g)$ is a Riemannian symmetric space, and the statistical structure $(g,\nabla)$ has a Levi--Civita parallel cubic form.

\item
For some fixed $\alpha\ne0$, every $p\in M$ admits a global isometry $s_{p}$ satisfying
\[
s_{p}(p)=p,\qquad (ds_{p})_{p}=-\id_{T_{p}M},\qquad
(s_{p})_{*}\nabla^{(\alpha)}=\nabla^{(-\alpha)}.
\]
\end{enumerate}
Under these conditions, the global dual geodesic symmetry at each point is unique and involutive.
For every $\eta\in\mathbb{R}$, it satisfies
\begin{equation}
(s_{p})_{*}\nabla^{(\eta)}=\nabla^{(-\eta)}.
\label{eq:global-dual-family}
\end{equation}
\end{proposition}

\begin{proof}
By \eqref{eq:parallel-K-C}, the parallelism condition in \textup{(ii)} is equivalent to $\nabla^{g}K=0$.

Under \textup{(i)}, each $s_{p}$ is a global isometry fixing $p$ with differential $-\id$.
Thus $(M,g)$ is a Riemannian symmetric space.
Applying Theorem~\ref{thm:local-stat-sym} to the local restrictions of these symmetries at every point gives $\nabla^{g}K=0$.
This proves \textup{(ii)}.

Assume \textup{(ii)} and let $s_{p}$ be the Riemannian geodesic symmetry at $p$.
Since $s_{p}$ preserves $\nabla^{g}$, the tensors $(s_{p})_{*}K$ and $-K$ are both $\nabla^{g}$-parallel.
They agree at $p$, because $(ds_{p})_{p}=-\id$ acts by $-1$ on a tensor of type $(1,2)$.
Their difference remains zero under parallel displacement along any curve issuing from $p$.
Connectedness of $M$ gives $(s_{p})_{*}K=-K$ on all of $M$.
By \eqref{eq:nabla-K}, this is equivalent to $(s_{p})_{*}\nabla=\nabla^{*}$, proving \textup{(i)}.
Equation~\eqref{eq:alpha-family} also gives \eqref{eq:global-dual-family} for every $\eta\in\mathbb{R}$, and in particular \textup{(iii)}.

Conversely, \textup{(iii)} implies that $(M,g)$ is a Riemannian symmetric space.
Since $s_{p}$ preserves $\nabla^{g}$, \eqref{eq:alpha-family} expresses the connection identity in \textup{(iii)} as $-\frac{\alpha}{2}(s_{p})_{*}K=\frac{\alpha}{2}K$.
The assumption $\alpha\ne0$ gives $(s_{p})_{*}K=-K$.
The fixed-point argument of Proposition~\ref{prop:dual-sym-parity} at every point gives $\nabla^{g}K=0$, proving \textup{(ii)}.

An isometry of a connected Riemannian manifold is determined by its value and differential at one point.
This proves uniqueness of $s_{p}$.
The isometry $s_{p}^{2}$ fixes $p$ and has differential $\id_{T_{p}M}$ there, so $s_{p}^{2}=\id$.
\end{proof}

For a statistical symmetric space, Proposition~\ref{prop:global-stat-sym} identifies the maps $s_{p}$ with the Riemannian geodesic symmetries.
The map $(p,q)\mapsto s_{p}(q)$ is smooth on all of $M\times M$.
For every global isometry $f$ and every $q\in M$, \eqref{eq:sym-conjugation} holds on all of $M$: both sides are global isometries agreeing on a neighborhood of $f(q)$, and $M$ is connected.
In particular, \eqref{eq:sym-reflection} holds globally.
Moreover, $(M,g)$ is complete.

On a complete simply connected Riemannian manifold $(M,g)$ with $\nabla^{g}R^{g}=0$, Ambrose constructs the global geodesic symmetry at each $p\in M$ by applying his global curvature comparison theorem to $M$ itself, with both base points equal to $p$ and initial isometry $-\id_{T_{p}M}$ \cite{Ambrose}.
For a statistical structure $(g,\nabla)$ on $M$ with $\nabla^{g}K=0$, Proposition~\ref{prop:global-stat-sym} shows that this symmetry also interchanges $\nabla$ and $\nabla^{*}$.

\begin{theorem}
\label{thm:global-stat-sym}
Let $(M,g,\nabla)$ be a simply connected statistical manifold such that $g$ is complete.
The following conditions are equivalent.
\begin{enumerate}[label={\rm(\arabic*)}]
\item
$(M,g,\nabla)$ is a locally statistical symmetric space.

\item
$\nabla^{g}R^{g}=0$ and $\nabla^{g}K=0$.

\item
$(M,g,\nabla)$ is a statistical symmetric space.
\end{enumerate}
Under these conditions, every local dual geodesic symmetry $s_{p}$ has a unique global extension, which satisfies \eqref{eq:global-dual-family}.
\end{theorem}

\begin{proof}
The equivalence of \textup{(1)} and \textup{(2)} is Theorem~\ref{thm:local-stat-sym}.
Assume \textup{(2)}.
The equality $\nabla^{g}R^{g}=0$ makes $(M,g)$ locally symmetric.
Since $(M,g)$ is complete and simply connected, it is a Riemannian symmetric space.
Since $\nabla^{g}K=0$, Proposition~\ref{prop:global-stat-sym} shows that its Riemannian geodesic symmetries are global dual geodesic symmetries.
Thus \textup{(3)} holds.
Conversely, \textup{(3)} implies \textup{(1)} by restriction to normal neighborhoods.

For $p\in M$, let $\hat{s}_{p}$ be the resulting global Riemannian geodesic symmetry.
Since it is an isometry fixing $p$ with differential $-\id_{T_{p}M}$, uniqueness of Levi--Civita geodesics gives
\[
\hat{s}_{p}\bigl(\operatorname{Exp}_{p}^{g}(v)\bigr)
=\operatorname{Exp}_{p}^{g}(-v)
=s_{p}\bigl(\operatorname{Exp}_{p}^{g}(v)\bigr)
\qquad(v\in\mathcal{O}_{p}^{(0)}).
\]
Here the second equality follows from \eqref{eq:LC-sym-spray}.
Thus $\hat{s}_{p}|_{O^{(0)}(p)}=s_{p}$.
Uniqueness of the extension and \eqref{eq:global-dual-family} follow from Proposition~\ref{prop:global-stat-sym}.
\end{proof}

\begin{remark}
On any connected statistical symmetric space, $g$ is complete.
Equivalently, the Levi--Civita connection $\nabla^{g}$ is complete.
The statistical connection $\nabla$ is complete if and only if $K=0$.
Indeed, $\nabla^{g}K=0$ and \cite[Theorem~4.1]{Opozda2021} show that completeness of $\nabla$ forces $K=0$.
Conversely, if $K=0$, then $\nabla=\nabla^{g}$ is complete.
\end{remark}

For a connected statistical symmetric space, the global dual geodesic symmetries lift to the universal Riemannian cover, as described in the following corollary.

\begin{corollary}
\label{cor:univ-cover}
Let $(M,g,\nabla)$ be a connected statistical symmetric space, and let $\pi:(\widetilde{M},\widetilde{g})\longrightarrow(M,g)$ be its universal Riemannian covering, where $\widetilde{g}=\pi^{*}g$.
Equip $\widetilde{M}$ with the lifted connection $\widetilde{\nabla}=\pi^{*}\nabla$.
Then $(\widetilde{M},\widetilde{g},\widetilde{\nabla})$ is a statistical symmetric space, and both $g$ and $\widetilde{g}$ are complete.
The global dual geodesic symmetries on the two spaces satisfy
\[
\pi\circ \widetilde{s}_{\widetilde{p}}
=s_{\pi(\widetilde{p})}\circ\pi
\qquad(\widetilde{p}\in\widetilde{M}).
\]
\end{corollary}

\begin{proof}
By Proposition~\ref{prop:global-stat-sym}, $(M,g)$ is a Riemannian symmetric space, so $g$ is complete.
Since $\pi$ is a Riemannian covering, $\widetilde{g}$ is complete as well.
The lifted pair $(\widetilde{g},\widetilde{\nabla})$ is a statistical structure, and its dual connection is $\widetilde{\nabla}^{*}=\pi^{*}\nabla^{*}$.

Fix $\widetilde{p}\in\widetilde{M}$ and put $p=\pi(\widetilde{p})$.
Since $\widetilde{M}$ is simply connected, there is a unique lift $\widetilde{S}:\widetilde{M}\longrightarrow\widetilde{M}$ satisfying
\[
\pi\circ\widetilde{S}=s_{p}\circ\pi,
\qquad
\widetilde{S}(\widetilde{p})=\widetilde{p}.
\]
Since $s_{p}^{2}=\id$, uniqueness of lifts gives $\widetilde{S}^{2}=\id$.
Thus $\widetilde{S}$ is a diffeomorphism.
Differentiating $\pi\circ\widetilde{S}=s_{p}\circ\pi$ at $\widetilde{p}$ and using $(ds_{p})_{p}=-\id$ gives $(d\widetilde{S})_{\widetilde{p}}=-\id_{T_{\widetilde{p}}\widetilde{M}}$.

The identity $\pi\circ\widetilde{S}=s_{p}\circ\pi$ also gives
\[
\begin{aligned}
\widetilde{S}^{*}\widetilde{g}
&=\pi^{*}(s_{p}^{*}g)
=\widetilde{g},\\
\widetilde{S}^{*}\widetilde{\nabla}^{*}
&=\pi^{*}(s_{p}^{*}\nabla^{*})
=\widetilde{\nabla}.
\end{aligned}
\]
Therefore, $\widetilde{S}$ is a global dual geodesic symmetry at $\widetilde{p}$.
Since $\widetilde{p}$ was arbitrary, $(\widetilde{M},\widetilde{g},\widetilde{\nabla})$ is a statistical symmetric space.
By uniqueness in Proposition~\ref{prop:global-stat-sym}, $\widetilde{S}=\widetilde{s}_{\widetilde{p}}$.
Hence $\pi\circ\widetilde{s}_{\widetilde{p}}=s_{p}\circ\pi$.
\end{proof}

Let $(M,g,\nabla)$ be a connected statistical symmetric space.
Define its \emph{Riemannian transvection group} by
\[
G:=\bigl\langle s_{p}\circ s_{q}\mid p,q\in M\bigr\rangle
\subset\Iso(M,g).
\]
The group $G$ is a connected Lie subgroup of $\Iso(M,g)$.

We denote the action of $G$ on $M$ by
\[
\rho:G\times M\longrightarrow M,\qquad
\rho(a,x)=a\cdot x.
\]
For each $a\in G$, define
\[
\rho_{a}:M\longrightarrow M,\qquad
\rho_{a}(x):=\rho(a,x)=a\cdot x.
\]
Each $\rho_{a}$ is a diffeomorphism of $M$ with inverse $\rho_{a^{-1}}$.
Since $G$ is a Lie subgroup of $\Iso(M,g)$, this action is smooth and \emph{effective}.
Here effectiveness means that the identity is the only element of $G$ fixing every point of $M$.

Since $M$ is connected and $g$ is complete, the Hopf--Rinow theorem gives a Levi--Civita geodesic joining any $p,q\in M$.
If $m$ is its midpoint, then $(s_{m}\circ s_{p})(p)=q$.
Thus $G$ acts transitively on $M$.

\begin{proposition}
\label{prop:global-stat-transv}
Let $(M,g,\nabla)$ be a connected statistical symmetric space, and let $G$ be its Riemannian transvection group.
Then the following statements hold.
\begin{enumerate}[label={\rm(\roman*)}]
\item
The group $G$ is a connected Lie group of statistical automorphisms.
In particular,
\[
G\subset\operatorname{Aut}(M,g,\nabla)_{0}
\subset\Iso(M,g)_{0},
\]
where the subscript $0$ denotes the identity component.

\item
The groups $G$, $\operatorname{Aut}(M,g,\nabla)_{0}$, and $\operatorname{Aut}(M,g,\nabla)$ act effectively and transitively on $M$.

\item
Conjugation by each $s_{p}$ preserves both $\operatorname{Aut}(M,g,\nabla)$ and $G$.
\end{enumerate}
\end{proposition}

\begin{proof}
The preceding discussion shows that $G$ is a connected Lie group.

By \eqref{eq:global-dual-family}, every generator $s_{p}\circ s_{q}$ preserves $\nabla$.
Each generator also preserves $g$, so it is a statistical automorphism.
Since statistical automorphisms form a group, $G\subset\operatorname{Aut}(M,g,\nabla)$.
By the connectedness of $G$, $G\subset\operatorname{Aut}(M,g,\nabla)_{0}$.
The group $\operatorname{Aut}(M,g,\nabla)_{0}$ is a connected subgroup of $\Iso(M,g)$ containing the identity, so $\operatorname{Aut}(M,g,\nabla)_{0}\subset\Iso(M,g)_{0}$.
This proves \textup{(i)}.

The preceding discussion also shows that $G$ acts effectively and transitively on $M$.
Since both $\operatorname{Aut}(M,g,\nabla)_{0}$ and $\operatorname{Aut}(M,g,\nabla)$ contain $G$, they act transitively on $M$.
Their actions are effective because they are subgroups of $\Iso(M,g)$.
This proves \textup{(ii)}.

Fix $p\in M$.
For $f\in\operatorname{Aut}(M,g,\nabla)$, the isometry $s_{p}\circ f\circ s_{p}^{-1}$ satisfies
\[
(s_{p}\circ f\circ s_{p}^{-1})_{*}K
=(s_{p})_{*}\bigl(f_{*}(-K)\bigr)=K.
\]
Thus conjugation by $s_{p}$ maps $\operatorname{Aut}(M,g,\nabla)$ into itself.
The global reflection identity also gives
\[
s_{p}\circ(s_{q}\circ s_{r})\circ s_{p}^{-1}
=s_{s_{p}(q)}\circ s_{s_{p}(r)}.
\]
Thus conjugation by $s_{p}$ maps each generator of $G$ to a generator of $G$.
Hence it maps $G$ into itself.
Since $s_{p}^{-1}=s_{p}$, this conjugation is involutive.
Therefore, it preserves both groups.
This proves \textup{(iii)}.
\end{proof}

\begin{remark}
The group $G$ is determined by the Riemannian metric $g$.
The inclusion $G\subset\operatorname{Aut}(M,g,\nabla)$ holds under the assumption that $(M,g,\nabla)$ is a statistical symmetric space.
Riemannian symmetry of $(M,g)$ alone does not imply that $G$ preserves $\nabla$.
\end{remark}

Every globally defined infinitesimal statistical automorphism $X$ generates a one-parameter subgroup of $\operatorname{Aut}(M,g,\nabla)$.
Since $X$ is a Killing vector field and $g$ is complete, its flow $\varphi_{t}$ is defined for all $t\in\mathbb{R}$.
The identities $\pounds_{X}g=0$ and $\pounds_{X}\nabla=0$ show that each $\varphi_{t}$ preserves both $g$ and $\nabla$.

\begin{corollary}
\label{cor:global-inf-transv}
Let $(M,g,\nabla)$ be a connected statistical symmetric space, and let $G$ be its Riemannian transvection group.
Fix $o\in M$.
For each $v\in T_{o}M$, the local infinitesimal transvection $X_{v}$ extends to a complete infinitesimal statistical automorphism whose flow is a one-parameter subgroup of $G$.
These one-parameter subgroups generate $G$.
\end{corollary}

\begin{proof}
Let $v\in T_{o}M$, and let $\gamma^{g}:\mathbb{R}\longrightarrow M$ be the complete Levi--Civita geodesic with initial data $\gamma^{g}(0)=o$ and $\dot{\gamma}^{g}(0)=v$.
By the classical Riemannian transvection construction, \eqref{eq:transv} defines a smooth global one-parameter group $\{\tau_{t}\}_{t\in\mathbb{R}}\subset G$, and \eqref{eq:transv-parallel} holds for all $r,t\in\mathbb{R}$.
By Proposition~\ref{prop:global-stat-transv}, each $\tau_{t}$ is a statistical automorphism.
The local flow defining $X_{v}$ is the restriction of this global flow.
Therefore, the infinitesimal generator of the global flow is a complete extension of $X_{v}$.

By completeness of $g$, for each $p\in M$ there is a complete Levi--Civita geodesic $\gamma^{g}$ with $\gamma^{g}(0)=o$ and $\gamma^{g}(1)=p$.
For the corresponding transvections $\tau_{t}=s_{\gamma^{g}(t/2)}\circ s_{o}$, we have $\tau_{2}=s_{\gamma^{g}(1)}\circ s_{o}=s_{p}\circ s_{o}$.
Since
\[
s_{p}\circ s_{q}
=(s_{p}\circ s_{o})\circ(s_{q}\circ s_{o})^{-1}
\qquad(p,q\in M),
\]
the one-parameter subgroups based at $o$ generate $G$.
\end{proof}

%%%%%%%%%%%%%%%%%%%%%%%%%%%%
%%%%%%%%%%%%%%%%%%%%%%%%%%%%
\subsection{Homogeneous characterization of statistical symmetric spaces}
%%%%%%%%%%%%%%%%%%%%%%%%%%%%
%%%%%%%%%%%%%%%%%%%%%%%%%%%%

This subsection establishes the equivalence between Definition~\ref{def:stat-sym-space} and the homogeneous definition of statistical symmetric spaces in \cite{InoguchiOhno}.
%%%%%%%%%%%%%%%%%%%%%%%%%%%%
\subsubsection{Canonical connections and invariant tensors on reductive homogeneous spaces}
%%%%%%%%%%%%%%%%%%%%%%%%%%%%

Let $(M,g,\nabla)$ be a connected statistical manifold, and let $G$ be a connected Lie group acting smoothly, effectively, and transitively on $M$ by statistical automorphisms.
Proposition~\ref{prop:global-stat-transv} shows that the Riemannian transvection action of a connected statistical symmetric space satisfies these assumptions.
Denote the action by $\rho:G\times M\longrightarrow M$, and write $\rho_{a}(x):=\rho(a,x)=a\cdot x$ for $a\in G$ and $x\in M$.

The action defines a group homomorphism
\[
G\longrightarrow\operatorname{Aut}(M,g,\nabla),\qquad
a\longmapsto\rho_{a}.
\]
Since the action is effective, this homomorphism is injective.
We identify $G$ with its image and use $a\in G$ also to denote the corresponding statistical automorphism of $(M,g,\nabla)$.

Fix $o\in M$ and let
\[
H=G_{o}:=\{a\in G\mid a\cdot o=o\}
\]
be its isotropy subgroup.
The orbit map at $o$ is
\[
\rho^{o}:G\longrightarrow M,\qquad
\rho^{o}(a):=\rho(a,o)=a\cdot o.
\]
The subgroup $H$ is closed, and $\rho^{o}$ induces the $G$-equivariant diffeomorphism
\[
\overline{\rho}^{\,o}:G/H\longrightarrow M,\qquad
\overline{\rho}^{\,o}(aH):=a\cdot o.
\]
Equip $G/H$ with the $G$-invariant statistical structure $((\overline{\rho}^{\,o})^{*}g, (\overline{\rho}^{\,o})^{*}\nabla)$, and denote it again by $(g,\nabla)$.
Thus, $\overline{\rho}^{\,o}$ is a statisticmorphism, and we identify $(G/H,g,\nabla)$ with $(M,g,\nabla)$.
Under this identification, $o=eH$, and the action of $G$ becomes left translation on $G/H$, given by $a\cdot(bH)=(ab)H$.

Put
\[
\mathfrak{g}=\operatorname{Lie}(G),\qquad
\mathfrak{h}=\operatorname{Lie}(H).
\]
Write $\exp:\mathfrak{g}\longrightarrow G$ for the Lie group exponential map.
For $X\in\mathfrak{g}$, define the fundamental vector field $X^{\#}$ by
\[
X^{\#}_{x}
=
\left.\frac{d}{dt}\right|_{t=0}\exp(tX)\cdot x.
\]
Each $X^{\#}$ is a complete infinitesimal statistical automorphism, with flow $\rho_{\exp(tX)}$.
For this left action, the correspondence $X\mapsto X^{\#}$ is an antihomomorphism: $[X,Y]^{\#}=-[X^{\#},Y^{\#}]$ for $X,Y\in\mathfrak{g}$.

For $a\in G$, let $i_{a}:G\longrightarrow G$ be the inner automorphism defined by $i_{a}(b)=aba^{-1}$.
Write $\operatorname{Ad}:G\longrightarrow\operatorname{Aut}(\mathfrak{g})$ for the adjoint representation, given by $\operatorname{Ad}(a)=(di_{a})_{e}$.
The differential at $e$ of the action homomorphism $G\to\Iso(M,g)$ identifies $\mathfrak{g}$ with a Lie subalgebra of $\operatorname{Lie}(\Iso(M,g))$.
We also use $\operatorname{Ad}$ for the adjoint representation of $\Iso(M,g)$; for $a\in G$, the two adjoint actions agree on $\mathfrak{g}$.

A vector-space decomposition $\mathfrak{g}=\mathfrak{h}\oplus\mathfrak{m}$ is called \emph{reductive} if $\mathfrak{m}$ is $\operatorname{Ad}(H)$-invariant, that is, $\operatorname{Ad}(h)(\mathfrak{m})=\mathfrak{m}$ for every $h\in H$.
If $H$ is disconnected, the infinitesimal condition $[\mathfrak{h},\mathfrak{m}]\subset\mathfrak{m}$ guarantees invariance under the identity component $H_{0}$, but not necessarily under all of $H$.

Fix an arbitrary reductive decomposition.
The differential of the orbit map at $e$ is
\[
(d\rho^{o})_{e}:\mathfrak{g}\longrightarrow T_{o}M,\qquad
(d\rho^{o})_{e}(X)=X^{\#}_{o}.
\]
It is surjective with kernel $\mathfrak{h}$.
Thus, its restriction to $\mathfrak{m}$ is a linear isomorphism, which gives the identification $\mathfrak{m}\simeq T_{o}M$.

Since each $h\in H$ fixes $o$, the differential $(d\rho_{h})_{o}:T_{o}M\longrightarrow T_{o}M$ is a linear automorphism.
The homomorphism $h\mapsto(d\rho_{h})_{o}$ is the isotropy representation of $H$ on $T_{o}M$.
Under the above identification, this representation corresponds to $h\mapsto\operatorname{Ad}(h)|_{\mathfrak{m}}$.
Indeed, for $h\in H$ and $X\in\mathfrak{m}$, the identity $h^{-1}\cdot o=o$ and the equivariance of the Lie group exponential map under conjugation give
\[
h\cdot\bigl(\exp(tX)\cdot o\bigr)
=(h\exp(tX)h^{-1})\cdot o
=\exp\!\bigl(t\,\operatorname{Ad}(h)X\bigr)\cdot o.
\]
Differentiating at $t=0$ yields
\begin{equation}
(d\rho_{h})_{o}\bigl(X^{\#}_{o}\bigr)
=
\bigl(\operatorname{Ad}(h)X\bigr)^{\#}_{o}
\qquad(h\in H,\ X\in\mathfrak{m}).
\label{eq:isotropy-Ad}
\end{equation}

The definition of a statistical symmetric space in \cite[Definition~10.3]{InoguchiOhno} requires such a reductive decomposition satisfying the additional condition
\begin{equation}
[\mathfrak{m},\mathfrak{m}]\subset\mathfrak{h}.
\label{eq:sym-pair-bracket}
\end{equation}
This additional condition is not assumed in the construction of the canonical connection or in Lemma~\ref{lem:canonical-parallel} below.

Let $\nabla^{-}$ denote the $G$-invariant canonical affine connection of the second kind associated with the chosen reductive decomposition; see \cite{Nomizu1954}.
This auxiliary connection need not be torsion free, as an exception to the convention in Section~\ref{sec:stat-curv}.
For $X\in\mathfrak{m}$, its parallel displacement from $o$ along $t\mapsto\exp(tX)\cdot o$ is the map $\bigl(d\rho_{\exp(tX)}\bigr)_{o}: T_{o}M\longrightarrow T_{\exp(tX)\cdot o}M$.
Under the identification $\mathfrak{m}\simeq T_{o}M$, its torsion satisfies
\begin{equation}
T^{-}_{o}(X,Y)=-[X,Y]_{\mathfrak{m}}
\qquad(X,Y\in\mathfrak{m}),
\label{eq:canonical-torsion}
\end{equation}
where $[X,Y]_{\mathfrak{m}}$ denotes the $\mathfrak{m}$-component of $[X,Y]$; see \cite[Chapter~X, \S2]{KN2}.

The following standard property is given in \cite[Chapter~X, Proposition~2.7]{KN2}.

\begin{lemma}
\label{lem:canonical-parallel}
For all nonnegative integers $p$ and $q$, every $G$-invariant $F\in\Gamma(T^{p}_{q}(M))$ satisfies $\nabla^{-}F=0$.
\end{lemma}

The Levi--Civita parallelism of $C$ under the hypotheses of the following proposition is established in \cite{KobayashiOkuda}.
By \eqref{eq:parallel-K-C}, this is equivalent to the Levi--Civita parallelism of $K$.
The parallelism of $R^{g}$ is classical.
Both assertions follow from the preceding argument using the canonical connection and give the conditions needed for Theorem~\ref{thm:homo-stat-sym}.

\begin{proposition}
\label{prop:homo-parallel}
Let $M=G/H$ be a connected homogeneous space of a connected Lie group $G$, equipped with a $G$-invariant statistical structure $(g,\nabla)$.
Assume that the action of $G$ is effective and that there is a reductive decomposition $\mathfrak{g}=\mathfrak{h}\oplus\mathfrak{m}$ satisfying $[\mathfrak{m},\mathfrak{m}]\subset\mathfrak{h}$.
Then both $R^{g}$ and $K$ are parallel with respect to $\nabla^{g}$\textup{:}
\begin{equation}
\nabla^{g}R^{g}=0,\qquad \nabla^{g}K=0.
\label{eq:homo-parallel-R-K}
\end{equation}
\end{proposition}

\begin{proof}
By the torsion formula~\eqref{eq:canonical-torsion}, condition~\eqref{eq:sym-pair-bracket} gives $T^{-}_{o}=0$.
Since $\nabla^{-}$ is $G$-invariant, its torsion $T^{-}$ is also $G$-invariant.
Transitivity gives $T^{-}=0$ on $M$.
By Lemma~\ref{lem:canonical-parallel}, the $G$-invariant metric $g$ is $\nabla^{-}$-parallel.
Thus $\nabla^{-}$ is torsion free and metric.
Uniqueness of the Levi--Civita connection gives $\nabla^{-}=\nabla^{g}$.

Since $G$ acts by statistical automorphisms, the tensor $K$ is $G$-invariant.
Since $\nabla^{g}$ is $G$-invariant, its curvature tensor $R^{g}$ is also $G$-invariant.
Applying Lemma~\ref{lem:canonical-parallel} to these two tensors gives \eqref{eq:homo-parallel-R-K}.
\end{proof}

%%%%%%%%%%%%%%%%%%%%%%%%%%%%
\subsubsection{Global Riemannian symmetry of homogeneous spaces}
%%%%%%%%%%%%%%%%%%%%%%%%%%%%

Proposition~\ref{prop:homo-parallel} gives local Riemannian symmetry.
Under the same assumptions on the homogeneous presentation and its invariant metric, the following lemma proves that $(M,g)$ is a Riemannian symmetric space without requiring simple connectivity.

\begin{lemma}
\label{lem:homo-global-Riem-sym}
Let $M=G/H$ be a connected homogeneous space of a connected Lie group $G$, and let $g$ be a $G$-invariant Riemannian metric on $M$.
Assume that the action of $G$ is effective and that there is an $\operatorname{Ad}(H)$-invariant decomposition $\mathfrak{g}=\mathfrak{h}\oplus\mathfrak{m}$ satisfying $[\mathfrak{m},\mathfrak{m}]\subset\mathfrak{h}$.
Then $(M,g)$ is a Riemannian symmetric space.
Moreover, the deck transformation group of its universal Riemannian covering $\pi:(\widetilde{M},\widetilde{g})\longrightarrow(M,g)$ centralizes the Riemannian transvection group of $(\widetilde{M},\widetilde{g})$.
\end{lemma}

\begin{proof}
Put $o=eH$ and let $\nabla^{-}$ be the canonical connection associated with the chosen reductive decomposition.
The metric and torsion arguments in the proof of Proposition~\ref{prop:homo-parallel} give $\nabla^{-}=\nabla^{g}$.
Since $R^{g}$ is $G$-invariant, Lemma~\ref{lem:canonical-parallel} gives $\nabla^{g}R^{g}=0$.

For $a\in G$ and $X\in\mathfrak{m}$, put
\[
p:=a\cdot o,\qquad
\gamma^{g}(t):=a\cdot(\exp(tX)\cdot o),\qquad
\varphi_{t}:=\rho_{a}\circ\rho_{\exp(tX)}\circ\rho_{a^{-1}}.
\]
By the property of parallel displacement for the canonical connection, $\gamma^{g}$ is a Levi--Civita geodesic and
\[
\varphi_{t}(p)=\gamma^{g}(t),\qquad
(d\varphi_{t})_{p}=P^{\nabla^{g}}_{\gamma^{g},0\to t}
\qquad(t\in\mathbb{R}).
\]
For $a=e$, the initial velocity of this geodesic is $X^{\#}_{o}$.
Therefore,
\begin{equation}
\operatorname{Exp}_{o}^{g}(tX^{\#}_{o})
=\rho^{o}\bigl(\exp(tX)\bigr)
=\exp(tX)\cdot o
\qquad(X\in\mathfrak{m},\ t\in\mathbb{R}).
\label{eq:Riem-Lie-exp}
\end{equation}
Since the map $X\mapsto X^{\#}_{o}$ identifies $\mathfrak{m}\simeq T_{o}M$, every initial tangent vector at every point $p$ is obtained in this way.
Thus every Levi--Civita geodesic is defined on $\mathbb{R}$, and $(M,g)$ is complete.
Therefore, its universal Riemannian covering $\pi:(\widetilde{M},\widetilde{g})\longrightarrow(M,g)$ is complete, simply connected, and locally symmetric.
Hence $(\widetilde{M},\widetilde{g})$ is a Riemannian symmetric space.

Let $\widetilde{\varphi}_{t}$ be the lift of the one-parameter group $\varphi_{t}$ with $\widetilde{\varphi}_{0}=\id$.
Uniqueness of the lift of this flow implies that $\widetilde{\varphi}_{t}$ commutes with every deck transformation of $\pi$.
For $\widetilde{p}\in\pi^{-1}(p)$, let $\widetilde{\gamma}^{\,\widetilde{g}}$ be the lift of $\gamma^{g}$ starting at $\widetilde{p}$, and put
\[
\widetilde{\tau}_{t}
:=\widetilde{s}_{\widetilde{\gamma}^{\,\widetilde{g}}(t/2)}
\circ \widetilde{s}_{\widetilde{p}}.
\]
These transvections satisfy the identity for parallel displacement in~\eqref{eq:transv-parallel} on $(\widetilde{M},\widetilde{g})$.
Since $\pi$ is a local isometry, we obtain
\[
\widetilde{\varphi}_{t}(\widetilde{p})
=\widetilde{\tau}_{t}(\widetilde{p})
=\widetilde{\gamma}^{\,\widetilde{g}}(t),\qquad
(d\widetilde{\varphi}_{t})_{\widetilde{p}}
=(d\widetilde{\tau}_{t})_{\widetilde{p}}
=P^{\nabla^{\widetilde{g}}}_{\widetilde{\gamma}^{\,\widetilde{g}},0\to t}.
\]
An isometry of a connected Riemannian manifold is determined by its value and differential at one point.
Therefore, $\widetilde{\varphi}_{t}=\widetilde{\tau}_{t}$.
Every geodesic on $\widetilde{M}$ is obtained by lifting a geodesic $\gamma^{g}$ of the above form.
Thus the deck transformation group of $\pi$ centralizes the Riemannian transvection group of $(\widetilde{M},\widetilde{g})$.
By \cite[Theorem~8.3.11]{Wolf}, the quotient $(M,g)$ is a Riemannian symmetric space.
\end{proof}

%%%%%%%%%%%%%%%%%%%%%%%%%%%%
\subsubsection{Global dual geodesic symmetries and homogeneous presentations}
%%%%%%%%%%%%%%%%%%%%%%%%%%%%

\begin{theorem}
\label{thm:homo-stat-sym}
Let $(M,g,\nabla)$ be a connected statistical manifold.
The following conditions are equivalent.
\begin{enumerate}[label={\rm(\roman*)}]
\item
Every point $p\in M$ admits a global dual geodesic symmetry in the sense of Definition~\textup{\ref{def:stat-sym-space}}.

\item
There is a homogeneous presentation $M=G/H$ in which $G$ is a connected Lie group acting effectively by statistical automorphisms.
Put $\mathfrak{g}=\operatorname{Lie}(G)$ and $\mathfrak{h}=\operatorname{Lie}(H)$.
There is an $\operatorname{Ad}(H)$-invariant decomposition $\mathfrak{g}=\mathfrak{h}\oplus\mathfrak{m}$ satisfying
\[
[\mathfrak{m},\mathfrak{m}]\subset\mathfrak{h}.
\]
\end{enumerate}
When these conditions hold, the group $G$ in \textup{(ii)} may be chosen as the Riemannian transvection group.
\end{theorem}

\begin{proof}
\textup{(i)} $\Longrightarrow$ \textup{(ii)}.
Assume \textup{(i)}.
By Proposition~\ref{prop:global-stat-sym}, $(M,g)$ is a Riemannian symmetric space.
Let $G$ be its Riemannian transvection group.
By Proposition~\ref{prop:global-stat-transv}, $G$ is connected and acts effectively and transitively on $M$ by statistical automorphisms.
Fix $o\in M$ and put $H=G_{o}$.

By Proposition~\ref{prop:global-stat-transv}, conjugation by $s_{o}$ preserves $G$ and defines a Lie group automorphism
\[
s:G\longrightarrow G,\qquad
s(a):=s_{o}\circ a\circ s_{o}^{-1}.
\]
Its differential $(ds)_{e}:\mathfrak{g}\longrightarrow\mathfrak{g}$ at $e\in G$ is given by $(ds)_{e}=\operatorname{Ad}(s_{o})|_{\mathfrak{g}}$.
Since $s_{o}^{2}=\id_{M}$, we have $s^{2}=\id_{G}$ and $((ds)_{e})^{2}=\id_{\mathfrak{g}}$.

For every $h\in H$, the isometry $h\circ s_{o}\circ h^{-1}$ fixes $o$ and has differential $-\id_{T_{o}M}$ at $o$.
By uniqueness of the Riemannian geodesic symmetry, $h\circ s_{o}\circ h^{-1}=s_{o}$.
Thus $s(h)=h$ for every $h\in H$, that is, $s|_{H}=\id_{H}$.
Differentiating this identity at $e$ gives $(ds)_{e}|_{\mathfrak{h}}=\id_{\mathfrak{h}}$.
Hence
\[
\mathfrak{h}
\subset\{X\in\mathfrak{g}\mid (ds)_{e}X=X\}.
\]

Conversely, let $X\in\mathfrak{g}$ satisfy $(ds)_{e}X=X$.
Equivariance of the Lie group exponential map under $s$ gives the following commutative diagram:
\[
\begin{tikzcd}
\mathfrak{g} \arrow[r, "(ds)_{e}"] \arrow[d, "\exp"']
  & \mathfrak{g} \arrow[d, "\exp"] \\
G \arrow[r, "s"'] & G
\end{tikzcd}
\]

Evaluating the two compositions at $tX$ and using $(ds)_{e}X=X$ yields
\[
s\bigl(\exp(tX)\bigr)
=\exp\bigl(t(ds)_{e}X\bigr)
=\exp(tX)
\qquad(t\in\mathbb{R}).
\]
By the definition of $s$, the equality $s(\exp(tX))=\exp(tX)$ is equivalent to
\[
s_{o}\circ\exp(tX)=\exp(tX)\circ s_{o}.
\]
Evaluating this identity at $o$ and using $s_{o}(o)=o$ yields
\[
\begin{aligned}
s_{o}\bigl(\exp(tX)\cdot o\bigr)
&=\exp(tX)\cdot s_{o}(o)\\
&=\exp(tX)\cdot o.
\end{aligned}
\]

Differentiating at $t=0$ gives $(ds_{o})_{o}X^{\#}_{o}=X^{\#}_{o}$.
Since $(ds_{o})_{o}=-\id_{T_{o}M}$, it follows that $X^{\#}_{o}=0$.
Since $\ker((d\rho^{o})_{e})=\mathfrak{h}$, we obtain $X\in\mathfrak{h}$.
Together with the preceding inclusion, this proves
\[
\mathfrak{h}
=\{X\in\mathfrak{g}\mid (ds)_{e}X=X\}
=\ker\bigl((ds)_{e}-\id_{\mathfrak{g}}\bigr).
\]
Thus $\mathfrak{h}$ is the $+1$-eigenspace of $(ds)_{e}$.

Let $\mathfrak{m}$ be the $-1$-eigenspace of $(ds)_{e}$.
Since $((ds)_{e})^{2}=\id_{\mathfrak{g}}$, the eigenspaces give the direct-sum decomposition $\mathfrak{g}=\mathfrak{h}\oplus\mathfrak{m}$.

Since $s$ is a group automorphism, the following diagram commutes for every $h\in G$:
\[
\begin{tikzcd}
G \arrow[r, "i_{h}"] \arrow[d, "s"']
  & G \arrow[d, "s"] \\
G \arrow[r, "i_{s(h)}"'] & G
\end{tikzcd}
\]
Differentiating $s\circ i_{h}=i_{s(h)}\circ s$ at $e$ and using $s(h)=h$ for $h\in H$ gives
\[
(ds)_{e}\circ\operatorname{Ad}(h)
=\operatorname{Ad}(s(h))\circ(ds)_{e}
=\operatorname{Ad}(h)\circ(ds)_{e}
\qquad(h\in H).
\]

For $X\in\mathfrak{m}$ and $h\in H$, this yields
\[
(ds)_{e}\bigl(\operatorname{Ad}(h)X\bigr)
=\operatorname{Ad}(h)\bigl((ds)_{e}X\bigr)
=-\operatorname{Ad}(h)X.
\]
Since $\mathfrak{m}$ is the $-1$-eigenspace of $(ds)_{e}$, we have $\operatorname{Ad}(h)X\in\mathfrak{m}$.
Hence $\operatorname{Ad}(h)(\mathfrak{m})\subset\mathfrak{m}$.
Since $H$ is a subgroup, $h^{-1}\in H$.
Therefore, we also obtain $\operatorname{Ad}(h^{-1})(\mathfrak{m})\subset\mathfrak{m}$.
Applying $\operatorname{Ad}(h)$ to both sides and using $\operatorname{Ad}(h)\circ\operatorname{Ad}(h^{-1}) =\id_{\mathfrak{g}}$ yields
\[
\mathfrak{m}
=\operatorname{Ad}(h)\bigl(\operatorname{Ad}(h^{-1})(\mathfrak{m})\bigr)
\subset\operatorname{Ad}(h)(\mathfrak{m}).
\]
Thus $\operatorname{Ad}(h)(\mathfrak{m})=\mathfrak{m}$ for every $h\in H$, proving that $\mathfrak{m}$ is $\operatorname{Ad}(H)$-invariant.

Moreover, for $X,Y\in\mathfrak{m}$, the Lie algebra automorphism $(ds)_{e}$ satisfies
\[
(ds)_{e}[X,Y]
=[(ds)_{e}X,(ds)_{e}Y]
=[-X,-Y]=[X,Y].
\]
Hence $[X,Y]\in\mathfrak{h}$, which proves \eqref{eq:sym-pair-bracket}.
This proves \textup{(ii)} with $G$ equal to the Riemannian transvection group.

\textup{(ii)} $\Longrightarrow$ \textup{(i)}.
Assume \textup{(ii)}.
By Lemma~\ref{lem:homo-global-Riem-sym}, $(M,g)$ is a Riemannian symmetric space.
By Proposition~\ref{prop:homo-parallel}, $\nabla^{g}K=0$.
Therefore, Proposition~\ref{prop:global-stat-sym} shows that every point of $M$ admits a global dual geodesic symmetry, proving \textup{(i)}.
\end{proof}

Proposition~\ref{prop:homo-parallel} gives the tensorial consequences of the homogeneous condition.
Theorem~\ref{thm:homo-stat-sym} characterizes the same condition by global dual geodesic symmetries, whose action on the dual connections is given by \eqref{eq:global-dual-family}.
These results provide the curvature and geodesic interpretation requested in \cite[Problem~10.5]{InoguchiOhno}.

For a fixed connected Riemannian symmetric space $(M,g)$, let $G$ be its effective Riemannian transvection group, fix $o\in M$, and put $H=G_{o}$.
By Proposition~\ref{prop:global-stat-transv} and Theorem~\ref{thm:homo-stat-sym}, a statistical connection $\nabla$ on $(M,g)$ makes $(M,g,\nabla)$ a statistical symmetric space if and only if $\nabla$ is $G$-invariant.
Combining this equivalence with the correspondence of Ishi, Kobayashi, and Okuda \cite[Proposition~5.4]{IshiKobayashiOkuda} gives a bijection from the statistical connections $\nabla$ for which $(M,g,\nabla)$ is a statistical symmetric space onto $\bigl(S^{3}(T_{o}^{*}M)\bigr)^{H}$ by $\nabla\longmapsto(\nabla g)_{o}$.
Here $\bigl(S^{3}(T_{o}^{*}M)\bigr)^{H}$ denotes the space of symmetric trilinear forms on $T_{o}M$ invariant under the isotropy representation of $H$.

\begin{remark}
\label{rem:local-global-transv}
Use the Riemannian transvection presentation and symmetric decomposition from the proof of \textup{(i)} $\Longrightarrow$ \textup{(ii)} in Theorem~\ref{thm:homo-stat-sym}.
Let $O$ and $\mathfrak{g}_{o}=\mathfrak{h}_{o}\oplus\mathfrak{m}_{o}$ be as in Section~\ref{subsec:local-stat-transv}.
Restriction of fundamental vector fields gives a Lie algebra anti-isomorphism
\[
\mathfrak{g}\longrightarrow\mathfrak{g}_{o},\qquad
X\longmapsto X^{\#}|_{O},
\]
which restricts to vector-space isomorphisms $\mathfrak{h}\simeq\mathfrak{h}_{o}$ and $\mathfrak{m}\simeq\mathfrak{m}_{o}$.

Indeed, the classical construction of Riemannian transvections along Levi--Civita geodesics gives $X^{\#}|_{O}=X_{v}$ for $X\in\mathfrak{m}$ and $v=X^{\#}_{o}$.
The generation of $G$ by transvections in Corollary~\ref{cor:global-inf-transv} and the uniqueness of Killing vector fields in Proposition~\ref{prop:inf-stat-aut} \textup{(ii)}, together with the effectiveness of the action of $G$, give bijectivity.
The identity $[X,Y]^{\#}=-[X^{\#},Y^{\#}]$ shows that the restriction map is an antihomomorphism.
\end{remark}

The tensorial conditions~\eqref{eq:parallel-R-K} characterize local statistical symmetry.
When $g$ is complete, simple connectivity is a sufficient additional hypothesis for these conditions to imply the existence of a global homogeneous presentation.

\begin{corollary}
\label{cor:complete-homo-sym}
Let $(M,g,\nabla)$ be a connected statistical manifold.
Suppose that $M$ is simply connected and $g$ is complete.
Then $(M,g,\nabla)$ admits a homogeneous presentation satisfying condition~\textup{(ii)} of Theorem~\textup{\ref{thm:homo-stat-sym}} if and only if
\begin{equation}
\nabla^{g}R^{g}=0,\qquad \nabla^{g}K=0.
\label{eq:homo-stat-sym}
\end{equation}
\end{corollary}

\begin{proof}
If such a homogeneous presentation exists, Proposition~\ref{prop:homo-parallel} gives \eqref{eq:homo-stat-sym}.
Conversely, suppose that \eqref{eq:homo-stat-sym} holds.
Since $M$ is simply connected and $g$ is complete, Theorem~\ref{thm:global-stat-sym} shows that $(M,g,\nabla)$ is a statistical symmetric space.
Therefore, Theorem~\ref{thm:homo-stat-sym} gives the required homogeneous presentation.
\end{proof}

%%%%%%%%%%%%%%%%%%%%%%%%%%%%
%%%%%%%%%%%%%%%%%%%%%%%%%%%%
\subsection{Dual geodesic symmetry on Hessian manifolds}
\label{subsec:Hess-sym}
%%%%%%%%%%%%%%%%%%%%%%%%%%%%
%%%%%%%%%%%%%%%%%%%%%%%%%%%%

For Hessian manifolds, the local and global criteria for statistical symmetry in Theorems~\ref{thm:local-stat-sym} and~\ref{thm:global-stat-sym} can be expressed in terms of the Hessian curvature tensors $Q$ and $Q^{*}$.

\begin{corollary}
\label{cor:Hess-stat-sym}
Let $(M,g,\nabla)$ be a Hessian manifold.
The following conditions are equivalent\textup{:}
\begin{enumerate}[label={\rm(\roman*)}]
\item $(M,g,\nabla)$ is a locally statistical symmetric space\textup{;}
\item $\nabla^{g}K=0$\textup{;}
\item $Q=Q^{*}$.
\end{enumerate}
If, in addition, $M$ is simply connected and $g$ is complete, these conditions are also equivalent to $(M,g,\nabla)$ being a statistical symmetric space.
\end{corollary}

\begin{proof}
For every Hessian manifold, $R=R^{*}=0$ and \eqref{eq:RS-RA-K} give
\begin{equation}
R^{g}=-\frac{1}{4}[K,K].
\label{eq:Hess-Rg}
\end{equation}

Condition \textup{(i)} implies \textup{(ii)} by Theorem~\ref{thm:local-stat-sym}.
Conversely, suppose that $\nabla^{g}K=0$.
By \eqref{eq:Hess-Rg}, $\nabla^{g}R^{g}=0$.
Together with $\nabla^{g}K=0$, this gives \textup{(i)} by Theorem~\ref{thm:local-stat-sym}.

The equivalence of \textup{(ii)} and \textup{(iii)} follows from $Q-Q^{*}=\nabla^{g}K$ in \eqref{eq:Q-dual-diff}.

If, in addition, $M$ is simply connected and $g$ is complete, Theorem~\ref{thm:global-stat-sym} shows that \textup{(i)} is equivalent to $(M,g,\nabla)$ being a statistical symmetric space.
\end{proof}

For every Hessian manifold, $R=R^{*}=0$ and \eqref{eq:conj-sym-dK} give $d^{\nabla^{g}}K=0$.
Therefore, \eqref{eq:Ralpha-Rg} and~\eqref{eq:Hess-Rg} give, for every $\alpha\in\mathbb{R}$,
\begin{equation}
R^{(\alpha)}
=(1-\alpha^{2})R^{g}
=\frac{\alpha^{2}-1}{4}[K,K].
\label{eq:Hess-Ralpha}
\end{equation}
Under the equivalent conditions in Corollary~\ref{cor:Hess-stat-sym}, both $R^{g}$ and $K$ are $\nabla^{g}$-parallel.
Therefore, $\nabla^{g}R^{(\alpha)}=0$ for every $\alpha\in\mathbb{R}$.

%%%%%%%%%%%%%%%%%%%%%%%%%%%%
%%%%%%%%%%%%%%%%%%%%%%%%%%%%
%%%%%%%%%%%%%%%%%%%%%%%%%%%%
\section{Dual geodesic symmetry in homogeneous examples}
\label{sec:examples}
%%%%%%%%%%%%%%%%%%%%%%%%%%%%
%%%%%%%%%%%%%%%%%%%%%%%%%%%%
%%%%%%%%%%%%%%%%%%%%%%%%%%%%

The cone of symmetric positive-definite matrices and the Euclidean motion group admit global dual geodesic symmetries.
On the cone, the spray maps $s_{P}^{(\alpha)}$ agree with the Riemannian geodesic symmetries.
On the Euclidean motion group, the transvection action and the left translation action give two homogeneous presentations by statistical automorphisms.
Only the transvection presentation satisfies $[\mathfrak{m},\mathfrak{m}]\subset\mathfrak{h}$.
For the family of normal distributions, global dual geodesic symmetries exist when $\alpha=0$, whereas no local dual geodesic symmetry exists when $\alpha\ne0$.
The spray map for $\alpha=1$ carries $\nabla$ to $\nabla^{*}$ but does not preserve the metric.

%%%%%%%%%%%%%%%%%%%%%%%%%%%%
%%%%%%%%%%%%%%%%%%%%%%%%%%%%
\subsection{The cone of symmetric positive-definite matrices}
\label{subsec:SPD}
\label{subsec:SPD-Hess}
%%%%%%%%%%%%%%%%%%%%%%%%%%%%
%%%%%%%%%%%%%%%%%%%%%%%%%%%%

%%%%%%%%%%%%%%%%%%%%%%%%%%%%
\subsubsection{Statistical structure and homogeneity}
%%%%%%%%%%%%%%%%%%%%%%%%%%%%

Let $n\geq2$, and let $E_{n}$ denote the $n\times n$ identity matrix.
Write $\operatorname{Sym}(n,\mathbb{R})$ for the vector space of real symmetric $n\times n$ matrices, and put $\mathcal{P}_{n}=\operatorname{Sym}^{+}(n,\mathbb{R})$.
For every $\Sigma\in\mathcal{P}_{n}$, we identify $T_{\Sigma}\mathcal{P}_{n}\simeq\operatorname{Sym}(n,\mathbb{R})$ by viewing a symmetric matrix $X$ as the initial velocity of the curve $t\mapsto\Sigma+tX$, defined for sufficiently small $|t|$.
Subscripts on $g$, $C$, and $K$ indicate their values at a point.
On this open cone, the Fisher metric of the zero-mean multivariate normal family is \cite{Skovgaard}
\begin{equation}
g_{\Sigma}(X,Y)
=\frac{1}{2}\operatorname{tr}
\bigl(\Sigma^{-1}X\Sigma^{-1}Y\bigr),
\qquad
X,Y\in T_{\Sigma}\mathcal{P}_{n}
\simeq\operatorname{Sym}(n,\mathbb{R}).
\label{eq:SPD-g}
\end{equation}
The Amari--Chentsov cubic form is \cite{Mitchell}
\begin{equation}
C_{\Sigma}(X,Y,Z)
=\operatorname{tr}
\bigl(\Sigma^{-1}X\Sigma^{-1}Y\Sigma^{-1}Z\bigr).
\label{eq:SPD-C}
\end{equation}
Cyclic invariance of the trace and invariance under transposition show that $C$ is totally symmetric and that
\[
g_{\Sigma}\bigl(X\Sigma^{-1}Y+Y\Sigma^{-1}X,Z\bigr)
=C_{\Sigma}(X,Y,Z).
\]
Therefore, the corresponding skewness tensor is
\begin{equation}
K_{\Sigma}(X,Y)
=X\Sigma^{-1}Y+Y\Sigma^{-1}X.
\label{eq:SPD-K}
\end{equation}

The congruence action of $\mathrm{GL}(n,\mathbb{R})$, $A\cdot\Sigma=A\Sigma A^{\mathsf{T}}$, preserves $g$ and $C$.
By Lemma~\ref{lem:stat-data}, it also preserves the corresponding statistical connection.
Its subgroup $\mathrm{GL}^{+}(n,\mathbb{R})$ acts transitively, with isotropy subgroup $\mathrm{SO}(n)$ at $E_{n}$.
The orbit map $\rho^{E_{n}}(A)=AA^{\mathsf{T}}$ induces the diffeomorphism
\[
\overline{\rho}^{\,E_{n}}:
\mathrm{GL}^{+}(n,\mathbb{R})/\mathrm{SO}(n)
\longrightarrow\mathcal{P}_{n},\qquad
A\mathrm{SO}(n)\longmapsto AA^{\mathsf{T}}.
\]
This identifies $(\mathcal{P}_{n},g)$ with a Riemannian symmetric space.
The cone is simply connected and complete with respect to $g$, and $g$ is nonflat for $n\geq2$.
Its Riemannian geodesic symmetry at $P\in\mathcal{P}_{n}$ is
\begin{equation}
s_{P}(\Sigma)=P\Sigma^{-1}P.
\label{eq:SPD-sym}
\end{equation}
Indeed, the symmetry at $E_{n}$ is the inversion map $j:\mathcal{P}_{n}\longrightarrow\mathcal{P}_{n}$, $j(\Sigma)=\Sigma^{-1}$.
The congruence action gives the formula at $P$.

For the cubic form \eqref{eq:SPD-C}, let $D^{\Sigma}$ be the standard flat connection in the covariance coordinates $\Sigma$.
The Levi--Civita connection is
\[
\nabla^{g}_{X}Y
=D^{\Sigma}_{X}Y
-\frac{1}{2}\bigl(X\Sigma^{-1}Y+Y\Sigma^{-1}X\bigr).
\]
Consequently,
\[
\nabla_{X}Y
=D^{\Sigma}_{X}Y-X\Sigma^{-1}Y-Y\Sigma^{-1}X.
\]
Let $D^{\Theta}$ be the standard flat connection in the precision coordinates $\Theta=\Sigma^{-1}$.
Direct calculation gives
\[
\nabla=D^{\Theta}=j^{*}D^{\Sigma},
\qquad
\nabla^{*}=D^{\Sigma}=j^{*}D^{\Theta}.
\]

%%%%%%%%%%%%%%%%%%%%%%%%%%%%
\subsubsection{Curvature and parallelism of the skewness tensor}
%%%%%%%%%%%%%%%%%%%%%%%%%%%%

The kernel of the $\mathrm{GL}^{+}(n,\mathbb{R})$ action is $\{\pm E_{n}\}\cap\mathrm{GL}^{+}(n,\mathbb{R})$.
Passing to the quotient by this discrete central subgroup gives an effective action and leaves the Lie algebra decomposition
\[
\mathfrak{gl}(n,\mathbb{R})
=\mathfrak{so}(n)\oplus\operatorname{Sym}(n,\mathbb{R})
\]
unchanged.
This decomposition is $\operatorname{Ad}(\mathrm{SO}(n))$-invariant.
Moreover, $[X,Y]\in\mathfrak{so}(n)$ for all $X,Y\in\operatorname{Sym}(n,\mathbb{R})$.
Therefore, Proposition~\ref{prop:homo-parallel} gives $\nabla^{g}R^{g}=0$ and $\nabla^{g}K=0$.
Every $\nabla^{(\alpha)}$ is invariant under the same effective action.
By Theorem~\ref{thm:homo-stat-sym}, $(\mathcal{P}_{n},g,\nabla^{(\alpha)})$ is a statistical symmetric space for every $\alpha\in\mathbb{R}$.
Its global dual geodesic symmetries are the Riemannian geodesic symmetries \eqref{eq:SPD-sym}, and \eqref{eq:global-dual-family} gives
\[
(s_{P})_{*}\nabla^{(\alpha)}=\nabla^{(-\alpha)}
\qquad(P\in\mathcal{P}_{n}).
\]
In particular, since $j=s_{E_{n}}$, we have $j_{*}\nabla=\nabla^{*}$; see also \cite[Section~5.1]{NielsenOkamura}.

Since $\nabla=D^{\Theta}$ and $\nabla^{*}=D^{\Sigma}$, we have $R=R^{*}=0$, and the structure is dually flat.
This dually flat structure, with metric $2g$, is studied in \cite{OharaSudaAmari}.
Since $\nabla^{g}K=0$, its Hessian curvature tensors satisfy $Q=Q^{*}$ by Corollary~\ref{cor:Hess-stat-sym}.

At $E_{n}$, the formula for $K$ gives $[K_{X},K_{Y}]Z=[[X,Y],Z]$.
Therefore, \eqref{eq:Hess-Ralpha} gives
\[
R^{(\alpha)}_{E_{n}}(X,Y)Z
=\frac{\alpha^{2}-1}{4}[[X,Y],Z].
\]
The double commutator is not identically zero for $n\geq2$.
Hence $(\mathcal{P}_{n},g,\nabla^{(\alpha)})$ is Hessian precisely when $\alpha=\pm1$, whereas it is a statistical symmetric space for every $\alpha$.

These connections form the usual Amari--Chentsov $\alpha$-family \cite{Mitchell}.
The associated Hessian geometry is studied in \cite{ShimaHao}.

%%%%%%%%%%%%%%%%%%%%%%%%%%%%
\subsubsection{Spray maps and global dual geodesic symmetries}
%%%%%%%%%%%%%%%%%%%%%%%%%%%%

For $\alpha\ne0$, the spray map in Corollary~\ref{cor:alpha-dual-CAH} can be computed from the $\nabla^{(\alpha)}$-geodesics.
For every $\alpha\in\mathbb{R}$, the identity $\nabla^{(\alpha)}=\nabla^{g}-\frac{\alpha}{2}K$ and the preceding formulas for $\nabla^{g}$ and $K$ give
\[
\nabla^{(\alpha)}_{X}Y
=D^{\Sigma}_{X}Y
-\frac{1+\alpha}{2}
\bigl(X\Sigma^{-1}Y+Y\Sigma^{-1}X\bigr).
\]
For a curve $\gamma(t)\in\mathcal{P}_{n}$, evaluate this formula at $\Sigma=\gamma(t)$ with $X=Y=\dot{\gamma}(t)$.
The standard flat connection satisfies $D^{\Sigma}_{\dot{\gamma}}\dot{\gamma}=\ddot{\gamma}$.
Therefore, the geodesic condition is
\[
\begin{aligned}
0=\nabla^{(\alpha)}_{\dot{\gamma}}\dot{\gamma}
&=\ddot{\gamma}-\frac{1+\alpha}{2}
\bigl(\dot{\gamma}\gamma^{-1}\dot{\gamma}
+\dot{\gamma}\gamma^{-1}\dot{\gamma}\bigr)\\
&=\ddot{\gamma}-(1+\alpha)\dot{\gamma}\gamma^{-1}\dot{\gamma}.
\end{aligned}
\]
Fix $\alpha\ne0$ and $V\in\operatorname{Sym}(n,\mathbb{R})$.
Choose an orthogonal diagonalization
\[
V=Q\operatorname{diag}(v_{1},\ldots,v_{n})Q^{\mathsf{T}},
\qquad Q^{\mathsf{T}}Q=E_{n},
\]
where $v_{1},\ldots,v_{n}\in\mathbb{R}$ are the eigenvalues of $V$.
The $\nabla^{(\alpha)}$-geodesic satisfying $\gamma_{V}^{(\alpha)}(0)=E_{n}$ and $\dot{\gamma}_{V}^{(\alpha)}(0)=V$ is given by
\[
\begin{aligned}
\gamma_{V}^{(\alpha)}(t)
&=(E_{n}-\alpha tV)^{-1/\alpha}\\
&=Q\operatorname{diag}\bigl(
(1-\alpha tv_{1})^{-1/\alpha},\ldots,
(1-\alpha tv_{n})^{-1/\alpha}
\bigr)Q^{\mathsf{T}}
\end{aligned}
\]
on the interval containing $0$ on which $1-\alpha tv_{i}>0$ for every $i=1,\ldots,n$.
Here and below, real powers of symmetric positive-definite matrices are defined by taking the corresponding powers of their positive eigenvalues, as in the second line.
Indeed, differentiation of the asserted geodesic formula, with $Q$ fixed, gives
\[
\begin{aligned}
\dot{\gamma}_{V}^{(\alpha)}(t)
&=V(E_{n}-\alpha tV)^{-1/\alpha-1},\\
\ddot{\gamma}_{V}^{(\alpha)}(t)
&=(1+\alpha)V^{2}(E_{n}-\alpha tV)^{-1/\alpha-2}.
\end{aligned}
\]
All matrix factors are functions of $V$ and commute.
Therefore,
\[
\dot{\gamma}_{V}^{(\alpha)}(t)
\bigl(\gamma_{V}^{(\alpha)}(t)\bigr)^{-1}
\dot{\gamma}_{V}^{(\alpha)}(t)
=V^{2}(E_{n}-\alpha tV)^{-1/\alpha-2}.
\]
Thus the displayed expression satisfies the preceding geodesic equation.
Evaluation at $t=0$ gives $\gamma_{V}^{(\alpha)}(0)=E_{n}$ and $\dot{\gamma}_{V}^{(\alpha)}(0)=V$.
Uniqueness of solutions to the geodesic equation proves the asserted formula.
See also \cite[equation~(5.1)]{NielsenOkamura}.

At any finite endpoint of the interval defined by the inequalities $1-\alpha tv_{i}>0$ for every $i=1,\ldots,n$, at least one of these quantities tends to zero.
The corresponding eigenvalue $(1-\alpha tv_{i})^{-1/\alpha}$ of $\gamma_{V}^{(\alpha)}(t)$ tends to infinity if $\alpha>0$, and to zero if $\alpha<0$.
Therefore, the curve cannot be extended in $\mathcal{P}_{n}$ through such an endpoint, and its maximal interval of definition is
\[
\begin{aligned}
I_{V}^{(\alpha)}
&=\{t\in\mathbb{R}\mid 1-\alpha tv_{i}>0
\text{ for every }i=1,\ldots,n\}\\
&=\{t\in\mathbb{R}\mid E_{n}-\alpha tV>0\}.
\end{aligned}
\]
Here $E_{n}-\alpha tV>0$ means that this matrix is positive definite, that is, $w^{\mathsf{T}}(E_{n}-\alpha tV)w>0$ for every $0\ne w\in\mathbb{R}^{n}$.

The exponential map evaluates this geodesic at $t=1$.
Therefore, its domain consists of the initial velocities $V$ for which $E_{n}-\alpha tV>0$ for every $t\in[0,1]$.
At $t=1$, this requires $E_{n}-\alpha V>0$.
Conversely, this inequality implies
\[
E_{n}-\alpha tV=(1-t)E_{n}+t(E_{n}-\alpha V)>0
\qquad(0\le t\le1).
\]
Consequently,
\[
\begin{aligned}
\Omega_{E_{n}}^{(\alpha)}
&=\{V\in\operatorname{Sym}(n,\mathbb{R})\mid 1\in I_{V}^{(\alpha)}\}\\
&=\{V\in\operatorname{Sym}(n,\mathbb{R})\mid E_{n}-\alpha V>0\},\\
\operatorname{Exp}_{E_{n}}^{(\alpha)}(V)
&=\gamma_{V}^{(\alpha)}(1)
=(E_{n}-\alpha V)^{-1/\alpha}.
\end{aligned}
\]

For $\alpha\ne0$, this exponential map is a diffeomorphism onto $\mathcal{P}_{n}$, with inverse
\[
\bigl(\operatorname{Exp}_{E_{n}}^{(\alpha)}\bigr)^{-1}(\Sigma)
=\frac{E_{n}-\Sigma^{-\alpha}}{\alpha}.
\]
For $V$ in this domain, $-V$ belongs to $\Omega_{E_{n}}^{(-\alpha)}$.
Therefore, for every $\Sigma\in\mathcal{P}_{n}$,
\[
\begin{aligned}
s_{E_{n}}^{(\alpha)}(\Sigma)
&=\operatorname{Exp}_{E_{n}}^{(-\alpha)}
\left(-\bigl(\operatorname{Exp}_{E_{n}}^{(\alpha)}\bigr)^{-1}(\Sigma)\right)\\
&=(\Sigma^{-\alpha})^{1/\alpha}
=\Sigma^{-1}=j(\Sigma).
\end{aligned}
\]
For $\alpha=0$, we use the Levi--Civita exponential map.
For the congruence action, $X^{\#}_{E_{n}}=2X$ for $X\in\operatorname{Sym}(n,\mathbb{R})$.
Therefore, \eqref{eq:Riem-Lie-exp} with $X=V/2$ gives
\[
\operatorname{Exp}_{E_{n}}^{(0)}(V)
=\exp(V/2)\cdot E_{n}
=\exp(V),
\qquad V\in\operatorname{Sym}(n,\mathbb{R}),
\]
where $\exp$ denotes the matrix exponential.
Thus $\Omega_{E_{n}}^{(0)}=\operatorname{Sym}(n,\mathbb{R})$.
The matrix exponential is a diffeomorphism from $\operatorname{Sym}(n,\mathbb{R})$ onto $\mathcal{P}_{n}$, with inverse
\[
\bigl(\operatorname{Exp}_{E_{n}}^{(0)}\bigr)^{-1}(\Sigma)
=\log\Sigma,
\]
where $\log\Sigma$ is the symmetric matrix logarithm.
Since $\exp(-V)=\exp(V)^{-1}$, the construction using exponential maps gives
\[
\begin{aligned}
s_{E_{n}}^{(0)}(\Sigma)
&=\operatorname{Exp}_{E_{n}}^{(0)}
\left(-\bigl(\operatorname{Exp}_{E_{n}}^{(0)}\bigr)^{-1}(\Sigma)\right)\\
&=\exp(-\log\Sigma)
=\Sigma^{-1}=j(\Sigma).
\end{aligned}
\]
Thus the same inversion is obtained for $\alpha=0$.
For every $\alpha\in\mathbb{R}$ and $P\in\mathcal{P}_{n}$, choose $A\in\mathrm{GL}^{+}(n,\mathbb{R})$ such that $AA^{\mathsf{T}}=P$.
The congruence map $\rho_{A}$ sends $E_{n}$ to $P$ and satisfies $(d\rho_{A})_{E_{n}}(V)=AVA^{\mathsf{T}}$.
Since $\rho_{A}$ preserves $\nabla^{(\alpha)}$, uniqueness of geodesics with prescribed initial conditions gives
\[
\operatorname{Exp}_{P}^{(\alpha)}(AVA^{\mathsf{T}})
=\rho_{A}\bigl(\operatorname{Exp}_{E_{n}}^{(\alpha)}(V)\bigr)
\qquad(V\in\Omega_{E_{n}}^{(\alpha)}).
\]
Consequently, for every $\alpha\in\mathbb{R}$, the exponential map at $P$ is a diffeomorphism from $\Omega_{P}^{(\alpha)} =\{V\in\operatorname{Sym}(n,\mathbb{R})\mid P-\alpha V>0\}$ onto $\mathcal{P}_{n}$.
For $\alpha=0$, this domain is all of $\operatorname{Sym}(n,\mathbb{R})$.
The identity for the exponential maps also holds with $\alpha$ replaced by $-\alpha$.
Moreover, $A(-V)A^{\mathsf{T}}=-AVA^{\mathsf{T}}$.
Therefore, the spray maps satisfy
\[
s_{P}^{(\alpha)}
=\rho_{A}\circ s_{E_{n}}^{(\alpha)}\circ\rho_{A}^{-1}.
\]
Using $s_{E_{n}}^{(\alpha)}=j$, we obtain
\[
\begin{aligned}
s_{P}^{(\alpha)}(\Sigma)
&=\operatorname{Exp}_{P}^{(-\alpha)}
\left(-\bigl(\operatorname{Exp}_{P}^{(\alpha)}\bigr)^{-1}(\Sigma)\right)\\
&=A\left(A^{-1}\Sigma(A^{\mathsf{T}})^{-1}\right)^{-1}A^{\mathsf{T}}\\
&=P\Sigma^{-1}P
\qquad(\Sigma\in\mathcal{P}_{n}).
\end{aligned}
\]
Thus, for every $\alpha\in\mathbb{R}$, the spray map $s_{P}^{(\alpha)}$ is defined on the entire cone and agrees with the global symmetry \eqref{eq:SPD-sym}.

The parallel identification $\Phi_{v,t}^{(\alpha,-\alpha)}$ can also be computed explicitly.
Fix $\alpha\ne0$, take the initial map to be $-\id$ at $E_{n}$, and put $\Sigma=\gamma_{v}^{(\alpha)}(t)$.
The corresponding $\nabla^{(-\alpha)}$-geodesic with initial velocity $-v$ has value $\Sigma^{-1}$ at time $t$.
Along these two geodesics, the parallel displacements in Setting~\ref{set:alpha-beta-CAH} are
\[
\begin{aligned}
P_{v,t}^{(\alpha)}(W)
&=\Sigma^{(1+\alpha)/2}W\Sigma^{(1+\alpha)/2},\\
\overline{P}_{v,t}^{(-\alpha)}(W)
&=\Sigma^{(\alpha-1)/2}W\Sigma^{(\alpha-1)/2}
\qquad\bigl(W\in\operatorname{Sym}(n,\mathbb{R})\bigr).
\end{aligned}
\]
These formulas follow by differentiating along the radial geodesics and substituting into the equations for parallel displacement.
Substitution into \eqref{eq:Phi-alpha-beta} gives
\[
\Phi_{v,t}^{(\alpha,-\alpha)}(W)
=-\Sigma^{-1}W\Sigma^{-1}
=(dj)_{\Sigma}(W).
\]
Moreover, \eqref{eq:SPD-g} gives $g_{\Sigma^{-1}}\bigl(\Phi_{v,t}^{(\alpha,-\alpha)}(W), \Phi_{v,t}^{(\alpha,-\alpha)}(Z)\bigr)=g_{\Sigma}(W,Z)$.
Thus the parallel identification is an isometry and is exactly the differential of the constructed spray map.

The statistical connections need not be geodesically complete for this construction to be global.
For example, $\gamma(t)=(1-t)^{-1}E_{n}$, $t<1$, is a $\nabla^{(1)}$-geodesic that cannot be extended to $t=1$.
Nevertheless, the spray map $s_{E_{n}}^{(1)}$ is the global inversion.

%%%%%%%%%%%%%%%%%%%%%%%%%%%%
%%%%%%%%%%%%%%%%%%%%%%%%%%%%
\subsection{The orientation-preserving Euclidean motion group \texorpdfstring{$\mathrm{E}(2)$}{E(2)}}
\label{subsec:E2}
%%%%%%%%%%%%%%%%%%%%%%%%%%%%
%%%%%%%%%%%%%%%%%%%%%%%%%%%%

With the statistical structure described below, the Euclidean motion group is a statistical symmetric space with a flat Riemannian metric and a nonflat statistical connection.
This example also distinguishes two homogeneous presentations of the same structure: the Riemannian transvection group gives a symmetric homogeneous presentation, whereas the left action of $\mathrm{E}(2)$ does not satisfy~\eqref{eq:sym-pair-bracket}.

%%%%%%%%%%%%%%%%%%%%%%%%%%%%
\subsubsection{Statistical structure and homogeneity}
%%%%%%%%%%%%%%%%%%%%%%%%%%%%

Let $\mathbb{E}^{2}=(\mathbb{R}^{2}(x,y),h)$ be the Euclidean plane, with $h=dx^{2}+dy^{2}$.
Its orientation-preserving isometry group is $\mathrm{E}(2)=\mathrm{SO}(2)\ltimes\mathbb{R}^{2}$.
Its matrix realization is
\[
\mathrm{E}(2)
=\left\{
\begin{pmatrix}
\cos\theta&-\sin\theta&x\\
\sin\theta&\cos\theta&y\\
0&0&1
\end{pmatrix}
\;\middle|\;
0\le\theta<2\pi,\quad x,y\in\mathbb{R}
\right\}.
\]
Write the matrix above as $([\theta],u)$, where $u=(x,y)\in\mathbb{R}^{2}$ and $[\theta]\in\mathbb{R}/2\pi\mathbb{Z}$.
If $A_{\theta}$ denotes rotation of $\mathbb{R}^{2}$ through angle $\theta$, the group law is
\[
([\theta],u)([\phi],v)
=([\theta+\phi],u+A_{\theta}v).
\]
The local forms $d\theta$ agree on overlaps and define a global one-form on $\mathrm{E}(2)$.
Put
\[
g=d\theta^{2}+h=d\theta^{2}+dx^{2}+dy^{2}.
\]
The left-invariant orthonormal frame
\[
e_{1}=\cos\theta\,\partial_{x}+\sin\theta\,\partial_{y},\qquad
e_{2}=-\sin\theta\,\partial_{x}+\cos\theta\,\partial_{y},\qquad
e_{3}=\partial_{\theta}
\]
satisfies
\[
[e_{1},e_{2}]=0,\qquad
[e_{2},e_{3}]=e_{1},\qquad
[e_{3},e_{1}]=e_{2}.
\]
Define the totally symmetric cubic form
\[
C(X,Y,Z)
=d\theta(X)h(Y,Z)+d\theta(Y)h(Z,X)+d\theta(Z)h(X,Y),
\]
and let $K$ be determined by $g(K(X,Y),Z)=C(X,Y,Z)$.
Its components in this frame are
\[
\begin{aligned}
K(e_{1},e_{1})&=K(e_{2},e_{2})=e_{3},\qquad
K(e_{1},e_{3})=K(e_{3},e_{1})=e_{1},\\
K(e_{2},e_{3})&=K(e_{3},e_{2})=e_{2},\qquad
K(e_{1},e_{2})=K(e_{2},e_{1})=K(e_{3},e_{3})=0.
\end{aligned}
\]
Put $\nabla=\nabla^{g}-\frac{1}{2}K$.
This is the case $c=1$ and $(\alpha,\beta)=(1,0)$ of \cite[Example~11]{InoguchiOhno}, after correcting the component $K^{(\alpha,\beta)}(e_{1},e_{1})$ in equation~(8.1) to $\alpha e_{3}$.
Here $\alpha$ and $\beta$ are the parameters used there.
Left translations of $\mathrm{E}(2)$ preserve $g$, $d\theta$, and $h$.
Therefore, they preserve $C$ and $\nabla$.

%%%%%%%%%%%%%%%%%%%%%%%%%%%%
\subsubsection{Curvature and parallelism of the skewness tensor}
%%%%%%%%%%%%%%%%%%%%%%%%%%%%

The coordinates $([\theta],u)$ give a diffeomorphism $\mathrm{E}(2)\simeq S^{1}\times\mathbb{R}^{2}$.
Under this identification, $g=d\theta^{2}+h$ is the complete flat product metric.
The tensors $d\theta$ and $h$ are $\nabla^{g}$-parallel.
Therefore, using \eqref{eq:parallel-K-C}, we obtain
\[
R^{g}=0,\qquad \nabla^{g}C=0,\qquad \nabla^{g}K=0.
\]
By Theorem~\ref{thm:local-stat-sym}, $(\mathrm{E}(2),g,\nabla)$ is a locally statistical symmetric space.
The curvature formula~\eqref{eq:Ralpha-Rg} gives
\[
R(X,Y)Z=\frac{1}{4}[K_{X},K_{Y}]Z.
\]
Since $K_{e_{2}}e_{1}=0$, $K_{e_{1}}e_{1}=e_{3}$, and $K_{e_{2}}e_{3}=e_{2}$, we obtain
\[
R(e_{1},e_{2})e_{1}
=\frac{1}{4}[K_{e_{1}},K_{e_{2}}]e_{1}
=-\frac{1}{4}e_{2}\ne0.
\]
Thus the Riemannian metric is flat, whereas the statistical connection is nonflat.

%%%%%%%%%%%%%%%%%%%%%%%%%%%%
\subsubsection{Global dual geodesic symmetries}
%%%%%%%%%%%%%%%%%%%%%%%%%%%%

For $p=([\theta_{0}],u_{0})\in\mathrm{E}(2)$, define
\[
s_{p}:\mathrm{E}(2)\longrightarrow\mathrm{E}(2),\qquad
s_{p}([\theta],u)=([2\theta_{0}-\theta],2u_{0}-u).
\]
This formula is independent of the representatives of the angles: changing $\theta_{0}$ and $\theta$ by multiples of $2\pi$ changes $2\theta_{0}-\theta$ by a multiple of $2\pi$.
It gives $s_{p}(p)=p$, $(ds_{p})_{p}=-\id$, and $s_{p}^{2}=\id$.
Moreover,
\[
s_{p}^{*}d\theta=-d\theta,\qquad s_{p}^{*}h=h,\qquad
s_{p}^{*}g=g.
\]
Each term of $C$ contains one factor $d\theta$ and one factor $h$.
Therefore, $s_{p}^{*}C=-C$.
The defining relation $g(K(X,Y),Z)=C(X,Y,Z)$ and the equality $s_{p}^{*}g=g$ give
\[
g\bigl((s_{p}^{*}K)(X,Y),Z\bigr)
=(s_{p}^{*}C)(X,Y,Z)
=-g(K(X,Y),Z).
\]
The nondegeneracy of $g$ gives $s_{p}^{*}K=-K$.
As an isometry, $s_{p}$ preserves the Levi--Civita connection.
Therefore,
\[
s_{p}^{*}\nabla
=\nabla^{g}-\frac{1}{2}s_{p}^{*}K
=\nabla^{g}+\frac{1}{2}K
=\nabla^{*}.
\]
Since $s_{p}^{-1}=s_{p}$, we obtain
\[
(s_{p})_{*}\nabla
=(s_{p}^{-1})^{*}\nabla
=s_{p}^{*}\nabla
=\nabla^{*}.
\]
Consequently, for each $\nabla$-geodesic $\gamma$ with $\gamma(0)=p$ and $\dot{\gamma}(0)=v$, the curve $s_{p}\circ\gamma$ is a $\nabla^{*}$-geodesic on the same interval, with initial point $p$ and initial velocity $-v$.
Thus $s_{p}$ is the global dual geodesic symmetry at $p$.
This proves directly that $(\mathrm{E}(2),g,\nabla)$ is a statistical symmetric space, although $\mathrm{E}(2)$ is not simply connected.

%%%%%%%%%%%%%%%%%%%%%%%%%%%%
\subsubsection{Homogeneous presentations by the transvection group and \texorpdfstring{$\mathrm{E}(2)$}{E(2)}}
%%%%%%%%%%%%%%%%%%%%%%%%%%%%

For $p=([\theta_{0}],u_{0})$ and $q=([\phi_{0}],v_{0})$, composition gives
\[
(s_{p}\circ s_{q})([\theta],u)
=([\theta+2(\theta_{0}-\phi_{0})],u+2(u_{0}-v_{0})).
\]
Every translation of the two factors occurs in this way: for $[\psi]\in\mathbb{R}/2\pi\mathbb{Z}$ and $a\in\mathbb{R}^{2}$, choose a representative $\psi\in\mathbb{R}$ and take $p=([\psi/2],a/2)$ and $q=([0],0)$.
The resulting composition is
\[
(s_{p}\circ s_{q})([\theta],u)=([\theta+\psi],u+a).
\]
Therefore, the Riemannian transvection group is identified with the abelian group $G=\mathrm{SO}(2)\times\mathbb{R}^{2}$, whose group law is
\[
([\psi],a)([\chi],b)=([\psi+\chi],a+b).
\]
Its action is
\[
\rho:G\times\mathrm{E}(2)\longrightarrow\mathrm{E}(2),\qquad
([\psi],a)\cdot([\theta],u)
:=\rho\bigl(([\psi],a),([\theta],u)\bigr)
=([\theta+\psi],u+a).
\]
This action preserves $d\theta$ and $h$.
Therefore, it preserves $g$, $C$, and $\nabla$.
For any two points $([\theta],u)$ and $([\phi],v)$, the unique element of $G$ that sends the first to the second is $([\phi-\theta],v-u)$.
Thus the action is effective and simply transitive.
Its isotropy subgroup is trivial, and its Lie algebra is abelian.
Consequently, its reductive decomposition has $\mathfrak{h}=0$ and $\mathfrak{m}=\mathfrak{g}$, and $[\mathfrak{m},\mathfrak{m}]=0\subset\mathfrak{h}$.
This is the symmetric homogeneous presentation in Theorem~\ref{thm:homo-stat-sym}.

On the same manifold, the left action of $\mathrm{E}(2)$ is
\[
([\psi],a)\cdot([\theta],u)
=([\psi+\theta],a+A_{\psi}u).
\]
This action also preserves $g$, $C$, and $\nabla$ and is effective and simply transitive.
Its isotropy subgroup is again trivial.
However, its Lie algebra is nonabelian, as the brackets of $e_{1},e_{2},e_{3}$ show.
Therefore, the corresponding reductive decomposition has $\mathfrak{h}=0$ and $\mathfrak{m}=\operatorname{Lie}(\mathrm{E}(2))$, but $[\mathfrak{m},\mathfrak{m}]\ne0$.
Thus this presentation does not satisfy~\eqref{eq:sym-pair-bracket}.

%%%%%%%%%%%%%%%%%%%%%%%%%%%%
%%%%%%%%%%%%%%%%%%%%%%%%%%%%
\subsection{The family of normal distributions}
\label{subsec:normal}
\label{subsec:normal-nonsym}
%%%%%%%%%%%%%%%%%%%%%%%%%%%%
%%%%%%%%%%%%%%%%%%%%%%%%%%%%

%%%%%%%%%%%%%%%%%%%%%%%%%%%%
\subsubsection{Statistical structure and homogeneity}
%%%%%%%%%%%%%%%%%%%%%%%%%%%%

Let $\mathbb{H}^{2}=\{(x,y)\in\mathbb{R}^{2}\mid y>0\}$, and associate to each $(x,y)$ the normal distribution with probability density
\[
p(t;x,y)=\frac{1}{\sqrt{2\pi}\,y}
\exp\!\left(-\frac{(t-x)^{2}}{2y^{2}}\right),
\qquad t\in\mathbb{R}.
\]
Here $x$ is the mean and $y>0$ is the standard deviation.
This parameterization identifies the family of normal distributions with $\mathbb{H}^{2}$.
Its Fisher metric is $g_{\sqrt{2}}=\frac{dx^{2}+2dy^{2}}{y^{2}}$.
Consider the one-parameter deformation
\[
g_{\lambda}=\frac{dx^{2}+\lambda^{2}dy^{2}}{y^{2}},
\qquad \lambda>0.
\]
Consider the skewness tensor
\begin{align}
K(\partial_{x},\partial_{x})
=\frac{2}{\lambda^{2}y}\partial_{y},
\qquad
K(\partial_{x},\partial_{y})
=K(\partial_{y},\partial_{x})
=\frac{2}{y}\partial_{x},
\qquad
K(\partial_{y},\partial_{y})
=\frac{4}{y}\partial_{y},
\label{eq:normal-K}
\end{align}
and put $\nabla^{(\alpha)}=\nabla^{g}-\frac{\alpha}{2}K$.
This family appears in \cite{FIK}.
At $\lambda=\sqrt{2}$, the metric and the $\alpha$-connections are those of the normal distribution family above.

The orientation-preserving affine group of the real line, $\mathrm{GA}^{+}(1)$, is represented by the matrices
\[
(a,b):=\begin{pmatrix}b&a\\0&1\end{pmatrix},
\qquad a\in\mathbb{R},\quad b>0.
\]
Here $a$ is the translation parameter and $b$ is the positive dilation parameter.
Matrix multiplication gives the group law
\[
(a,b)(c,d)=(a+bc,bd).
\]
The affine transformation $t\mapsto bt+a$ sends the normal distribution with parameters $(x,y)$ to the one with parameters $(bx+a,by)$.
Thus the group acts on the parameter space $\mathbb{H}^{2}$ by
\[
\rho:\mathrm{GA}^{+}(1)\times\mathbb{H}^{2}\longrightarrow\mathbb{H}^{2},
\qquad
\rho\bigl((a,b),(x,y)\bigr)=(bx+a,by).
\]
For $(a,b)\in\mathrm{GA}^{+}(1)$, write $\rho_{a,b}(x,y):=\rho\bigl((a,b),(x,y)\bigr)$.
At $o=(0,1)$, the orbit map $\rho^{o}:\mathrm{GA}^{+}(1)\to\mathbb{H}^{2}$ is given by $\rho^{o}(a,b)=\rho_{a,b}(o)=(a,b)$.
Its inverse is
\[
(\rho^{o})^{-1}(x,y)=\begin{pmatrix}y&x\\0&1\end{pmatrix}.
\]
$\rho^{o}$ is a diffeomorphism that identifies the group coordinates $(a,b)$ with the upper half-plane coordinates $(x,y)$, and the action corresponds to left multiplication.
Therefore, the action is simply transitive.
The frame $e_{1}=y\partial_{x}$, $e_{2}=\frac{y}{\lambda}\partial_{y}$ is left invariant and orthonormal, and its Levi--Civita derivatives are
\begin{alignat*}{2}
\nabla^{g}_{e_{1}}e_{1}=\frac{1}{\lambda} e_{2},\qquad
\nabla^{g}_{e_{1}}e_{2}=-\frac{1}{\lambda} e_{1},\qquad
\nabla^{g}_{e_{2}}e_{1}=0,\qquad
\nabla^{g}_{e_{2}}e_{2}=0.
\end{alignat*}
In this frame, $K$ has the constant components
\begin{equation}
K(e_{1},e_{1})=\frac{2}{\lambda} e_{2},\qquad
K(e_{1},e_{2})=\frac{2}{\lambda} e_{1},\qquad
K(e_{2},e_{2})=\frac{4}{\lambda} e_{2}.
\label{eq:normal-frame-K}
\end{equation}
Therefore, the action preserves $g_{\lambda}$, $K$, and every $\nabla^{(\alpha)}$.
In particular, both $(g_{\lambda},\nabla)$ and $(g_{\lambda},\nabla^{*})$ are homogeneous statistical structures.
The isotropy subgroup of this action is trivial, and $[e_{1},e_{2}]=-\frac{1}{\lambda} e_{1}$.
Thus this homogeneous presentation does not satisfy the bracket condition~\eqref{eq:sym-pair-bracket}.

%%%%%%%%%%%%%%%%%%%%%%%%%%%%
\subsubsection{Curvature and parallelism of the skewness tensor}
%%%%%%%%%%%%%%%%%%%%%%%%%%%%

In the coordinates $u=x/\lambda$ and $y$, $g_{\lambda}=\lambda^{2}(du^{2}+dy^{2})/y^{2}$.
Hence the Riemannian manifold is simply connected and complete, with constant sectional curvature $-1/\lambda^{2}$.
In particular, $\nabla^{g}R^{g}=0$.
However, direct covariant differentiation of \eqref{eq:normal-frame-K} gives $(\nabla^{g}_{e_{1}}K)(e_{1},e_{1}) =-\frac{6}{\lambda^{2}}e_{1}$; all other components $(\nabla^{g}_{e_{i}}K)(e_{j},e_{k})$ vanish.
Equivalently,
\begin{equation}
(\nabla^{g}_{\partial_{x}}K)(\partial_{x},\partial_{x})
=-\frac{6}{\lambda^{2}y^{2}}\partial_{x}.
\label{eq:normal-LC-K}
\end{equation}
Thus $\nabla^{g}K\ne0$ at every point, although $d^{\nabla^{g}}K=0$.

The Riemannian curvature and the commutator of the skewness operators satisfy
\begin{alignat*}{2}
R^{g}(e_{1},e_{2})e_{1}&=\frac{1}{\lambda^{2}}e_{2},\qquad
&R^{g}(e_{1},e_{2})e_{2}&=-\frac{1}{\lambda^{2}}e_{1},\\
[K_{e_{1}},K_{e_{2}}]e_{1}&=-\frac{4}{\lambda^{2}}e_{2},\qquad
&[K_{e_{1}},K_{e_{2}}]e_{2}&=\frac{4}{\lambda^{2}}e_{1}.
\end{alignat*}
Since $d^{\nabla^{g}}K=0$, \eqref{eq:Ralpha-Rg} gives
\begin{equation}
R^{(\alpha)}(X,Y)Z
=-\frac{1-\alpha^{2}}{\lambda^{2}}
\{g_{\lambda}(Y,Z)X-g_{\lambda}(X,Z)Y\}.
\label{eq:normal-Ralpha}
\end{equation}
In particular, $R^{(\alpha)}=R^{(-\alpha)}$ for every $\alpha$.
Thus every member of the family is conjugate symmetric.
At $\alpha=\pm1$, the connections are flat.
Therefore, $(\mathbb{H}^{2},g_{\lambda},\nabla)$ is a dually flat Hessian manifold.
The Hessian curvature tensors of this manifold and its dual satisfy
\[
Q(e_{1},e_{1})e_{1}=-\frac{2}{\lambda^{2}}e_{1},\qquad
Q^{*}(e_{1},e_{1})e_{1}=\frac{4}{\lambda^{2}}e_{1}.
\]
By \eqref{eq:Q-dual-diff}, $Q-Q^{*}=\nabla^{g}K\ne0$.
Thus the condition $Q=Q^{*}$ in Corollary~\ref{cor:Hess-stat-sym} fails.

For $\alpha=0$, the skewness tensor of $(g_{\lambda},\nabla^{(0)})$ vanishes, and $\nabla^{(0)}=\nabla^{g}$.
Therefore, the Riemannian geodesic symmetries are global dual geodesic symmetries of $(\mathbb{H}^{2},g_{\lambda},\nabla^{(0)})$.
In contrast, for $\alpha\ne0$, the skewness tensor is $\alpha K$, with $\nabla^{g}(\alpha K)=\alpha\nabla^{g}K\ne0$ at every point.
By Proposition~\ref{prop:dual-sym-parity}, no local dual geodesic symmetry exists at any point.
Therefore, by Theorem~\ref{thm:homo-stat-sym}, $(\mathbb{H}^{2},g_{\lambda},\nabla^{(\alpha)})$ has no homogeneous presentation satisfying condition~\textup{(ii)} of that theorem.

%%%%%%%%%%%%%%%%%%%%%%%%%%%%
\subsubsection{Riemannian geodesic symmetry and the spray map}
%%%%%%%%%%%%%%%%%%%%%%%%%%%%

In the coordinate $z=x/\lambda+iy$, the M\"{o}bius transformations form the Riemannian transvection group
\[
G=\mathrm{PSL}(2,\mathbb{R})=\mathrm{SL}(2,\mathbb{R})/\{\pm I\}.
\]
Its isotropy subgroup at $o$ is $\mathrm{PSO}(2)$, and $\mathbb{H}^{2}\simeq G/\mathrm{PSO}(2)$.
The $\mathrm{GA}^{+}(1)$-action above preserves every $\nabla^{(\alpha)}$, whereas the $G$-action preserves $\nabla^{(\alpha)}$ only for $\alpha=0$.

At $o=(0,1)$, the Riemannian geodesic symmetry is
\[
s_{o}^{(0)}(x,y)=\operatorname{Exp}_{o}^{(0)}\left(-\bigl(\operatorname{Exp}_{o}^{(0)}\bigr)^{-1}(x,y)\right)=\left(-\frac{\lambda^{2}x}{x^{2}+\lambda^{2}y^{2}},\frac{\lambda^{2}y}{x^{2}+\lambda^{2}y^{2}}\right).
\]
Indeed, in the coordinate $z$, this is the hyperbolic isometry $z\mapsto-1/z$.
Thus $s_{o}^{(0)}$ preserves $g_{\lambda}$, fixes $o$, and has differential $-\id$ there.
However, the preceding nonexistence result shows that $(s_{o}^{(0)})_{*}\nabla=\nabla^{*}$ does not hold on any neighborhood of $o$.

In contrast, the spray map $s_{o}^{(1)}$ satisfies $(s_{o}^{(1)})_{*}\nabla=\nabla^{*}$ but fails to preserve $g_{\lambda}$.
The flat connections $\nabla$ and $\nabla^{*}$ admit the affine coordinates
\[
\theta(x,y)=\left(\frac{x}{y^{2}},\frac{1}{y^{2}}\right),
\qquad
\eta(x,y)=\left(x,y^{2}+\frac{2x^{2}}{\lambda^{2}}\right),
\]
respectively, as follows by substitution into their connection coefficients.
At $o$, $\theta(o)=\eta(o)=(0,1)$.
With respect to the coordinate basis $(\partial_{x},\partial_{y})$, the differentials are $d\theta_{o}=\operatorname{diag}(1,-2)$ and $d\eta_{o}=\operatorname{diag}(1,2)$.
In their respective affine coordinates, geodesics are straight lines.
Therefore, for sufficiently small $v=a(\partial_{x})_{o}+b(\partial_{y})_{o}$, the exponential maps satisfy
\begin{align*}
\theta\bigl(\operatorname{Exp}_{o}^{(1)}(v)\bigr)
&=\theta(o)+d\theta_{o}(v)=(a,1-2b),\\
\eta\bigl(\operatorname{Exp}_{o}^{(-1)}(-v)\bigr)
&=\eta(o)+d\eta_{o}(-v)=(-a,1-2b).
\end{align*}
For $(x,y)=\operatorname{Exp}_{o}^{(1)}(v)$, the definition of $s_{o}^{(1)}$ gives
\[
s_{o}^{(1)}(x,y)
=\operatorname{Exp}_{o}^{(-1)}
\left(-\bigl(\operatorname{Exp}_{o}^{(1)}\bigr)^{-1}(x,y)\right)
=\operatorname{Exp}_{o}^{(-1)}(-v).
\]
Therefore, the preceding formulas give $\theta(x,y)=(a,1-2b)$ and $\eta\bigl(s_{o}^{(1)}(x,y)\bigr)=(-a,1-2b)$.
Comparison of these coordinate values yields
\[
\eta\bigl(s_{o}^{(1)}(x,y)\bigr)
=\bigl(-\theta^{1}(x,y),\theta^{2}(x,y)\bigr)
=\left(-\frac{x}{y^{2}},\frac{1}{y^{2}}\right).
\]
The inverse coordinate map is $\eta^{-1}(r,s)=(r,\sqrt{s-2r^{2}/\lambda^{2}})$ on the domain $s>2r^{2}/\lambda^{2}$.
Therefore, applying $\eta^{-1}$ to this equality gives
\begin{align*}
s_{o}^{(1)}(x,y)
&=\eta^{-1}\bigl(-\theta^{1}(x,y),\theta^{2}(x,y)\bigr)
=\eta^{-1}\left(-\frac{x}{y^{2}},\frac{1}{y^{2}}\right)\\
&=\left(-\frac{x}{y^{2}},
\frac{\sqrt{y^{2}-2x^{2}/\lambda^{2}}}{y^{2}}\right)
\end{align*}
on a sufficiently small neighborhood of $o$ contained in $\{(x,y)\in\mathbb{H}^{2}\mid \lambda^{2}y^{2}>2x^{2}\}$.
Its coordinate expression $\bigl(\eta\circ s_{o}^{(1)}\circ\theta^{-1}\bigr)(\xi^{1},\xi^{2}) =(-\xi^{1},\xi^{2})$ is linear and invertible.
The affine coordinate vector fields satisfy $\nabla_{\partial_{\theta^{i}}}\partial_{\theta^{j}}=0$ and $\nabla^{*}_{\partial_{\eta^{i}}}\partial_{\eta^{j}}=0$ for $i,j=1,2$.
Differentiation gives $(s_{o}^{(1)})_{*}\partial_{\theta^{1}}=-\partial_{\eta^{1}}$ and $(s_{o}^{(1)})_{*}\partial_{\theta^{2}}=\partial_{\eta^{2}}$.
The signs in this correspondence are constant.
Hence, for $i,j=1,2$,
\[
\nabla^{*}_{(s_{o}^{(1)})_{*}\partial_{\theta^{i}}}
\bigl((s_{o}^{(1)})_{*}\partial_{\theta^{j}}\bigr)
=(s_{o}^{(1)})_{*}
\bigl(\nabla_{\partial_{\theta^{i}}}\partial_{\theta^{j}}\bigr)=0.
\]
Therefore, by the definition of the pushforward connection, $(s_{o}^{(1)})_{*}\nabla=\nabla^{*}$.
Moreover, $s_{o}^{(1)}(o)=o$ and $(ds_{o}^{(1)})_{o}=-\id$.
Thus the fixed-point, differential, and connection conditions for a local dual geodesic symmetry are satisfied.
However, for sufficiently small $x\ne0$, write $s_{o}^{(1)}(x,1)=(X(x),Y(x))$, where $X(x)=-x$ and $Y(x)=\sqrt{1-2x^{2}/\lambda^{2}}$.
The differential sends $\partial_{x}$ to a vector with components $X^{\prime}(x)=-1$ and $Y^{\prime}(x)=-2x/(\lambda^{2}Y(x))$ at $s_{o}^{(1)}(x,1)$.
Therefore, the definition of the pullback metric and the target-coordinate expression $g_{\lambda}=(dX^{2}+\lambda^{2}dY^{2})/Y^{2}$ give
\begin{align*}
\bigl((s_{o}^{(1)})^{*}g_{\lambda}\bigr)_{(x,1)}
(\partial_{x},\partial_{x})
=&\,\frac{(X^{\prime}(x))^{2}
+\lambda^{2}(Y^{\prime}(x))^{2}}{Y(x)^{2}}
=\frac{1+2x^{2}/\lambda^{2}}
{(1-2x^{2}/\lambda^{2})^{2}}\\
>&\,(g_{\lambda})_{(x,1)}(\partial_{x},\partial_{x})=1.
\end{align*}
Consequently, $s_{o}^{(1)}$ fails the isometry condition and is not a local dual geodesic symmetry.
The spray map $s_{o}^{(1)}$ does not send $\nabla^{*}$ to $\nabla$ on any neighborhood of $o$.
Indeed, suppose that $(s_{o}^{(1)})_{*}\nabla^{*}=\nabla$ on such a neighborhood.
Then $s_{o}^{(1)}\circ s_{o}^{(1)}$ preserves $\nabla$ on a sufficiently small neighborhood of $o$, fixes $o$, and has differential $\id_{T_{o}M}$ there.
By uniqueness of $\nabla$-geodesics, $s_{o}^{(1)}\circ s_{o}^{(1)}=\id$ around $o$.
However, the formula for $s_{o}^{(1)}$ gives
\[
\bigl(s_{o}^{(1)}\circ s_{o}^{(1)}\bigr)(x,1)
=
\left(
\frac{x}{1-2x^{2}/\lambda^{2}},
\frac{\sqrt{1-4x^{2}/\lambda^{2}}}
{1-2x^{2}/\lambda^{2}}
\right)
\ne(x,1)
\]
for every sufficiently small nonzero $x$.
This contradiction proves the claim.

\end{document}